\documentclass{amsart} 

\usepackage{amssymb, amsmath}
\usepackage{amsfonts}
\usepackage{blindtext}
\usepackage{color}
\usepackage{comment}
\usepackage{epsfig}
\usepackage{float}
\usepackage{graphicx}
\usepackage{array,makecell}
\usepackage{url}
\newtheorem{theorem}{Theorem}[section]
\newtheorem{lemma}[theorem]{Lemma}
\newtheorem{proposition}[theorem]{Proposition}
\newtheorem{corollary}[theorem]{Corollary}

\theoremstyle{definition}

\theoremstyle{remark}
\newtheorem{remark}[theorem]{Remark}

\newcommand{\cB}{\mathcal B}

\newcommand{\cF}{\mathcal F}

\newcommand{\R}{\mathbb R}

\newcommand{\dist}{\text{\rm dist}}
\newcommand{\ol}{\overline}
\def\be{\begin{equation}}
\def\ee{\end{equation}}

\def\p{\partial}
\def\tilde{\widetilde}
\def\a{\alpha}
\def\b{\beta}
\def\g{\gamma}
\def\e{\epsilon}

\numberwithin{equation}{section} 
\numberwithin{figure}{section}
\vspace{1cm}

\def\O{\Omega}
\DeclareMathOperator{\diam}{{\rm diam}}

\DeclareMathOperator*{\esssup}{ess\,sup}

\newcommand{\gain}{\text{\rm gain}}
\newcommand{\loss}{\text{\rm loss}}

\makeatletter
\newcommand{\opnorm}{\@ifstar\@opnorms\@opnorm}
\newcommand{\@opnorms}[1]{%
  \left|\mkern-1.5mu\left|\mkern-1.5mu\left|
   #1
  \right|\mkern-1.5mu\right|\mkern-1.5mu\right|
}
\newcommand{\@opnorm}[2][]{%
  \mathopen{#1|\mkern-1.5mu#1|\mkern-1.5mu#1|}
  #2
  \mathclose{#1|\mkern-1.5mu#1|\mkern-1.5mu#1|}
}
\makeatother

\allowdisplaybreaks

\begin{document}
\bibliographystyle{siam}

\title[$H^s_x$ regularity of SBE with Incoming BC]
{$H^s_x$ regularity of solutions to the stationary Boltzmann equation with the incoming boundary condition}

\author[D.~Kawagoe]{DAISUKE KAWAGOE}

\date{\today}

\begin{abstract}
We consider the stationary Boltzmann equation with the angular cutoff cross section in a bounded convex domain under the incoming boundary condition. In this article, we discuss the fractional Sobolev regularity of the solution without assuming the positivity of the Gaussian curvature on the boundary. For a boundary data sufficiently smooth and close to the standard Maxwellian, the solution has $H^{1-}_x$ regularity for hard potentials and soft potentials ($-2 \leq \gamma \leq 1$), while $H^{((4 + \gamma)/2)-}_x$ regularity is obtained for very soft potentials ($-3 < \gamma < -2$). We first show the well-posedness of the linearized problem on a weighted $L^2$ space and develop the $L^2-L^\infty$ estimate without the stochastic cycle. We next investigate $H^s_x$ regularity of the solution to the linearized problem. The idea of the celebrated velocity averaging lemma plays a key role in our analysis. We finally derive a bilinear estimate to extend the result on the linearized problem to the weakly nonlinear problem.

\end{abstract}

\maketitle

\setcounter{tocdepth}{1}
\tableofcontents

\section{Introduction} \label{sec:intro}

In this article, we consider the following boundary value problem:
\begin{equation} \label{BVP}
\begin{cases}
v \cdot \nabla_x F = Q(F, F) &\mbox{ in } \O \times \R^3,\\
F = F_0 &\mbox{ on } \Gamma^-,
\end{cases}
\end{equation}
where $\O$ is a bounded convex domain in $\R^3$ with $C^1$ boundary $\p \O$. The collision operator $Q$ reads
\[
Q(F, G) := \int_{\R^3} \int_0^{2\pi} \int_0^{\frac{\pi}{2}} [F(v') G(v_*') - F(v) G(v_*)] B(|v - v_*|, \theta)\,d\theta d\phi dv_*,
\]
where
\begin{equation} \label{collision_coordinate}
\begin{split}
v' :=& v + ((v_* - v) \cdot \omega) \omega, \quad v_*' := v_* - ((v_* - v) \cdot \omega) \omega,\\
\omega :=& \cos \theta \frac{v_* - v}{|v_* - v|} + (\sin \theta \cos \phi) e_2 + (\sin \theta \sin \phi) e_3.
\end{split}
\end{equation}
Here, $0 \leq \phi < 2\pi$, $0 < \theta < \pi/2$, and $e_2$ and $e_3$ are unit vectors such that the pair $\{ (v_* - v)/|v_* - v|, e_2, e_3 \}$ forms an orthonomal basis in $\R^3$. In this article, we consider the cross section $B$ with Grad's angular cutoff \cite{Grad}:
\begin{equation} \label{assumption_B1}
B(|v - v_*|,\theta) = |v - v_*|^{\gamma} b(\cos \theta) \sin \theta, \quad 0 < b(\cos \theta) \leq C \cos \theta
\end{equation}
for some positive constant $C$ and $-3 < \gamma \leq 1$. The incoming boundary $\Gamma^-$ is defined by
\[
\Gamma^- := \{ (z, v) \in \partial \O \times \R^3 \mid n(z) \cdot v < 0 \},
\]
where $n(z)$ denotes the outward unit normal vector at $z \in \p \O$. $F_0$ is a given function on $\Gamma^-$.

Instead of studying the boundary value problem \eqref{BVP} directly, we consider the perturbation of the solution around the standard Maxwellian $M(v) := \pi^{-3/2} e^{-|v|^2}$. Let $F = M + M^{1/2} f$ and $F_0 = M + M^{1/2} f_0$. Then, the function $f$ solves the following boundary value problem:
\begin{equation} \label{BVP_red}
\begin{cases}
v \cdot \nabla_x f = Lf + \Gamma(f, f) &\mbox{ in } \O \times \R^3,\\
f = f_0 &\mbox{ on } \Gamma^-, 
\end{cases}
\end{equation}
where
\begin{align*}
Lh :=& M^{-\frac{1}{2}} \left( Q(M, M^{\frac{1}{2}}h) + Q(M^{\frac{1}{2}}h, M) \right),\\
\Gamma(h_1, h_2) :=& M^{-\frac{1}{2}} Q(M^{\frac{1}{2}}h_1, M^{\frac{1}{2}}h_2).
\end{align*}
In what follows, we shall discuss regularity of a solution to the boundary value problem \eqref{BVP_red} and its regularity.

Regularity issues on boundary value problems of the stationary Boltzmann equation have attracted attention in recent years. For the case with the diffuse reflection boundary condition, I.~Chen, Hsia and the author \cite{CHK} showed differentiability of the solution to the linearized problem when $0 \leq \gamma \leq 1$ assuming the uniform convexity, or the positivity of the Gaussian curvature on $\p \O$. H.~Chen and Kim \cite{ChenKim} considered the nonlinear problem with $\gamma = 1$, and gave weighed $L^\infty$ estimates for the first derivatives of the solution. Based on this result, they discussed the $W^{1, p}_x$ estimate of the solution for $1 \leq p < 3$ under the same setting \cite{ChenKimGra}. For the case with the incoming boundary condition, I.~Chen \cite{RegularChen} showed local $1^-/2$ H\"older regularity of the solution to the linearized problem when $0 \leq \gamma < 1$ without the positivity of the Gaussian curvature. Recently, Wu and Wang \cite{WW} established H\"older regularity of the solution to the nonlinear problem when $-3 < \gamma < 0$ with the positive Gaussian curvature condition. Also, H.~Chen \cite{HChen} introduced a weighted $C^1$ estimate for the solution to the nonlinear problem with $\gamma = 1$ in the half space, which implies the $W^{1, p}$ regularity for $1 \leq p < 2$. I.~Chen, Hsia, Su and the author \cite{CHKS2} proved the existence of a $W^{1, p}$ solution to the nonlinear problem for $1 \leq p < 3$ and $0 \leq \gamma \leq 1$ with the positive Gaussian curvature when the domain was sufficiently small. We remark that the smallness assumption on the domain was removed in \cite{CHKFred} to obtain the same regularity result.

In this article, we focus on the fractional Sobolev, or $H^s$, regularity of solutions to the boundary value problem \eqref{BVP_red} for $-3 < \gamma \leq 1$ without the positivity of the Gaussian curvature. As obtained in \cite{CHKS1}, in the case $0 \leq \gamma \leq 1$, the solution to the linearized problem has $W^{1, p}$ regularity with $1 \leq p < 2$ for a bounded convex small domain without the curvature condition, and the upper bound $p = 2$ is not achieved in general. Our purpose is to give another characterization for the upper bound in terms of fractional Sobolev regularity in the weakly nonlinear case. We remark that $H^s_x$ regularity of the solution to the linearized problem for $0 \leq s < 1$ was discussed in \cite{CCHS} for the case $0 \leq \gamma \leq 1$ assuming the positive Gaussian curvature condition and the existence of the solution in $L^2(\O \times \R^3)$.
 
To state our main theorem, we introduce some function spaces. For $\a \geq 0$ and $\b \in \R$, we say $f \in L^2_{\a, \b}(\O \times \R^3)$ if
\[
\| f \|_{L^2_{\a, \b}(\O \times \R^3)} := \left( \int_\O \int_{\R^3} |f(x, v)|^2  e^{2\a|v|^2} (1 + |v|)^{2\b}\,dvdx \right)^{\frac{1}{2}} < \infty,
\]
and $f_0 \in L^2_\a(\Gamma^-; d\xi_-)$ if
\[
\| f_0 \|_{L^2_\a(\Gamma^-; d\xi_-)} := \left( \int_{\Gamma^-} |f_0(z, v)|^2 e^{2\a|v|^2}\,d\xi_- \right)^{\frac{1}{2}} < \infty,
\]
where $d\xi_- := |n(z) \cdot v|\,d\sigma_z dv$ and $d\sigma_z$ denotes the surface measure on $\p \O$. Also, we say $f \in L^\infty_{\a, \b}(\O \times \R^3)$ if
\[
\| f \|_{L^\infty_{\a, \b}(\O \times \R^3)} := \esssup_{(x, v) \in \O \times \R^3} |f(x, v)| e^{\a |v|^2} (1 + |v|)^\b < \infty,
\]
and $f_0 \in L^\infty_{\a, \b}(\Gamma^-)$ if
\[
\| f_0 \|_{L^\infty_{\a, \b}(\Gamma^-)} := \esssup_{(z, v) \in \Gamma^-} |f_0(z, v)| e^{\a |v|^2} (1 + |v|)^\b < \infty.
\] 
In addition, for $0 < s < 1$, let
\[
| f_0 |_{\cB^s_{\a, \b}(\Gamma^-)} := \sup_{\substack{(z_1, v), (z_2, v) \in \Gamma^- \\ z_1 \neq z_2}} \frac{|f_0(z_1, v) - f_0(z_2, v)|}{|z_1 - z_2|^s} e^{\a |v|^2} (1 + |v|)^\b
\]
and we say $f_0 \in \cB^s_{\a, \b}(\Gamma^-)$ if
\[
\| f \|_{\cB^s_{\a, \b}(\Gamma^-)} := \| f_0 \|_{L^\infty_{\a, \b}(\Gamma^-)} + | f_0 |_{\cB^s_{\a, \b}(\Gamma^-)} < \infty.
\]
For $0 < s < 1$, the Slobodeckij seminorm is defined by
\[
|u|_{H^s(\O)} := \left( \int_\O \int_\O \frac{|u(x) - u(y)|^2}{|x - y|^{3 + 2s}}\,dxdy \right)^{\frac{1}{2}}, 
\]
and we say $u \in H^s(\O)$ if
\[
\| u \|_{H^s(\O)} := \left( \| u \|_{L^2(\O)}^2 + |u|_{H^s(\O)}^2 \right)^{\frac{1}{2}} < \infty.
\]
Furthermore, we say $f \in L^2_{\a, \b}(\R^3; H^s(\O))$ if
\[
\| f \|_{L^2_{\a, \b}(\R^3; H^s(\O))} := \left( \int_{\R^3} \| f(\cdot, v) \|_{H^s(\O)}^2 e^{2 \a |v|^2} (1 + |v|)^{2\b}\,dv \right)^{\frac{1}{2}} < \infty.
\]
Finally, for $-3 < \gamma \leq 1$, let
\[
X^s_{\a, \b, \gamma} := L^\infty_{\a, \b}(\O \times \R^3) \cap L^2_{\a, \gamma/2}(\R^3; H^s(\O))
\]
equipped with the norm $\| \cdot \|_{X^s_{\a, \b, \g}}$ defined by
\[
\| f \|_{X^s_{\a, \b, \g}} := \| f \|_{L^\infty_{\a, \b}(\O \times \R^3)} + \| f \|_{L^2_{\a, \gamma/2}(\R^3; H^s(\O))}.
\]

With the above notations, the main theorem is stated as follows.
\begin{theorem} \label{main theorem}
Let $\O$ be a bounded convex domain with $C^1$ boundary, $0 \leq \a < 1/2$, $\b > (3 + \gamma)/2$, $-3 < \gamma \leq 1$ and $0 < s_1 \leq 1$. 
There exists a positive constant $\delta_0 > 0$ such that, if $\| f_0 \|_{\cB^{s_1}_{\a, \b}(\Gamma^-)} < \delta_0$, then the boundary value problem \eqref{BVP_red} has a unique solution $f \in X^s_{\a, \b, \gamma}$ with $0 < s < \min \left\{ s_1, s_\gamma \right\}$, where $s_\gamma$ is a constant defined by 
\begin{equation} \label{s_gamma}
s_\gamma := 
\begin{cases}
1, &-2 \leq \gamma \leq 1,\\
(4 + \gamma)/2, &-3 < \gamma < -2.
\end{cases}
\end{equation}
\end{theorem}

Compared with the previous work \cite{CCHS}, we remove the extra positive Gaussian curvature condition on the boundary, and extend their result to the soft potential case $-3 < \gamma < 0$. We shall not discuss $H^s_v$ regularity of the solution in this article because weighted $L^2$ and $L^\infty$ spaces are natural in our problem and we cannot define a fractional Sobolev space with these weights.

We sketch a proof of Theorem \ref{main theorem}. We consider the iteration scheme:
\begin{equation} \label{iteration_1}
\begin{cases}
v \cdot \nabla_x f_1 = Lf_1 &\mbox{ in } \O \times \R^3,\\
f_1 = f_0 &\mbox{ on } \Gamma^-
\end{cases}
\end{equation}
and
\begin{equation} \label{iteration_j}
\begin{cases}
v \cdot \nabla_x f_{i + 1} = Lf_{i + 1} + \Gamma(f_i, f_i) &\mbox{ in } \O \times \R^3,\\
f_{i+1} = f_0 &\mbox{ on } \Gamma^-.
\end{cases}
\end{equation}
Our goal is to show that the above iteration scheme creates a convergent sequence $\{ f_i \}$ in $X^s_{\a, \b, \g}$. To this end, we begin by considering the following auxiliary boundary value problem:
\begin{equation} \label{BVP_lin}
\begin{cases}
v \cdot \nabla_x f = Lf + \phi &\mbox{ in } \O \times \R^3,\\
f = f_0 &\mbox{ on } \Gamma^-.
\end{cases}
\end{equation}
Under Grad's angular cutoff assumption \eqref{assumption_B1}, the linearized collision operator $L$ can be decomposed into a multiplication operator and an integral operator $K$: 
\[
Lf = -\nu f + Kf.
\]
With the above decomposition, we reduce the boundary value problem \eqref{BVP_lin} into the following integral equation:
\begin{equation} \label{IE}
f = S_\O K f +J f_0 + S_\O \phi,
\end{equation}
where
\begin{align*}
S_\O h(x, v) :=& \int_0^{\tau_-(x, v)} e^{-\nu(|v|)t} h(x - tv, v)\,ds,\\
Jf_0(x, v) :=& e^{-\nu(|v|) \tau_-(x, v)} f_0(q(x, v), v),\\
\tau_-(x, v) :=& \inf \{ t > 0 \mid x - tv \notin \O \},\\
q(x, v) := & x - \tau_-(x, v)v.
\end{align*}
We let $k$ be the integral kernel of $K$, and show some properties and estimates of $\nu$, $k$, $K$, $S_\O$ and $J$ in Section \ref{sec:pre}. 

Our strategy for the proof consists of four steps. The first step is to establish the well-posedness of the boundary value problem \eqref{BVP_lin} in the weighted $L^2$ space.

\begin{lemma} \label{lem:existence_lin}
Let $\O$ be a bounded convex domain in $\R^3$ with $C^1$ boundary, $0 \leq \a < 1/2$ and $-3 < \gamma \leq 1$. Suppose that $J f_0 + S_\O \phi$ belongs to $L^2_{\a, \gamma/2}(\O \times \R^3)$. Then, there exists a unique solution $f$ to the boundary value problem \eqref{BVP_lin} in $L^2_{\a, \gamma/2}(\O \times \R^3)$. Moreover, there exists a positive constant $C_1$ such that
\[
\| f \|_{L^2_{\a, \gamma/2}(\O \times \R^3)} \leq C_1 \| J f_0 + S_\O \phi \|_{L^2_{\a, \gamma/2}(\O \times \R^3)}
\]
for all $f_0$ and $\phi$ with $J f_0 + S_\O \phi \in L^2_{\a, \gamma/2}(\O \times \R^3)$.
\end{lemma}

\begin{remark}
By Lemma \ref{lem:est_J} and Corollary \ref{cor:S_bound}, $J f_0 + S_\O \phi$ belong to $L^2_{\a, \gamma/2}(\O \times \R^3)$ if $f_0 \in L^2_\a(\Gamma^-; d\xi_-)$ and $\phi \in L^2_{\a, -\gamma/2}(\O \times \R^3)$. In this case, there exists a positive constant $C_1$ such that
\[
\| f \|_{L^2_{\a, \gamma/2}(\O \times \R^3)} \leq C_1 \left( \| f_0 \|_{L^2_\a(\Gamma^-; d\xi_-)} + \| \phi \|_{L^2_{\a, -\gamma/2}(\O \times \R^3)} \right)
\]
for all $f_0 \in L^2_\a(\Gamma^-; d\xi_-)$ and $\phi \in L^2_{\a, -\gamma/2}(\O \times \R^3)$, where $f$ is the solution to the boundary value problem \eqref{BVP_lin}.
\end{remark}

For Lemma \ref{lem:existence_lin}, it suffices to show that the operator $I - S_\O K$ has a bounded inverse on $L^2_{\a, \gamma/2}(\O \times \R^3)$. To this end, we shall show the compactness of the operator $S_\O K$. As we shall see in Section \ref{sec:L2}, $S_\O K$ is compact on $L^2_{\a, \gamma/2}(\O \times \R^3)$ if and only if $S_\O K_\a$ is compact on $L^2_{0, \gamma/2}(\O \times \R^3)$, where
\begin{align*}
K_\a h(x, v) :=& \int_{\R^3} k_\a(v, v^*) h(x, v^*)\,dv^*,\\
k_\a(v, v^*) :=& e^{\a |v|^2} k(v, v^*) e^{-\a |v^*|^2}. 
\end{align*}
To show the compactness of the operator $S_\O K_\a$, we investigate a smoothing effect of the operator $V_\gamma K_\a^* S_\O K_\a$, where
\begin{align*}
K^*_\a h (x, v) :=& \int_{\R^3} k_\a^*(v, v^*) h(x, v^*)\,dv^*,\\
k_\a^*(v, v^*) :=& k_\a(v^*, v) = k_{-\a}(v, v^*),\\
V_\gamma h :=& (1 + |v|)^{-\gamma} h.
\end{align*}
In particular, we shall show that operators $S_\O K_\a: L^2_{0, \gamma/2}(\O \times \R^3) = L^2_{0, \gamma/2}(\R^3; L^2(\O)) \to L^2_{0, \gamma/2}(\R^3; H^{s_{2, \gamma}}(\O))$ and $V_\gamma K_\a^*: L^2_{0, \gamma/2}(\O \times \R^3) = L^2(\O; L^2_{0, \gamma/2}(\R^3)) \to L^2(\O; H^{s_{3, \gamma}}(\R^3))$ are bounded for certain positive constants $s_{2, \gamma}$ and $s_{3, \gamma}$, where the fractional Sobolev space $H^s(\R^3)$ is defined in the same way as $H^s(\O)$. Precise definitions of these two constants will be given by \eqref{def:s2g} and \eqref{def:s3g} respectively. 

Here, we adopt the idea of the velocity averaging lemma. It is observed by Agoshkov \cite{Ag84} and Golse, Lions, Perthame and Sentis \cite{GLPS} independently that, for the $L^2$ solution $f$ to the equation $v \cdot \nabla_x f + f = h$ with $h \in L^2(\R^3 \times \R^3)$, its moment with respect to $v$ with a suitable cutoff function has the $H^{1/2}_x$ regularity. Their proofs rely on the Fourier analysis. In the same spirit, I.~Chen, Chuang, Hsia and Su \cite{CCHS} investigated the $H^s_x$ regularity of the function $SK f$ for $f \in L^2_\a(\R^3 \times \R^3)$ and $0 \leq \gamma \leq 1$ by the Fourier analysis, where $Sh$ denotes the solution to the equation $v \cdot \nabla_x f + \nu f = h$. In this article, we extend their argument to the case where $-3 < \gamma < 0$. Thanks to the convexity of the domain, we can see that $S_\O K f = S K \tilde{f}|_{\O \times \R^3}$, where $\tilde{f}$ is the zero extension of $f \in L^2_\a(\O \times \R^3)$. By this identity, we can obtain the regularity of $S_\O K f$.

By a suitable truncation and the Rellich theorem for the fractional Sobolev spaces, we show the compactness of the operator $V_\gamma K_\a^* S_\O K_\a$ on $L^2_{0, \gamma/2}(\O \times \R^3)$, which implies that of $S_\O K$ on $L^2_{\a, \gamma/2}(\O \times \R^3)$. Lemma \ref{lem:existence_lin} follows from the injectivity of the operator $I - S_\O K$ by the Fredholm alternative theorem.

The second step is to derive the $L^2-L^\infty$ estimate for the solution to the boundary value problem \eqref{BVP_lin}.

\begin{lemma} \label{lem:L2-Linfty}
Let $\O$ be a bounded convex domain with $C^1$ boundary, $0 \leq \a < 1/2$, $\b > (3 + \gamma)/2$ and $-3 < \gamma \leq 1$. Suppose that $f_0 \in L^\infty_{\a, \b}(\Gamma^-)$ and $\phi \in L^\infty_{\a, \b - \gamma}(\O \times \R^3)$. Then, the solution $f$ to the boundary value problem \eqref{BVP_lin} belongs to $L^\infty_{\a, \b}(\O \times \R^3)$. Moreover, there exists a positive constant $C_2$ such that
\[
\| f \|_{L^\infty_{\a, \b}(\O \times \R^3)} \leq C_2 \left( \| f_0 \|_{L^\infty_{\a, \b}(\Gamma^-)} + \| \phi \|_{L^\infty_{\a, \b-\gamma}(\O \times \R^3)} \right)
\]
for all $f_0 \in L^\infty_{\a, \b}(\Gamma^-)$ and $\phi \in L^\infty_{\a, \b - \gamma}(\O \times \R^3)$.
\end{lemma}

To prove Lemma \ref{lem:L2-Linfty}, we shall investigate some regularization effects of $S_\O K$ and $K$ in terms of $L^p_x$ and $L^p_v$ regularity. We emphasize that our approach is different from the stochastic cycle, which is commonly used to establish the $L^2-L^\infty$ estimate for the solution with the diffuse reflection boundary condition \cite{DHWZ2019, GuoKim}.

The third step is to investigate regularity of the solution to the boundary value problem \eqref{BVP_lin} in fractional Sobolev spaces. For this purpose, we introduce the following function space: For $\a \geq 0$, $\b \in \R$, $-3 < \gamma \leq 1$ and $0 < s < 1$, let
\[
Y^s_{\a, \b, \gamma} := L^\infty_{\a, \b - \gamma}(\O \times \R^3) \cap L^2_{\a, -\gamma/2}(\R^3; H^s(\O))
\]
equipped with the norm
\[
\| \phi \|_{Y^s_{\a, \b, \g}} := \| \phi \|_{L^\infty_{\a, \b - \gamma}(\O \times \R^3)} + \| \phi \|_{L^2_{\a, -\gamma/2}(\R^3; H^s(\O))}.
\]
We remark that the norm $\| \cdot \|_{Y^s_{\a, \b, \g}}$ is equivalent to $\| \nu^{-1} \cdot \|_{X^s_{\a, \b, \g}}$. Namely, there exist positive constants $C$ and $C'$ such that
\[
C \| \nu^{-1} \phi \|_{X^s_{\a, \b, \g}} \leq \| \phi \|_{Y^s_{\a, \b, \g}} \leq C' \| \nu^{-1} \phi \|_{X^s_{\a, \b, \g}}
\]
for all $\phi \in Y^s_{\a, \b, \g}$.

\begin{lemma} \label{lem:regularity}
Let $\O$ be a bounded convex domain with $C^1$ boundary, $0 \leq \a < 1/2$, $\b > (3 + \gamma)/2$, $-3 < \gamma \leq 1$, $0 < s_1 \leq 1$ and $0 < s_2 < 1$. Suppose that $f_0 \in \cB^{s_1}_{\a, \b}(\Gamma^-)$ and $\phi \in Y^{s_2}_{\a, \b, \gamma}$. Then, the solution $f$ to the boundary value problem \eqref{BVP_lin} belongs to $X^s_{\a, \b, \gamma}$ with $0 < s < \min \left\{ s_1, s_2, s_\gamma \right\}$, where $s_\gamma$ is a constant defined by \eqref{s_gamma}. Moreover, there exists a positive constant $C_3$ such that
\[
\| f \|_{X^s_{\a, \b, \g}} \leq C_3 \left( \| f_0 \|_{\cB^{s_1}_{\a, \b}(\Gamma^-)} + \| \phi \|_{Y^{s_2}_{\a, \b, \g}} \right)
\]
for all $f_0 \in \cB^{s_1}_{\a, \b}(\Gamma^-)$ and $\phi \in Y^{s_2}_{\a, \b, \gamma}$.
\end{lemma}

For Lemma \ref{lem:regularity}, we investigate regularity of solutions to the integral equation \eqref{IE}. Under the assumption of Lemma \ref{lem:regularity}, it is shown that $J f_0 \in L^2_{\a, \gamma/2}(\R^3; H^s(\O))$ for $0 < s < s_1$ and $S_\O \phi \in L^2_{\a, \gamma/2}(\R^3; H^{s_2}_x(\O))$. The boundedness of the operator $S_\O K: L^2_{\a, \gamma/2}(\R^3; L^2(\O)) \to L^2_{\a, \gamma/2}(\R^3; H^{s_{2, \gamma}}(\O))$ implies that $f \in L^2_{\a, \gamma/2}(\R^3; H^s(\O))$ for $0 < s < \min \{ s_1, s_2, s_{2, \gamma} \}$. Lemma \ref{lem:L2-Linfty} allows the zero extension $\tilde{f} \in L^\infty_{\a, \b}(\R^3 \times \R^3) \cap L^2_{\a, \gamma/2}(\R^3; H^s(\R^3))$ for $f \in X^s_{\a, \b, \gamma}$ with $0 \leq s < 1/2$ (see Lemma \ref{lem:smoothing_x2}). The bootstrap argument in \eqref{IE} implies the conclusion of Lemma \ref{lem:regularity}.

The final step is to obtain the bilinear estimate.

\begin{lemma} \label{lem:bilinear}
Let $\O$ be a bounded convex domain with $C^1$ boundary, $0 \leq \a < 1/2$, $\b > (3 + \gamma)/2$, $-3 < \gamma \leq 1$ and $0 < s < 1$. For $h_1, h_2 \in X^s_{\a, \b, \g}$, we have $\Gamma(h_1, h_2) \in Y^s_{\a, \b, \g}$. Moreover, there exists a positive constant $C_4$ such that
\[
\| \Gamma(h_1, h_2) \|_{Y^s_{\a, \b, \g}} \leq C_4 \| h_1 \|_{X^s_{\a, \b, \g}} \| h_2 \|_{X^s_{\a, \b, \g}}
\]
for all $h_1, h_2 \in X^s_{\a, \b, \g}$.
\end{lemma}

By Lemma \ref{lem:regularity} with $s_2 = \min \{s_1, s_\gamma \}$ and Lemma \ref{lem:bilinear}, the iteration scheme \eqref{iteration_1}-\eqref{iteration_j} generates the sequence $\{ f_i \}$ in $X^s_{\a, \b, \gamma}$. Also, we have $\| f_1 \|_{X^s_{\a, \b, \g}} \leq C \| f_0 \|_{\cB^s_{\a, \b}(\Gamma^-)}$ and
\[
\| f_{i + 1} \|_{X^s_{\a, \b, \gamma}} \leq C \| f_i \|_{X^s_{\a, \b, \gamma}}^2 + C \| f_0 \|_{\cB^s_{\a, \b}(\Gamma^-)}
\]
for some positive constant $C$ and $ i \geq 1$. By these estimates, we show that $\| f_i \|_{X^s_{\a, \b, \g}} \leq C \delta_0$ for all $i \geq 1$ if $\| f_0 \|_{\cB^s_{\a, \b}(\Gamma^-)} \leq \delta_0$ and $\delta_0$ is sufficiently small. Furthermore, taking the difference for $i \geq 2$ gives us 
\[
f_{i + 1} - f_i = S_\O K(f_{i + 1} - f_i) + S_\O (\Gamma(f_i - f_{i-1}, f_i) + \Gamma(f_{i-1}, f_i - f_{i-1})),
\]
which implies that
\begin{align*}
&\| f_{i + 1} - f_i \|_{X^s_{\a, \b, \g}}\\ 
\leq& C \left( \| f_i - f_{i-1} \|_{X^s_{\a, \b, \g}} \| f_i \|_{X^s_{\a, \b, \g}} + \| f_{i-1} \|_{X^s_{\a, \b, \g}} \| f_i - f_{i-1} \|_{X^s_{\a, \b, \g}} \right)\\
\leq& C \delta_0 \| f_i - f_{i-1} \|_{X^s_{\a, \b, \g}}
\end{align*}
for some positive constant $C$. Thus, by taking $\delta_0$ smaller if necessary, we obtain
\[
\| f_{i + 1} - f_i \|_{X^s_{\a, \b, \g}} \leq \frac{1}{2} \| f_i - f_{i-1} \|_{X^s_{\a, \b, \g}}
\]
for all $i \geq 2$, which means that $\{ f_i \}$ is a Cauchy sequence in $X^s_{\a, \b, \g}$. Hence, the iteration scheme converge to a solution $f$ to the boundary value problem \eqref{BVP_red}. Therefore, Theorem \ref{main theorem} is proved. We mention that uniqueness of the solution also follows from the above estimate.

The organization of the rest part of this article is the following. In Section \ref{sec:pre}, we introduce some estimates which we shall employ in our analysis. Section \ref{sec:VAL} is devoted to developing  the velocity averaging lemma, which is a key tool in our analysis. We remark that it is an extension of one obtained in \cite{CCHS}. By the velocity averaging lemma, we shall give a proof of Lemma \ref{lem:existence_lin} in Section \ref{sec:L2}. In section \ref{sec:L2_Linfty}, we establish the $L^2-L^\infty$ estimate without the stochastic cycle. Thanks to the $L^\infty$ estimate, we can ensure that the zero extension of the solution is an $H^s_x$ extension for $0 \leq s < 1/2$. Employing the velocity averaging lemma again, we reach at the $H^s_x$ regularity of the solution for $s > 1/2$. A detailed argument will be given in Section \ref{sec:Hsx}. In Section \ref{sec:non}, we shall give the bilinear estimate for the nonlinear term $\Gamma$ to conclude that Theorem \ref{main theorem} holds. We provide a proof of estimates for the integral kernel $k$ in the appendix.

\section{Preliminaries} \label{sec:pre}

In this section, we introduce some estimates which we shall employ in our analysis.

\subsection{Estimates for $\nu$ and $k$} \label{subsec:v_and_k}

As was mentioned in the introduction, under the assumption \eqref{assumption_B1}, the operator $L$ can be decomposed into the multiplication operator $\nu$ and the integral operator $K$:
\[
Lh = -\nu h + Kh
\]
with
\[
Kh(x, v) := \int_{\R^3} k(v, v^*) h(x, v^*)\,dv^*.
\]
Their explicit formulae read 
\begin{align*}
\nu(|v|) =& B_0 \int_{\R^3} |v - v_*|^\gamma e^{- |v_*|^2}\,dv_*,\\
k(v, v^*) =& k_1(v, v^*) - k_2(v, v^*),\\
k_1(v, v^*) =& B_0 |v - v^*|^\gamma e^{-\frac{1}{2}(|v|^2 + |v^*|^2)},\\
k_2(v, v^*) =& \frac{1}{\pi^{\frac{3}{2}}} \frac{1}{|v - v^*|} e^{-\frac{1}{4} |v - v^*|^2 -|V_1(v, v^*)|^2} \int_{W_{v - v^*}} e^{- |w + V_2(v, v^*)|^2}\\
&\times \frac{\left( |v - v^*|^2 + |w|^2 \right)^{\frac{\gamma}{2}}}{|w| |v - v^*|} \left( b ( \cos \theta ) |w| + b ( \sin \theta ) |v - v^*| \right) \,dw,
\end{align*}
where 
\begin{align*}
B_0 :=& 2\pi \int_0^{\frac{\pi}{2}} b(\cos \theta) \sin \theta\,d\theta < \infty,\\
\cos \theta :=& \frac{|v - v^*|}{(|v - v^*|^2 + |w|^2)^{\frac{1}{2}}}, \quad \sin \theta := \frac{|w|}{(|v - v^*|^2 + |w|^2)^{\frac{1}{2}}},\\
V_1(v, v^*) :=& \frac{1}{2} \frac{|v|^2 - |v^*|^2}{|v - v^*|^2} (v - v^*),\\
V_2(v, v^*) :=& \frac{\{ (v - v^*) \cdot v \} v^* - \{ (v - v^*) \cdot v^* \} v}{|v - v^*|^2} = \frac{(v - v^*) \times (v^* \times v)}{|v - v^*|^2},\\
W_{v - v^*} :=& \{w \in \R^3 \mid w \perp v - v^*\}.
\end{align*}
We note that 
\[
|V_1(v, v^*)|^2 + |V_2(v, v^*)|^2 = \frac{1}{4} |v + v^*|^2, \quad V_2(v, v^*) \cdot (v - v^*) = 0.
\]
Furthermore, the first identity indicates that
\begin{equation} \label{identity_V2}
\frac{1}{4} |v - v^*|^2 + |V_1(v, v^*)|^2 + |V_2(v, v^*)|^2 = \frac{1}{2} \left( |v|^2 + |v^*|^2 \right).
\end{equation}
For their derivations, see \cite{Glassey, Grad}. 

We first introduce an estimate for the collision frequency $\nu$.

\begin{proposition} \label{prop:est_nu}
Let $-3 < \gamma \leq 1$. There exist positive constants $\nu_0$ and $\nu_1$ such that
\[
\nu_{0}(1+|v|)^{\gamma} \leq \nu(|v|) \leq \nu_{1} (1+|v|)^{\gamma}
\]
for all $v \in \R^3$.
\end{proposition}
For a proof of Proposition \ref{prop:est_nu}, see \cite{DHWY2017}.

We next introduce some properties and estimates for the integral kernel $k$. Here and in what follows, $C$ denotes the general constant and its suffixes imply that it depends on them.

\begin{proposition} \label{prop:est_k}
Let $- 3 < \gamma \leq 1$ and $0 < \delta < 1$. We have
\begin{equation} \label{est:k}
|k(v, v^*)| \leq \frac{C_{\gamma, \delta} w_\gamma(|v - v^*|)}{\left(1 + |v| + |v^*| \right)^{1 - \gamma}} E_\delta(v, v^*)
\end{equation}
and
\begin{equation} \label{est:dk}
|\nabla_v k(v, v^*)| \leq \frac{C_{\gamma, \delta} w_\gamma(|v - v^*|) (1+|v|) }{|v - v^*| \left(1 + |v| + |v^*| \right)^{1 - \gamma}} E_\delta(v, v^*)
\end{equation}
for all $v, v^* \in \R^3$, where
\[
E_\delta(v, v^*) := e^{- \frac{1 - \delta}{4} \left( |v - v^*|^2 + 4 |V_1(v, v^*)|^2 \right)} = e^{- \frac{1 - \delta}{4} \left( |v - v^*|^2 + \left( \frac{|v|^2 - |v^*|^2}{|v -v^*|} \right)^2 \right)}.
\]
and
\begin{equation} \label{def:wg}
w_\gamma(|v - v^*|) :=
\begin{cases}
\dfrac{1}{|v - v^*|}, &-1 < \gamma \leq 1,\\
\dfrac{|\log |v - v^*|| + 1}{|v - v^*|}, &\gamma = -1,\\
\dfrac{1}{|v - v^*|^{|\gamma|}}, &-3 < \gamma < -1,
\end{cases}
\end{equation}
\end{proposition}

\begin{remark}
Our estimate \eqref{est:k} is sharper than that in \cite{DHWY2017}.
\end{remark}

We leave a proof of Proposition \ref{prop:est_k} in the appendix for readers' convenience.

\subsection{Estimates for $K_\a$, $K_\a^*$ and $V_\gamma$} \label{subsec:KandV}

It is known in \cite{CHKS1} that
\begin{align*}
E_\delta(v, v^*) =& e^{a |v|^2} e^{-\a_{1, a, \delta} |v - v^*|^2} e^{-(1-\delta) \left( \frac{(v - v^*) \cdot v}{|v - v^*|} - \alpha_{2, a, \delta} |v - v^*| \right)^2} e^{-a |v^*|^2}\\
=& e^{-a |v|^2} e^{-\a_{1, a, \delta} |v - v^*|^2} e^{-(1-\delta) \left( \frac{(v - v^*) \cdot v^*}{|v - v^*|} + \alpha_{2, a, \delta} |v - v^*| \right)^2} e^{a |v^*|^2}
\end{align*}
for $-1/2 < a < 1/2$, $\delta > 0$ and $v, v^* \in \R^3$, where
\[
\alpha_{1, a, \delta} := \frac{(1 - \delta + 2a)(1 - \delta - 2a)}{4(1 -\delta)},\quad \alpha_{2, a, \delta} := \frac{1 - \delta - 2a}{2(1-\delta)}.
\]
Plugging these identities into the estimate \eqref{est:k} with $a = \a$ and $\delta = 1/2 - \a$, we have
\begin{align}
|k_\a(v, v^*)| \leq& \frac{C w_\gamma(|v - v^*|)}{(1 + |v| + |v^*|)^{1 - \gamma}} e^{-\a_1|v - v^*|^2} e^{-\a_3 \left( \frac{(v - v^*) \cdot v}{|v - v^*|} - \alpha_2 |v - v^*| \right)^2}, \label{est:ka}\\
|k_\a(v, v^*)| \leq& \frac{C w_\gamma(|v - v^*|)}{(1 + |v| + |v^*|)^{1 - \gamma}} e^{-\a_1|v - v^*|^2} e^{-\a_3 \left( \frac{(v - v^*) \cdot v^*}{|v - v^*|} + \alpha_2 |v - v^*| \right)^2}, \label{est:ka2}\\
|k_\a^*(v, v^*)| \leq& \frac{C w_\gamma(|v - v^*|)}{(1 + |v| + |v^*|)^{1 - \gamma}} e^{-\a_1|v - v^*|^2} e^{-\a_3 \left( \frac{(v - v^*) \cdot v}{|v - v^*|} - \alpha_2 |v - v^*| \right)^2}, \label{est:ka*}\\
|k_\a^*(v, v^*)| \leq& \frac{C w_\gamma(|v - v^*|)}{(1 + |v| + |v^*|)^{1 - \gamma}} e^{-\a_1|v - v^*|^2} e^{-\a_3 \left( \frac{(v - v^*) \cdot v^*}{|v - v^*|} + \alpha_2 |v - v^*| \right)^2} \label{est:ka*2}
\end{align}
for some $\a_1, \a_2, \a_3 > 0$ depending only on $\a$. Instead of \eqref{est:ka}, we also use the following estimate in Section \ref{sec:VAL}:
\begin{equation} \label{est:ka3}
|k_\a(v, v^*)| \leq \frac{C w_\gamma(|v - v^*|)}{(1 + |v| + |v^*|)^{1 - \gamma}} e^{-\a_1|v - v^*|^2}. 
\end{equation}

For the sake of simplicity, we further introduce 
\[
k_{\a, \gamma}^*(v, v^*) := (1 + |v|)^{-\gamma} k_\a^*(v, v^*),
\]
which is the integral kernel of the operator $V_\gamma K_\a^*$. By \eqref{est:ka*} and \eqref{est:ka*2}, we have
\begin{align}
|k_{\a, \gamma}^*(v, v^*)| \leq& \frac{C w_\gamma(|v - v^*|)}{(1 + |v| + |v^*|)^{1 - \gamma} (1 + |v|)^\gamma} e^{-\a_1|v - v^*|^2} e^{-\a_3 \left( \frac{(v - v^*) \cdot v}{|v - v^*|} - \alpha_2 |v - v^*| \right)^2}, \label{est:vka*}\\
|k_{\a, \gamma}^*(v, v^*)| \leq& \frac{C w_\gamma(|v - v^*|)}{(1 + |v| + |v^*|)^{1 - \gamma} (1 + |v|)^\gamma} e^{-\a_1|v - v^*|^2} e^{-\a_3 \left( \frac{(v - v^*) \cdot v^*}{|v - v^*|} + \alpha_2 |v - v^*| \right)^2} \label{est:vka*2}
\end{align}
for some $\a_1, \a_2, \a_3 > 0$ depending only on $\a$, respectively.

We also introduce an estimate for $\nabla_v k_{\a, \gamma}^*$. A straightforward calculation yields
\begin{align*}
\nabla_v k_{\a, \gamma}^*(v, v^*) =& -\gamma (1 + |v|)^{-\gamma - 1} \frac{v}{|v|} k_\a^*(v, v^*) -2\a v (1 + |v|)^{-\gamma} k_\a^*(v, v^*)\\
&+ (1 + |v|)^{-\gamma} e^{-\a |v|^2} \nabla_v k(v, v^*) e^{\a |v^*|^2}.
\end{align*}
For the first two terms on the right hand side, we have
\begin{align*}
&\left| -\gamma (1 + |v|)^{-\gamma - 1} \frac{v}{|v|} k_\a^*(v, v^*) -2\a v (1 + |v|)^{-\gamma} k_\a^*(v, v^*) \right|\\ 
\leq& C (1 + |v|)^{-\gamma + 1} |k_\a^*(v, v^*)|\\
\leq& \frac{C w_\gamma(|v - v^*|) (1 + |v|)^{1 - \gamma}}{(1 + |v| + |v^*|)^{1 - \gamma}} e^{-\a_1|v - v^*|^2} e^{-\a_3 \left( \frac{(v - v^*) \cdot v}{|v - v^*|} - \alpha_2 |v - v^*| \right)^2}\\
\leq& \frac{C w_\gamma(|v - v^*|) (1 + |v|)^{1 - \gamma}}{|v - v^*| (1 + |v| + |v^*|)^{1 - \gamma}} e^{-\a_1'|v - v^*|^2} e^{-\a_3 \left( \frac{(v - v^*) \cdot v}{|v - v^*|} - \alpha_2 |v - v^*| \right)^2}
\end{align*}
for some $0 < \a_1' < \a_1$. Here, we have used the following estimate:
\[
e^{-\delta |v -v^*|^2} \leq \frac{C_\delta}{|v - v^*|}
\]
for some positive constant $\delta$. For the third term on the right hand side, we invoke the estimate \eqref{est:dk} and choose $\delta = 1/2 - \a$ to obtain
\begin{align*}
&\left| (1 + |v|)^{-\gamma} e^{-\a |v|^2} \nabla_v k(v, v^*) e^{\a |v^*|^2} \right|\\ 
\leq& \frac{C w_\gamma(|v - v^*|) (1 + |v|)^{1 - \gamma}}{|v - v^*| (1 + |v| + |v^*|)^{1 - \gamma}} e^{-\a_1|v - v^*|^2} e^{-\a_3 \left( \frac{(v - v^*) \cdot v}{|v - v^*|} - \alpha_2 |v - v^*| \right)^2}\\
\leq& \frac{C w_\gamma(|v - v^*|) (1 + |v|)^{1 - \gamma}}{|v - v^*| (1 + |v| + |v^*|)^{1 - \gamma}} e^{-\a_1'|v - v^*|^2} e^{-\a_3 \left( \frac{(v - v^*) \cdot v}{|v - v^*|} - \alpha_2 |v - v^*| \right)^2}.
\end{align*}
Therefore, we have
\begin{equation} \label{est:dvka*}
|\nabla_v k_{\a, \gamma}^*(v, v^*)| \leq \frac{C w_\gamma(|v - v^*|) (1 + |v|)^{1 - \gamma}}{|v - v^*| (1 + |v| + |v^*|)^{1 - \gamma}} e^{-\a_1'|v - v^*|^2} e^{-\a_3 \left( \frac{(v - v^*) \cdot v}{|v - v^*|} - \alpha_2 |v - v^*| \right)^2}.
\end{equation}

We introduce some decay estimates arising from the operator $V_\gamma K_\a^*$. In what follows, we use the following estimate to include the case $\gamma = -1$ into $-3 < \gamma < -1$: 
\begin{equation} \label{gamma-1}
w_{-1}(|v - v^*|) e^{-\frac{\a_1}{2} |v -v^*|^2} \leq \frac{C_{\e}}{|v - v^*|^{1+\epsilon}}
\end{equation}
for some small positive constant $\epsilon$.

Based on the above calculations, we obtain the following estimates, which are related to boundedness of the operator $V_\gamma K_{\a, \gamma}^*$.

\begin{proposition} \label{prop:est_vka*v*}
Let $0 \leq \a < 1/2$ and $-3 < \gamma \leq 1$. We have
\[
\int_{\R^3} |k_{\a, \gamma}^*(v, v^*)|\,dv^* \leq C_{a, \gamma} (1 + |v|)^{-1}
\]
for all $v \in \R^3$, and
\[
\int_{\R^3} |k_{\a, \gamma}^*(v, v^*)|(1 + |v|)^\gamma\,dv \leq C_{a, \gamma} (1 + |v^*|)^{\gamma-1}
\]
for all $v^* \in \R^3$.
\end{proposition}

\begin{proof}
Since
\[
\frac{1}{(1 + |v| + |v^*|)^{1 - \gamma} (1 + |v|)^\gamma} \leq \frac{1}{1 + |v|}
\]
for all $v, v^* \in \R^3$ in the estimate \eqref{est:vka*}, we obtain
\begin{align*}
\int_{\R^3} |k_{\a, \gamma}^*(v, v^*)|\,dv^*\leq& C_{\a, \gamma} (1 + |v|)^{- 1} \int_{\R^3} w_\gamma(|v -v^*|)  e^{-\a_1 |v - v^*|^2}\,dv^*\\
\leq& C_{\a, \gamma} (1 + |v|)^{- 1} \int_{\R^3} w_\gamma(|v -v^*|)  e^{-\a_1 |v - v^*|^2}\,dv^*
\end{align*}
for all $v \in \R^3$. By the estimate \eqref{est:vka*2}, we obtain
\begin{align*}
\int_{\R^3} |k_{\a, \gamma}^*(v, v^*)| (1 + |v|)^{\gamma}\,dv\ \leq& C_{\a, \gamma} (1 + |v^*|)^{\gamma - 1} \int_{\R^3} w_\gamma(|v -v^*|)  e^{-\a_1 |v - v^*|^2}\,dv\\
\leq& C_{\a, \gamma} (1 + |v^*|)^{\gamma-1}
\end{align*}
for all $v^* \in \R^3$. This completes the proof.
\end{proof}

For $-3 < \gamma \leq -2$, we shall introduce fractional Sobolev spaces for the $v$ variable without the weight. To show boundedness of the operator $V_\gamma K_{\a, \gamma}^*$ on such spaces, we introduce the following estimate.

\begin{proposition} \label{prop:est_vka*v2}
Let $0 \leq \a < 1/2$ and $-3 < \gamma \leq -2$. We have
\[
\int_{\R^3} |k_{\a, \gamma}^*(v, v^*)|(1 + |v|)^{-1}\,dv \leq C_{a, \gamma} (1 + |v^*|)^{\gamma}
\]
for all $v^* \in \R^3$.
\end{proposition}

\begin{proof}
We notice that $-1 - \gamma \geq 0$ when $\gamma \leq -2$. Thus, by \eqref{est:vka*2}, we have
\begin{align*}
&\int_{\R^3} |k_{\a, \gamma}^*(v, v^*)| (1 + |v|)^{-1}\,dv\\ 
\leq& C_{\a, \gamma} \int_{\R^3} \frac{w_\gamma(|v -v^*|)}{(1 + |v| + |v^*|)^2 (1 + |v| + |v^*|)^{-1 - \gamma} (1 + |v|)^{1 + \gamma}}\\
&\times e^{-\a_1 |v - v^*|^2} e^{-\a_3 \left( \frac{(v - v^*) \cdot v^*}{|v - v^*|} + \alpha_2 |v - v^*| \right)^2} \,dv\\
\leq& \frac{C_{\a, \gamma}}{(1 + |v^*|)^2} \int_{\R^3} w_\gamma(|v -v^*|) e^{-\a_1 |v - v^*|^2} e^{-\a_3 \left( \frac{(v - v^*) \cdot v^*}{|v - v^*|} + \alpha_2 |v - v^*| \right)^2}\,dv\\
\leq& \frac{C_{\a, \gamma}}{(1 + |v^*|)^3}\\
\leq& C_{\a, \gamma} (1 + |v^*|)^\gamma
\end{align*}
for $-3 < \gamma \leq -2$ and $v^* \in \R^3$. Here, we have used the estimate:
\begin{equation} \label{est:Caf2}
\int_{\R^3} \frac{1}{|v - v^*|^\mu} e^{-\a_1 |v - v^*|^2} e^{-\a_3 (\frac{(v -v^*) \cdot v^*}{|v - v^*|} + \a_2 |v -v^*|)^2}\,dv \leq \frac{C_{\a, \mu}}{1 + |v^*|}
\end{equation}
for $\mu < 3$ and $v^* \in \R^3$ \cite{Caf 1}. This completes the proof.
\end{proof}

Corresponding to Proposition \ref{prop:est_vka*v*}, we have the following estimates for $\nabla_v k_{\a, \gamma}^*$. We remark that $\nabla_v k_{\a, \gamma}^*$ is not integrable with respect to $v^*$ (or $v$) on $\R^3$ for $-3 < \gamma \leq -2$ due to its strong singularity at $v^* = v$.

\begin{proposition} \label{prop:est_dvka*v*}
Let $0 \leq \a < 1/2$ and $-2 < \gamma \leq 1$. We have
\begin{equation} \label{ineq:est_dvka*1}
\int_{\R^3} |\nabla_v k_{\a, \gamma}^*(v, v^*)|\,dv^* \leq C_{a, \gamma} (1 + |v|)^{- 1}
\end{equation}
for all $v \in \R^3$ and
\begin{equation} \label{ineq:est_dvka*2}
\int_{\R^3} |\nabla_v k_{\a, \gamma}^*(v, v^*)|(1 + |v|)^{\gamma-1} \,dv \leq C_{a, \gamma} (1 + |v^*|)^{\gamma - 1}
\end{equation}
for all $v^* \in \R^3$.
\end{proposition}

\begin{proof}
Since
\[
\frac{(1 + |v|)^{1 - \gamma}}{(1 + |v| + |v^*|)^{1- \gamma}} \leq 1
\]
for all $v, v^* \in \R^3$, we have
\begin{align*}
\int_{\R^3} |\nabla_v k_{\a, \gamma}^*(v, v^*)|\,dv^* \leq C \int_{\R^3} \frac{w_\gamma(|v - v^*|)}{|v - v^*|} e^{-\a_1'|v - v^*|^2} e^{-\a_3 \left( \frac{(v - v^*) \cdot v}{|v - v^*|} - \alpha_2 |v - v^*| \right)^2}\,dv^*.
\end{align*}
The estimate \eqref{ineq:est_dvka*1} follows from the estimate:
\begin{equation} \label{est:Caf}
\int_{\R^3} \frac{1}{|v - v^*|^\mu} e^{-\a_1 |v - v^*|^2} e^{-\a_3 (\frac{(v -v^*) \cdot v}{|v - v^*|} - \a_2 |v -v^*|)^2}\,dv^* \leq \frac{C_{\a, \mu}}{1 + |v|}
\end{equation}
for $\mu < 3$ and $v \in \R^3$ \cite{Caf 1}. 

For the estimate \eqref{ineq:est_dvka*2}, we have
\begin{align*}
\int_{\R^3} |\nabla_v k_{\a, \gamma}^*(v, v^*)| (1 + |v|)^{\gamma - 1}\,dv \leq& \frac{C}{(1 + |v^*|)^{1 - \gamma}} \int_{\R^3} \frac{w_\gamma(|v - v^*|)}{|v - v^*|} e^{-\a_1'|v - v^*|^2} \,dv\\
=& C (1 + |v^*|)^{\gamma - 1}
\end{align*}
for all $v^* \in \R^3$. This completes the proof.
\end{proof}

For the case $-3 < \gamma \leq -2$, we give a H\"older estimate for $k_{\a, \gamma}^*$ following \cite{WW}.

\begin{proposition} \label{prop:Holder_vka*v*}
Let $0 \leq \a < 1/2$ and $-3 < \gamma \leq -2$. We have
\begin{align*}
&\int_{\R^3} |k_{\a, \gamma}^*(u, v^*) - k_{\a, \gamma}^*(v, v^*)|\,dv^*\\
\leq& C \max \left\{ \frac{1}{1 + |u|}, \frac{1}{1 + |v|} \right\} \times
\begin{cases} 
|u - v|^{3 + \gamma}, &-3 < \gamma < -2,\\
|u - v| (|\log|u - v|| + 1), &\gamma = -2
\end{cases}
\end{align*}
for all $u, v \in \R^3$ with $|u - v| \leq 1$.
\end{proposition}

\begin{proof}
Fix $u, v \in \R^3$ with $|u - v| \leq 1$. By the fundamental theorem of calculus, we have
\begin{align*}
|k_{\a, \gamma}^*(u, v^*) - k_{\a, \gamma}^*(v, v^*)| =& \left| \int_0^1 \frac{d}{dt} k_{\a, \gamma}^*(u(t), v^*)\,dt \right|\\
\leq& |u - v| \int_0^1 |\nabla_v k_{\a, \gamma}^*(u(t), v^*)|\,dt
\end{align*}
for all $v^* \in \R^3 \setminus \{ u(t) \mid 0 \leq t \leq 1\}$, where $u(t) := tu + (1 - t)v$, $0 \leq t \leq 1$. Thus, we employ the estimate \eqref{est:dvka*} to obtain
\begin{equation} \label{est:Holder_vka*}
\begin{aligned}
&|k_{\a, \gamma}^*(u, v^*) - k_{\a, \gamma}^*(v, v^*)|\\
\leq& C |u - v| \int_0^1 \frac{w_\gamma(|u(t) - v^*|)}{|u(t) - v^*|} e^{-\a_1' |u(t) - v^*|^2} e^{-\a_3 \left( \frac{(u(t) - v^*) \cdot u(t)}{|u(t) - v^*|} - \alpha_2 |u(t) - v^*| \right)^2}\,dt.
\end{aligned}
\end{equation}

We decompose the integration as follows:
\begin{align*}
&\int_{\R^3} |k_{\a, \gamma}^*(u, v^*) - k_{\a, \gamma}^*(v, v^*)|\,dv^*\\
\leq& \int_{\{|u - v^*| \leq 2 |u - v| \}} |k_{\a, \gamma}^*(u, v^*)|\,dv^* + \int_{\{|u - v^*| \leq 2 |u - v| \}} |k_{\a, \gamma}^*(v, v^*)|\,dv^*\\
&+ \int_{\{|u - v^*| > 2 |u - v| \}} |k_{\a, \gamma}^*(u, v^*) - k_{\a, \gamma}^*(v, v^*)|\,dv^*.
\end{align*}
For the first term on the right hand side, by the estimate \eqref{est:vka*} and by introducing the spherical coordinates $v^* = u + \rho \omega$ with $0 < \rho \leq 2|u - v|$ and $\omega \in S^2$, we have
\begin{align*}
&\int_{\{|u - v^*| \leq 2 |u - v| \}} |k_{\a, \gamma}^*(u, v^*)|\,dv^*\\ 
\leq& \frac{C}{1 + |u|} \int_{\{|u - v^*| \leq 2 |u - v| \}} w_\gamma(|u - v^*|) e^{-\a_1 |u - v^*|^2}\,dv^*\\
\leq& \frac{C}{1 + |u|} \int_0^{2|u - v|} \rho^{\gamma + 2}\,d\rho\\
=& \frac{C}{1 + |u|} |u - v|^{\gamma + 3}.
\end{align*}
For the second term on the right hand side, we notice that
\[
|v - v^*| \leq |u - v^*| + |u - v| \leq 3 |u - v|
\]
if $|u - v^*| \leq 2 |u - v|$. Thus, we follow the same argument as for the first term to obtain
\begin{align*}
\int_{\{|u - v^*| \leq 2 |u - v| \}} |k_{\a, \gamma}^*(v, v^*)|\,dv^* \leq& \int_{\{|v - v^*| \leq 3 |u - v| \}} |k_{\a, \gamma}^*(v, v^*)|\,dv^*\\
\leq& \frac{C}{1 + |v|} |u - v|^{\gamma + 3}.
\end{align*}

For the third term on the right hand side, we also notice that
\[
|u(t) - v^*| \geq |u - v^*| - (1 - t)|u - v| > |u - v|
\]
if $|u - v^*| > 2 |u - v|$. Thus, by the estimate \eqref{est:Holder_vka*}, we have
\begin{align*}
&\int_{\{|u - v^*| > 2 |u - v| \}} |k_{\a, \gamma}^*(u, v^*) - k_{\a, \gamma}^*(v, v^*)|\,dv^*\\
\leq& C |u - v| \int_{\{|u - v^*| > 2|u - v| \}} \int_0^1 \frac{w_\gamma(|u(t) - v^*|)}{|u(t) - v^*|} e^{-\a_1' |u(t) - v^*|^2}\\
&\times e^{-\a_3 \left( \frac{(u(t) - v^*) \cdot u(t)}{|u(t) - v^*|} - \alpha_2 |u(t) - v^*| \right)^2}\,dt \,dv^*\\
\leq& C |u - v| \int_0^1 \int_{\{|u(t) - v^*| > |u - v| \}} \frac{w_\gamma(|u(t) - v^*|)}{|u(t) - v^*|} e^{-\a_1' |u(t) - v^*|^2}\\
&\times e^{-\a_3 \left( \frac{(u(t) - v^*) \cdot u(t)}{|u(t) - v^*|} - \alpha_2 |u(t) - v^*| \right)^2}\,dv^* \,dt.
\end{align*}
For fixed $t$, we introduce the spherical coordinates $v^* = u(t) + \rho \omega$ with $\rho > |u - v|$ and $\omega \in S^2$ to obtain
\begin{align*}
&\int_{\{|u(t) - v^*| > |u - v| \}} \frac{w_\gamma(|u(t) - v^*|)}{|u(t) - v^*|} e^{-\a_1' |u(t) - v^*|^2} e^{-\a_3 \left( \frac{(u(t) - v^*) \cdot u(t)}{|u(t) - v^*|} - \alpha_2 |u(t) - v^*| \right)^2}\,dv^*\\
=& \int_{|u - v|}^\infty \rho^{\gamma+1} e^{- \a_1' \rho^2} \left( \int_{S^2} e^{-\a_3 \left( \omega \cdot u(t) - \alpha_2 \rho \right)^2}\,d\omega \right)\,d\rho.
\end{align*}
We further change the variable of integration $\omega \cdot u(t) = r|u(t)|$ with $-1 \leq r \leq 1$ to obtain
\[
\int_{S^2} e^{-\a_3 \left( \omega \cdot u(t) - \alpha_2 \rho \right)^2}\,d\omega  = \int_{-1}^1 e^{-\a_3 \left( r|u(t)| - \alpha_2 \rho \right)^2}\,dr.
\]
On the one hand, we have
\[
\int_{-1}^1 e^{-\a_3 \left( r|u(t)| - \alpha_2 \rho \right)^2}\,dr \leq \int_{-1}^1\,dr = 2.
\]
On the other hand, we get
\begin{align*}
\int_{-1}^1 e^{-\a_3 \left( r|u(t)| - \alpha_2 \rho \right)^2}\,dr =& \frac{1}{|u(t)|} \int_{-|u(t)|}^{|u(t)|} e^{-\a_3 \left( s - \alpha_2 \rho \right)^2}\,ds\\ 
\leq& \frac{1}{|u(t)|} \int_{-\infty}^\infty e^{-\a_3 s^2}\,ds\\ 
\leq& \frac{C}{|u(t)|}.
\end{align*}
Thus, we obtain
\[
\int_{S^2} e^{-\a_3 \left( \omega \cdot u(t) - \alpha_2 \rho \right)^2}\,d\omega \leq \frac{C}{1 + |u(t)|},
\]
and hence we have
\begin{align*}
&\int_{\{|u(t) - v^*| > |u - v| \}} \frac{w_\gamma(|u(t) - v^*|)}{|u(t) - v^*|} e^{-\a_1' |u(t) - v^*|^2} e^{-\a_3 \left( \frac{(u(t) - v^*) \cdot u(t)}{|u(t) - v^*|} - \alpha_2 |u(t) - v^*| \right)^2}\,dv^*\\
\leq& \frac{C}{1 + |u(t)|} \int_{|u - v|}^\infty \rho^{\gamma+1} e^{- \a_1' \rho^2}\,d\rho\\
\leq& \frac{C}{1 + |u(t)|} (1 + |u - v|^{\gamma + 2})\\
\leq& \frac{C}{1 + |u(t)|} |u - v|^{\gamma + 2}
\end{align*}
for $-3 < \gamma < -2$ while
\begin{align*}
&\int_{\{|u(t) - v^*| > |u - v| \}} \frac{w_\gamma(|u(t) - v^*|)}{|u(t) - v^*|} e^{-\a_1' |u(t) - v^*|^2} e^{-\a_3 \left( \frac{(u(t) - v^*) \cdot u(t)}{|u(t) - v^*|} - \alpha_2 |u(t) - v^*| \right)^2}\,dv^*\\
\leq& \frac{C}{1 + |u(t)|} (|\log|u - v|| + 1)
\end{align*}
for $\gamma = -2$. Therefore, we reach at
\begin{align*}
&\int_{\{|u - v^*| > 2 |u - v| \}} |k_{\a, \gamma}^*(u, v^*) - k_{\a, \gamma}^*(v, v^*)|\,dv^*\\ 
\leq& C \int_0^1 \frac{1}{1 + |u(t)|}\,dt \times
\begin{cases} 
|u - v|^{3 + \gamma}, &-3 < \gamma < -2,\\
|u - v| (|\log|u - v|| + 1), &\gamma = -2.
\end{cases}
\end{align*}

At the end of the proof, we give an estimate for $|u(t)|$. When $|u| > 2$, we have
\[
|u(t)| \geq |u| - (1 - t)|u - v| \geq |u| - 1,
\]
which leads us to $|u| + 1 \leq 2(|u(t)| + 1)$, or
\[
\frac{1}{1 + |u(t)|} \leq \frac{2}{1 + |u|}.
\]
When $|v| > 2$, in the same manner, we see that
\[
|u(t)| \geq |v| - t|u - v| \geq |v| - 1
\]
and
\[
\frac{1}{1 + |u(t)|} \leq \frac{2}{1 + |v|}.
\]
When $|u| \leq 2$ and $|v| \leq 2$, we get
\[
\frac{1}{1 + |u(t)|} \leq 1 \leq 3 \max \left\{ \frac{1}{1 + |u|}, \frac{1}{1 + |v|} \right\}.
\]
Thus, we obtain
\[
\frac{1}{1 + |u(t)|} \leq 3 \max \left\{ \frac{1}{1 + |u|}, \frac{1}{1 + |v|} \right\}
\]
for all $u, v \in \R^3$ with $|u - v| \leq 1$ and $0 \leq t \leq 1$.

Summarizing the above argument, we conclude that the desired estimate holds for all $u, v \in \R^3$ with $|u - v| \leq 1$. This completes the proof.
\end{proof}

We are ready to show some estimates on the operator $V_\gamma K_\a^*$.

\begin{proposition} \label{prop:bounded_VK*}
For $0 \leq \a < 1/2$ and $-3 < \gamma \leq 1$, the operator $V_\gamma K_\a^*: L^2_{0, \gamma/2}(\R^3) \to L^2_{0, \gamma/2}(\R^3)$ is bounded.
\end{proposition}

\begin{proof}
Let $h \in L^2_{0, \gamma/2}(\R^3)$. We employ the Cauchy-Schwarz inequality and Proposition \ref{prop:est_vka*v*} to obtain
\begin{align*}
&\| V_\gamma K_\a^* h \|_{L^2_{0, \gamma/2}(\R^3)}^2\\ 
=& \int_{\R^3} \left( \int_{\R^3} k_{\a, \gamma}^*(v, v^*) h(v^*)\,dv^* \right)^2 (1 + |v|)^{\gamma}\,dv\\
\leq& \int_{\R^3} \left( \int_{\R^3} |k_{\a, \gamma}^*(v, v^*)| \,dv^* \right) \left( \int_{\R^3} |k_{\a, \gamma}^*(v, v^*)| |h(v^*)|^2\,dv^* \right) (1 + |v|)^{\gamma}\,dv\\
\leq& C \int_{\R^3} \left( \int_{\R^3} |k_{\a, \gamma}^*(v, v^*)| (1 + |v|)^{\gamma -1}\,dv \right) |h(v^*)|^2\,dv^*\\
\leq& C \int_{\R^3} |h(v^*)|^2 (1+|v^*|)^{\gamma-1}\,dv^*\\
\leq& C \int_{\R^3} |h(v^*)|^2 (1+|v^*|)^{\gamma}\,dv^*.
\end{align*}
Thus, we obtain
\[
\| V_\gamma K_\a^* h \|_{L^2_{0, \gamma/2}(\R^3)}^2 \leq C \| h \|_{L^2_{0, \gamma/2}(\R^3)}^2.
\]
This completes the proof.
\end{proof}

\begin{lemma} \label{lem:H1_VK*}
For $0 \leq \a < 1/2$ and $-2 < \gamma \leq 1$, the operator $V_\gamma K_\a^*: L^2_{0, \gamma/2}(\R^3) \to H^1_{0, \gamma/2}(\R^3)$ is bounded, where the norm $\| \cdot \|_{H^1_{\a, \b}(\R^3)}$ of the function space $H^1_{\a, \b}(\R^3)$ for $\a \geq 0$ and $\b \in \R$ is given by
\[
\| h \|_{H^1_{\a, \b}(\R^3)} := \left( \| h \|_{L^2_{\a, \b}(\R^3)}^2 + \| \nabla_v h \|_{L^2_{\a, \b}(\R^3)}^2 \right)^{\frac{1}{2}}.
\]
\end{lemma}

\begin{proof}
We have already seen that the operator $V_\gamma K_\a^*$ is bounded on $L^2_{0, \gamma/2}(\R^3)$. We can show in the same way as the proof of Proposition \ref{prop:bounded_VK*} with Proposition \ref{prop:est_dvka*v*} that 
\[
\| \nabla_v V_\gamma K_\a^* h \|_{L^2_{0, \gamma/2}(\R^3)} \leq C \| h \|_{L^2_{0, \gamma/2}(\R^3)}
\]
for all $h \in L^2_{0, \gamma/2}(\R^3)$. The lemma is proved. 
\end{proof}

\begin{lemma} \label{lem:Hs_VK*}
For $0 \leq \a < 1/2$, $-3 < \gamma \leq -2$ and $0 < s < (3 + \gamma)/2$, the operator $V_\gamma K_{\a, \gamma}^*: L^2_{0, \gamma/2}(\R^3) \to H^s(\R^3)$ is bounded.
\end{lemma}

\begin{remark}
The range of $s$ could be extended to $0 < s < 3 + \gamma$ as suggested in \cite{LWW}. However, since Lemma \ref{lem:Hs_VK*} is enough for our purpose, we shall not discuss the improvement in this article.
\end{remark}

\begin{proof}
Since $(1 + |v|)^{-2} \leq (1 + |v|)^{-1}$, Proposition \ref{prop:est_vka*v2} also implies that
\[
\int_{\R^3} |k_{\a, \gamma}^*(v, v^*)|(1 + |v|)^{-2}\,dv \leq C_{a, \gamma} (1 + |v^*|)^{\gamma}
\]
for all $v^* \in \R^3$, $0 \leq \a < 1/2$ and $-3 < \gamma \leq -2$. Thus, by Proposition \ref{prop:est_vka*v*}, it is seen that the operator $V_\gamma K_\a^*: L^2_{0, \gamma/2}(\R^3) \to L^2_{0, 0}(\R^3)$ is bounded. Hence, it suffices to show that
\[
|V_\gamma K_\a^* h|_{H^s(\R^3)} \leq C \| h \|_{L^2_{0, \gamma/2}(\R^3)}.
\]
We give a proof in the case $-3 < \gamma < -2$. The case $\gamma = -2$ can be proved in the same way.

We decompose the seminorm $|V_\gamma K_\a^* h|_{H^s(\R^3)}$ into two parts:
\begin{align*}
|V_\gamma K_\a^* h|_{H^s(\R^3)}^2 =& \int_{\R^3} \int_{\R^3} \frac{|V_\gamma K_\a^*h(u) - V_\gamma K_\a^*h(v)|^2}{|u - v|^{3 + 2s}}\,dudv\\
=& \iint_{\{u, v \mid |u - v| > 1\}} \frac{|V_\gamma K_\a^*h(u) - V_\gamma K_\a^*h(v)|^2}{|u - v|^{3 + 2s}}\,dudv\\ 
&+ \iint_{\{u, v \mid |u - v| \leq 1\}} \frac{|V_\gamma K_\a^*h(u) - V_\gamma K_\a^*h(v)|^2}{|u - v|^{3 + 2s}}\,dudv.
\end{align*}
For the former integral on the right hand side, we have
\begin{align*}
&\iint_{\{u, v \mid |u - v| > 1\}} \frac{|V_\gamma K_\a^*h(u) - V_\gamma K_\a^*h(v)|^2}{|u - v|^{3 + 2s}}\,dudv\\
\leq& C \int_{\R^3} \left( \int_{\{v \mid |u - v| > 1\}} \frac{1}{|u - v|^{3 + 2s}}\,dv \right) |V_\gamma K_\a^* h(u)|^2\,du\\ 
&+ C \int_{\R^3} \left( \int_{\{u \mid |u - v| > 1\}} \frac{1}{|u - v|^{3 + 2s}}\,du \right) |V_\gamma K_\a^* h(v)|^2\,dv\\
\leq& C \| V_\gamma K_\a^* h \|_{L^2_{0, 0}(\R^3)}\\
\leq& C \| h \|_{L^2_{0, \gamma/2}(\R^3)}. 
\end{align*}

For the latter integral on the right hand side, we apply the Cauchy-Schwarz inequality and Proposition \ref{prop:est_vka*v*} to obtain
\begin{align*}
&\iint_{\{u, v \mid |u - v| \leq 1\}} \frac{|V_\gamma K_\a^*h(u) - V_\gamma K_\a^*h(v)|^2}{|u - v|^{3 + 2s}}\,dudv\\
\leq& \int \int_{\{u, v \mid |u - v| \leq 1\}} \frac{1}{|u - v|^{3 + 2s}} \left( \int_{\R^3} |k_{\a, \gamma}^*(u, v^*) - k_{\a, \gamma}^*(v, v^*)|\,dv^* \right)\\ 
&\times \left( \int_{\R^3} |k_{\a, \gamma}^*(u, v^*) - k_{\a, \gamma}^*(v, v^*)| |h(v^*)|^2\,dv^* \right)\,dudv\\
\leq& C \int \int_{\{u, v \mid |u - v| \leq 1\}} \frac{1}{|u - v|^{2s + \gamma}} \max\left\{ \frac{1}{1+|u|}, \frac{1}{1+|v|} \right\}\\ 
&\times \left( \int_{\R^3} |k_{\a, \gamma}^*(u, v^*) - k_{\a, \gamma}^*(v, v^*)| |h(v^*)|^2\,dv^* \right)\,dudv\\
=& C \int_{\R^3} \left( \int \int_{\{u, v \mid |u - v| \leq 1\}} \frac{|k_{\a, \gamma}^*(u, v^*) - k_{\a, \gamma}^*(v, v^*)|}{|u - v|^{2s + \gamma}} \max\left\{ \frac{1}{1+|u|}, \frac{1}{1+|v|} \right\} \,dudv \right)\\ 
& \times  |h(v^*)|^2\,dv^*.
\end{align*}
Since $|u - v| \leq 1$, we have
\[
\frac{1}{1+|u|} \leq \frac{C}{1+|v|} \leq \frac{C}{1+|u|}.
\]
Also, we have $-3 < \gamma < 2s + \gamma$. Thus, Proposition \ref{prop:est_vka*v2} implies that
\begin{align*}
&\int \int_{\{u, v \mid |u - v| \leq 1\}} \frac{|k_{\a, \gamma}^*(u, v^*) - k_{\a, \gamma}^*(v, v^*)|}{|u - v|^{2s + \gamma}} \max\left\{ \frac{1}{1+|u|}, \frac{1}{1+|v|} \right\} \,dudv\\
\leq& C \int_{\R^3} \left( \int_{\{v \mid |u - v| \leq 1\}} \frac{1}{|u - v|^{2s + \gamma}}\,dv \right) |k_{\a, \gamma}^*(u, v^*)| (1 + |u|)^{-1}\,du\\
&+ C \int_{\R^3} \left( \int_{\{u \mid |u - v| \leq 1\}} \frac{1}{|u - v|^{2s + \gamma}}\,du \right) |k_{\a, \gamma}^*(v, v^*)| (1 + |v|)^{-1}\,dv\\
\leq& C(1 + |v^*|)^\gamma.
\end{align*}

Therefore, we have $|V_\gamma K_\a^* h|_{H^s(\R^3)} \leq C \| h \|_{L^2_{0, \gamma/2}(\R^3)}$. This completes the proof.
\end{proof}

We proceed to show some estimates for the integral kernel of the operator $K_\a$.

\begin{proposition} \label{prop:est_ka}
Let $0 \leq \a < 1/2$, $\b \in \R$, $-3 < \gamma \leq 1$ and $1 \leq q < \min \{3, 3/|\gamma| \}$. We have
\begin{equation} \label{ineq:est_ka1}
\int_{\R^3} |k_\a(v, v^*)|^q (1 + |v^*|)^\b \,dv^* \leq C_{\a, \gamma} (1 + |v|)^{\b + q(\gamma - 1) - 1}
\end{equation}
for all $v \in \R^3$, and
\begin{equation} \label{ineq:est_ka2}
\int_{\R^3} |k_\a(v, v^*)|^q (1 + |v|)^\b \,dv \leq C_{\a, \gamma} (1 + |v^*|)^{\b + q(\gamma - 1) - 1}
\end{equation}
for all $v^* \in \R^3$.
\end{proposition}

\begin{remark}
We set $\min \{3, 3/|\gamma| \} = 3$ when $\gamma = 0$.
\end{remark}

\begin{proof}
When $\b \geq 0$, we notice that $(1 + |v^*|)^\b \leq C_\b \{ (1 + |v|)^\b + |v - v^*|^\b \}$ for all $v, v^* \in \R^3$. By the estimate \eqref{est:ka}, we have
\begin{align*}
&|k_\a(v, v^*)|^q (1 + |v^*|)^\b\\
\leq& C (1 + |v|)^{\b + q(\gamma - 1)} w_\gamma(|v - v^*|)^q e^{-q\a_1|v - v^*|^2} e^{-q\a_3 \left( \frac{(v - v^*) \cdot v}{|v - v^*|} - \alpha_2 |v - v^*| \right)^2}\\
&+ C (1 + |v|)^{q(\gamma - 1)} w_\gamma(|v - v^*|)^q |v - v^*|^\b e^{-q\a_1|v - v^*|^2} e^{-q\a_3 \left( \frac{(v - v^*) \cdot v}{|v - v^*|} - \alpha_2 |v - v^*| \right)^2}.
\end{align*}
We notice that, for any $-3 < \gamma \leq 1$ with $\gamma \neq -1$ and $1 \leq q < \min \{3, 3/|\gamma| \}$, there exists $\mu < 3$ such that $w_\gamma(|v - v^*|)^q \leq |v - v^*|^{-\mu}$. Also, by the estimate \eqref{gamma-1}, the case $\gamma = -1$ can be treated in the same way as for $-3< \gamma < -1$. In any cases, the estimate \eqref{est:Caf} yields
\begin{align*}
\int_{\R^3} |k_\a(v, v^*)|^q (1 + |v^*|)^\b\,dv^* \leq& C (1 + |v|)^{\b + q(\gamma - 1) - 1} + C (1 + |v|)^{q(\gamma - 1) - 1}\\ 
\leq& C (1 + |v|)^{\b + q(\gamma - 1) - 1}.
\end{align*}

When $\b < 0$, we have
\begin{align*}
(1 + |v^*|)^\b =& (1 + |v|)^\b \frac{(1 + |v|)^{|\b|}}{(1 + |v^*|)^{|\b|}}\\
\leq& (1 + |v|)^\b \frac{(1 + |v^*|)^{|\b|} + |v - v^*|^{|\b|}}{(1 + |v^*|)^{|\b|}}\\
\leq& (1 + |v|)^\b \left( 1 + |v - v^*|^{|\b|} \right),
\end{align*}
and we apply the above argument to obtain the same estimate. Thus, the estimate \eqref{ineq:est_ka1} is proved.

The estimate \eqref{ineq:est_ka2} is obtained in the same way with estimates \eqref{est:ka2} and \eqref{est:Caf2}. This completes the proof.
\end{proof}

\begin{proposition} \label{prop:bounded_Ka}
For $0 \leq \a < 1/2$ and $-3 < \gamma \leq 1$, the operator $K_\a: L^2_{0, \gamma/2}(\R^3) \to L^2_{0, (4 - \gamma)/2}(\R^3)$ is bounded.
\end{proposition}

\begin{proof}
By Proposition \ref{prop:est_ka}, we have
\begin{align*}
&\| K_\a h \|_{L^2_{0, (4 - \gamma)/2}(\R^3)}^2\\ 
=& \int_{\R^3} \left( \int_{\R^3} k_\a(v, v^*) h(v^*)\,dv^* \right)^2 (1 + |v|)^{4 - \gamma}\,dv\\
\leq& \int_{\R^3} \left( \int_{\R^3} |k_\a(v, v^*)|\,dv^* \right) \left( \int_{\R^3} |k_\a(v, v^*)| |h(v^*)|^2\,dv^* \right) (1 + |v|)^{4 - \gamma}\,dv\\
\leq& C \int_{\R^3} \left( \int_{\R^3} |k_\a(v, v^*)| (1+|v|)^2\,dv \right) |h(v^*)|^2\,dv^*\\
\leq& C \| h \|_{L^2_{0, \gamma/2}(\R^3)}.
\end{align*}
This completes the proof.
\end{proof}

\subsection{Property of $J$}

We introduce the following formula.

\begin{lemma}[{\cite[Lemma 2.1]{CS}}] \label{change_integration}
For $f \in L^1(\Omega \times \mathbb{R}^3)$, we have
\[
\int_\Omega \int_{\mathbb{R}^3} |f(x, v)|\,dv dx = \int_{\Gamma^-} \int_0^{\tau_+(z, v)} |f(z + tv, v)| |n(z) \cdot v|\,dt dv d\sigma_z,
\]
where $\tau_+(z, v) := \tau_-(z, -v)$.
\end{lemma}

The following lemmas hold for the operator $J$.

\begin{lemma} \label{lem:est_J}
For $\a \geq 0$ and $-3 < \gamma \leq 1$, the operator $J: L^2_\a(\Gamma^-; d\xi_-) \to L^2_{\a, \gamma/2}(\O \times \R^3)$ is bounded.
\end{lemma}

\begin{proof}
For $f_0 \in L^2_\a(\Gamma^-; d\xi_-)$, by Lemma \ref{change_integration}, we have
\begin{align*}
\| J f_0 \|_{L^2_{\a, \gamma/2}(\O \times \R^3)}^2 =& \int_\O \int_{\R^3} \left| e^{-\nu(|v|) \tau_-(x, v)} f_0(q(x, v), v) \right|^2 e^{2 \a |v|^2} (1 + |v|)^\gamma \,dv dx\\
=& \int_{\Gamma^-} \int_0^{\tau_-(x, -v)} e^{-2 \nu(|v|) t} |f_0(z, v)|^2 |n(z) \cdot v|\,dt\\
&\times e^{2 \a |v|^2} (1 + |v|)^\gamma \,d\sigma_z dv\\
\leq& \int_{\Gamma^-} |f_0(z, v)|^2 e^{2 \a |v|^2} \frac{(1 + |v|)^\gamma}{2 \nu(|v|)} |n(z) \cdot v| \,d\sigma_z dv\\
\leq& C \| f_0 \|_{L^2_\a(\Gamma^-; d\xi_-)}^2.
\end{align*}
This completes the proof.
\end{proof}

\begin{lemma} \label{lem:est_J_infty}
For $0 \leq \a < 1/2$, $\b \in \R$, $-3 < \gamma \leq 1$, the operator $J: L^\infty_{\a, \b}(\Gamma^-) \to L^\infty_{\a, \b}(\O \times \R^3)$ is bounded. 
\end{lemma}

\begin{proof}
Due to the nonnegativity of the function $\nu(|v|)$, we have
\[
|Jf_0(x, v)| \leq |f_0(q(x, v), v)| \leq \| f_0 \|_{L^\infty_{\a, \b}(\Gamma^-)} e^{-\a |v|^2} (1 + |v|)^{-\b}
\]
for a.e.~$(x, v) \in \O \times \R^3$. This completes the proof.
\end{proof}

\subsection{Property of $S_\O$}

We begin with the following geometric estimate.

\begin{proposition} \label{prop:S_decay}
Let $\O$ be a bounded domain in $\R^3$ and $-3 < \gamma \leq 1$. We have
\[
\int_0^{\tau_-(x, v)} e^{-\nu(|v|)t}\,dt \leq \frac{C_\gamma}{1 + |v|}
\]
for all $(x, v) \in \O \times \R^3$.
\end{proposition}

\begin{proof}
since the collision frequency $\nu(|v|)$ is nonnegative, we have
\[
\int_0^{\tau_-(x, v)} e^{- \nu(|v|) t}\,dt \leq \int_0^\infty e^{- \nu(|v|) t}\,dt = \frac{1}{\nu(|v|)}.
\]
and
\[
\int_0^{\tau_-(x, v)} e^{- \nu(|v|) t}\,dt \leq \tau_-(x, v) \leq \frac{\diam(\O)}{|v|}.
\]
Thus, we obtain
\[
\int_0^{\tau_-(x, v)} e^{- \nu(|v|) t}\,dt \leq C \min \left\{ \frac{1}{\nu(|v|)}, \frac{1}{|v|} \right\}.
\]
The first argument $\nu(|v|)^{-1}$ is bounded near $v  = 0$ and the second one decays $|v|^{-1}$ as $|v| \to \infty$. Thus, the desired estimate holds.
\end{proof}

For the $L^2-L^\infty$ estimate in Section \ref{sec:L2_Linfty}, we introduce an intermediate function space. For $\a \geq 0$, $\b \in \R$ and $2 < p < \infty$, let
\[
\| f \|_{L^p_{\a, \b}(\O \times \R^3)} := \left( \int_\O \int_{\R^3} |f(x, v)|^p e^{p \a |v|^2} (1 + |v|)^{p \b}\,dvdx \right)^{\frac{1}{p}}.
\]
Then, we have the following proposition. 

\begin{proposition} \label{prop:S_bound}
Let $\O$ be a bounded convex domain with $C^1$ boundary. For $\a \geq 0$, $\b \in \R$, $-3 < \gamma \leq 1$ and $2 \leq p \leq \infty$, the operator $S_\O: L^p_{\a, \b}(\O \times \R^3) \to L^p_{\a, \b + 1}(\O \times \R^3)$ is bounded.
\end{proposition}

\begin{proof}
We first consider the case where $2 \leq p < \infty$. Let $h \in L^p_{\a, \b}(\O \times \R^3)$. The H\"older inequality and Proposition \ref{prop:S_decay} yield
\begin{align*}
&\| S_\O h \|_{L^p_{\a, \b + 1}(\O \times \R^3)}^p\\ 
=& \int_\O \int_{\R^3} |S_\O h(x, v)|^p e^{p \a |v|^2} (1 + |v|)^{p (\b + 1)}\,dvdx\\
\leq& \int_\O \int_{\R^3} \left( \int_0^{\tau_-(x, v)} e^{-\nu(|v|)t} |h(x - tv, v)|^p\,dt \right)\\
&\times \left( \int_0^{\tau_-(x, v)} e^{-\nu(|v|)t}\,dt \right)^{p - 1} e^{p \a |v|^2} (1 + |v|)^{p (\b + 1)}\,dvdx\\
\leq& C \int_\O \int_{\R^3} \left( \int_0^{\tau_-(y, -u)} e^{-\nu(|u|)s}\,ds \right) |h(y, u)|^p e^{p \a |u|^2} (1 + |u|)^{p \b + 1}\,dudy\\
\leq& C \| h \|_{L^p_{\a, \b}(\O \times \R^3)}^p.
\end{align*}
Here, we have used the identity:
\begin{equation} \label{change of variables}
\int_\O \int_{\R^3} \int_0^{\tau_-(x, v)} h_1(x, v, t)\,dt dv dx = \int_\O \int_{\R^3} \int_0^{\tau_-(y, -u)} h_1(y + su, u, s)\,ds du dy
\end{equation} 
for a nonnegative measurable function $h_1$ on $\O \times \R^3 \times [0, \infty)$, which was proved in \cite{CCHS}, and let $h_1(x, v, t) := e^{-\nu(|v|)t} h(x - tv, v)$. 

For the case where $h \in L^\infty_{\a, \b}(\O \times \R^3)$, by Proposition \ref{prop:S_decay}, we have
\begin{align*}
| S_\O h(x, v) | \leq& \int_0^{\tau_-(x, v)} e^{-\nu(|v|)t} |h(x - tv, v)|\,dt\\ 
\leq& \| h \|_{L^\infty_{\a, \b}(\O \times \R^3)} e^{-\a |v|^2} (1 + |v|)^{-\b} \int_0^{\tau_-(x, v)} e^{-\nu(|v|)t}\,dt\\
\leq& C \| h \|_{L^\infty_{\a, \b}(\O \times \R^3)} e^{-\a |v|^2} (1 + |v|)^{-\b - 1}
\end{align*}
for a.e.~$(x, v) \in \O \times \R^3$. This completes the proof.
\end{proof}

\begin{corollary} \label{cor:S_bound}
Let $\O$ be a bounded convex domain with $C^1$ boundary. For $\a \geq 0$ and $-3 < \gamma \leq 1$, the operator $S_\O: L^2_{\a, -\gamma/2}(\O \times \R^3) \to L^2_{\a, \gamma/2}(\O \times \R^3)$ is bounded.
\end{corollary}

\begin{proof}
Let $p = 2$ and $\b = -\gamma$ in Proposition \ref{prop:S_bound}. Then, the operator $S_\O: L^2_{\a, -\gamma/2}(\O \times \R^3) \to L^2_{\a, (2 - \gamma)/2}(\O \times \R^3)$ is bounded. The conclusion follows from the inequality $2 -\gamma \geq \gamma$, which is equivalent to $\gamma \leq 1$.
\end{proof}

Combining Proposition \ref{prop:bounded_Ka} with Proposition \ref{prop:S_bound}, we obtain the following boundedness.

\begin{proposition} \label{prop:bounded_SKa}
Let $\O$ be a bounded convex domain with $C^1$ boundary. For $0 \leq \a < 1$ and $-3 < \gamma \leq 1$, the operator $S_\O K_\a: L^2_{0, \gamma/2}(\O \times \R^3) \to L^2_{0, \gamma/2}(\O \times \R^3)$ is bounded.
\end{proposition}

\begin{proof}
By Proposition \ref{prop:bounded_Ka} and Proposition \ref{prop:S_bound}, we see that the operator $S_\O K_\a: L^2_{0, \gamma/2}(\O \times \R^3) \to L^2_{0, (6 - \gamma)/2}(\O \times \R^3)$ is bounded. The conclusion follows from the inclusion $L^2_{0, (6 - \gamma)/2}(\O \times \R^3) \subset L^2_{0, \gamma/2}(\O \times \R^3)$.
\end{proof}

\begin{corollary} \label{cor:S_bound2}
Let $\O$ be a bounded convex domain with $C^1$ boundary. For $\a \geq 0$, $\b \in \R$ and $-3 < \gamma \leq 1$, the operator $S_\O: L^\infty_{\a, \b - \gamma}(\O \times \R^3) \to L^\infty_{\a, \b}(\O \times \R^3)$ is bounded.
\end{corollary}

\begin{proof}
Replacing $\b$ with $\b - \gamma$ in Proposition \ref{prop:S_bound} yields that the operator $S_\O: L^\infty_{\a, \b - \gamma}(\O \times \R^3) \to L^\infty_{\a, \b - \gamma + 1}(\O \times \R^3)$ is bounded. The conclusion follows from the inequality $- \gamma + 1 \geq 0$ for $\gamma \leq 1$. 
\end{proof}

\section{Velocity averaging lemma} \label{sec:VAL}

In this section, we develop the velocity averaging lemma to describe the smoothing effect of the operator $S_\O K_\a$ in terms of the $H^s_x$ regularity, which will be used in Section \ref{sec:L2}. 

For $s \geq 0$, we say that $u \in \tilde{H}^s(\R^3)$ if
\[
\| u \|_{\tilde{H}^s(\R^3)} := \left( \int_{\R^3} (1 + |\xi|^2)^s |\cF u(\xi)|^2\,d\xi \right)^{\frac{1}{2}}, 
\]
where $\cF u$ denotes the Fourier transform of the function $u$. Also, for $\a \geq 0$, $\b \in \R$ and $s \geq 0$, we say that $f \in L^2_{\a, \b}(\R^3; \tilde{H}^s(\R^3))$ if 
\[
\| f \|_{L^2_{\a, \b}(\R^3; \tilde{H}^s(\R^3))} := \left( \int_{\R^3} \| f(\cdot, v) \|_{\tilde{H}^s(\R^3)}^2 e^{2\a |v|^2} (1 + |v|)^{2\b}\,dv \right)^{\frac{1}{2}} < \infty.
\] 
We remark that the norm $\| \cdot \|_{\tilde{H}^s(\R^3)}$ is equivalent to the Slobodeckij norm $\| \cdot \|_{H^s(\R^3)}$ for $0 < s < 1$ \cite{DPV}.

We introduce an operator $S$ defined by
\[
Sh(x, v) := \int_0^\infty e^{-\nu(|v|) t} h(x - tv, v)\,dt.
\]
We see that the function $Sh$ is the solution to the equation
\[
v \cdot \nabla_x f(x, v) + \nu(|v|) f = h
\]
with the far-field condition: $Sh(x, v) \to 0$ as $|x| \to \infty$ with $x \cdot v < 0$. Thus, we have
\[
\cF Sh(\xi, v) = \frac{1}{\nu(|v|) + i \xi \cdot v} \cF h(\xi, v)
\]
for a.e. $(\xi, v) \in \R^3 \times \R^3$.

\begin{lemma} \label{lem:smoothing_x}
Let $0 \leq \a < 1/2$ and $-3 < \gamma \leq 1$. The operator $S K_\a: L^2_{0, \gamma/2}(\R^3; \tilde{H}^s(\R^3)) \to L^2_{0, \gamma/2}(\R^3; \tilde{H}^{s + s_{2,\gamma}}(\R^3))$ is bounded for any $s \geq 0$, where
\begin{equation} \label{def:s2g}
s_{2, \gamma} := 
\begin{cases}
1/2, &-2 < \gamma \leq 1,\\
1/2 - \epsilon, &\gamma = -2,\\
(3 + \gamma)/2, &-3 < \gamma < -2
\end{cases}
\end{equation}
with $0 < \epsilon < 1/2$.
\end{lemma}

\begin{proof}
By the estimate \eqref{est:ka3}, Proposition \ref{prop:est_ka} and the Cauchy-Schwarz inequality, we have
\begin{align*}
& \int_{\R^3} (1 + |\xi|^2)^{s + s_{2, \gamma}} \| \cF S K_\a f (\xi, \cdot) \|_{L^2_{0, \gamma/2}(\R^3)}^2 d\xi\\
=& C \int_{\R^3} (1 + |\xi|^2)^{s + s_{2, \gamma}} \int_{\R^3} \frac{1}{\nu(|v|)^2 + (v \cdot \xi)^2} | K_\a \cF f (\xi, v) |^2 (1 + |v|)^{\gamma}\,dv d\xi\\
\leq& C \int_{\R^3} (1 + |\xi|^2)^{s + s_{2, \gamma}} \int_{\R^3} \frac{1}{\nu(|v|)^2 + (v \cdot \xi)^2} \left( \int_{\R^3} |k_\a(v, v^*)|\,dv^* \right)\\
&\times \left( \int_{\R^3} |k_\a(v, v^*)| |\cF f(\xi, v^*)|^2 \,dv^* \right) (1 + |v|)^{\gamma}\,dv d\xi\\
\leq& C \int_{\R^3} (1 + |\xi|^2)^{s + s_{2, \gamma}} \int_{\R^3} \left( \int_{\R^3} \frac{1}{\nu(|v|)^2 + (v \cdot \xi)^2} |k_\a(v, v^*)| (1 + |v|)^{2\gamma - 1}\,dv \right)\\
&\times |\cF f(\xi, v^*)|^2\,dv^* d\xi.
\end{align*}

In what follows, we give an estimate for the inner integral with respect to the $v$ variable. We note that, by Proposition \ref{prop:est_nu} and the estimate \eqref{est:ka3},
\begin{align*}
&\int_{\R^3} \frac{1}{\nu(|v|)^2 + (v \cdot \xi)^2} |k_\a(v, v^*)| (1 + |v|)^{2\gamma - 1}\,dv\\
\leq& C (1 + |v^*|)^{\gamma - 1} \int_{\R^3} w_\gamma(|v - v^*|) e^{-\a_1 |v - v^*|^2}\,dv\\
\leq& C (1 + |v^*|)^\gamma,
\end{align*}
which means that, for a fixed $v^*$, the inner integral is uniformly bounded with respect to $\xi$. Thus, in what follows, we discuss the bound for large $|\xi|$.

When $0 \leq \gamma \leq 1$, by Proposition \ref{prop:est_nu}, we have $\nu(|v|) \geq \nu_0$ for all $v \in \R^3$. By the estimate \eqref{est:ka2} and the inequality $(1 + |v|)^\gamma \leq (1 + |v^*|)^\gamma + |v - v^*|^\gamma$, we have
\begin{align*}
&\int_{\R^3} \frac{1}{\nu(|v|)^2 + (v \cdot \xi)^2} |k_\a(v, v^*)| (1 + |v|)^{2\gamma - 1}\,dv\\
\leq& C \int_{\R^3} \frac{(1 + |v|)^{3\gamma -2}}{(\nu_0^2 + (v \cdot \xi)^2)|v -v^*|} e^{- \a_1|v -v^*|^2}\,dv\\
\leq& C \int_{\R^3} \frac{(1 + |v|)^\gamma}{(\nu_0^2 + (v \cdot \xi)^2)|v -v^*|} e^{- \a_1|v -v^*|^2}\,dv\\
\leq& C (1 + |v^*|)^\gamma \int_{\R^3} \frac{1}{(\nu_0^2 + (v \cdot \xi)^2)|v -v^*|} e^{- \a_1|v -v^*|^2}\,dv\\
&+ C \int_{\R^3} \frac{1}{(\nu_0^2 + (v \cdot \xi)^2)|v -v^*|^{1 - \gamma}} e^{- \a_1|v -v^*|^2}\,dv.
\end{align*}
Let $e_\xi$ be the unit vector with its direction $\xi$ and decompose the $v$ variable as follows:
\begin{equation} \label{decomposition_integral}
v = v_\parallel e_\xi + v_\perp, \quad v^* = v^*_{\parallel} e_\xi + v^*_{\perp},
\end{equation}
where $v_\parallel, v^*_{\parallel} \in \R$ and $v_\perp \cdot e_\xi = v^*_{\perp} \cdot e_\xi = 0$. Identifying the hyperplane $\{ v \in \R^3 \mid v \cdot e_\xi = 0\}$ with $\R^2$, we have
\begin{align*}
&\int_{\R^3} \frac{1}{(\nu_0^2 + (v \cdot \xi)^2)|v -v^*|} e^{- \a_1|v -v^*|^2}\,dv\\
\leq& C \int_\R \frac{1}{1 + (v_\parallel)^2 |\xi|^2} \left( \int_{\R^2} \frac{1}{|v_\perp - v^*_{\perp}|} e^{- \a_1|v_\perp - v^*_{\perp}|^2}\,dv_\perp \right)\,dv_\parallel\\
\leq& \frac{C}{|\xi|} \int_\R \frac{1}{1 + t^2}\,dt\\
=& \frac{C}{|\xi|}
\end{align*}
and
\begin{align*}
&\int_{\R^3} \frac{1}{(\nu_0^2 + (v \cdot \xi)^2)|v -v^*|^{1 - \gamma}} e^{- \a_1|v -v^*|^2}\,dv\\
\leq& C \int_\R \frac{1}{1 + (v_\parallel)^2 |\xi|^2} \left( \int_{\R^2} \frac{1}{|v_\perp - v^*_{\perp}|^{1 - \gamma}} e^{- \a_1|v_\perp - v^*_{\perp}|^2}\,dv_\perp \right)\,dv_\parallel\\
\leq& \frac{C}{|\xi|}.
\end{align*}
Here, we have used the change of variable $t = |\xi| v_\parallel$. Thus, we obtain
\begin{equation} \label{est:smoothing}
\int_{\R^3} \frac{1}{\nu(|v|)^2 + (v \cdot \xi)^2} |k_\a(v, v^*)| (1 + |v|)^{2\gamma - 1}\,dv \leq \frac{C}{(1 + |\xi|^2)^{s_{2, \gamma}}} (1 + |v^*|)^\gamma.
\end{equation}

When $\gamma < 0$, by Proposition \ref{prop:est_nu} and the estimate \eqref{est:ka2} again, we notice that
\begin{align*}
&\int_{\R^3} \frac{1}{\nu(|v|)^2 + (v \cdot \xi)^2} |k_\a(v, v^*)| (1 + |v|)^{2\gamma - 1}\,dv\\
\leq& C (1 + |v^*|)^\gamma \int_{\R^3} \frac{w_\gamma(|v -v^*|) (1 + |v|)^{2\gamma - 2}}{\nu(|v|)^2 + (v \cdot \xi)^2} e^{-\a_1 |v -v^*|^2} \,dv\\
\leq& C (1 + |v^*|)^\gamma \int_{\R^3} \frac{w_\gamma(|v -v^*|)}{\nu_0^2 + (1 + |v|)^{2|\gamma|}(v \cdot \xi)^2} e^{-\a_1 |v -v^*|^2} \,dv\\
\leq& C (1 + |v^*|)^\gamma \int_{\R^3} \frac{w_\gamma(|v -v^*|)}{\nu_0^2 + (v \cdot \xi)^2} e^{-\a_1 |v -v^*|^2} \,dv.
\end{align*}
For $-1< \gamma < 0$, we have
\[
\int_{\R^3} \frac{w_\gamma(|v -v^*|)}{\nu_0^2 + (v \cdot \xi)^2} e^{-\a_1 |v -v^*|^2} \,dv = \int_{\R^3} \frac{1}{(\nu_0^2 + (v \cdot \xi)^2)|v -v^*|} e^{- \a_1|v -v^*|^2}\,dv.
\]
Thus, we may apply the decomposition \eqref{decomposition_integral} to obtain \eqref{est:smoothing} when $-1 < \gamma < 0$.

When $-2 < \gamma < -1$, we have
\[
\int_{\R^3} \frac{w_\gamma(|v -v^*|)}{\nu_0^2 + (v \cdot \xi)^2} e^{-\a_1 |v -v^*|^2} \,dv = \int_{\R^3} \frac{1}{(\nu_0^2 + (v \cdot \xi)^2)|v -v^*|^{|\gamma|}} e^{- \a_1|v -v^*|^2}\,dv.
\]
Since 
\[
\int_{\R^2}\frac{1}{|\ol{v}|^{|\gamma|}} e^{-\a_1|\ol{v}|^2}\,d\ol{v} \leq C_\a
\]
for $-2 < \gamma < -1$, we can still do the integration with the coordinates \eqref{decomposition_integral} to obtain the estimate \eqref{est:smoothing}. This argument also holds for $\gamma = -1$ since
\[
\int_{\R^2} \frac{|\log|\ol{v}||}{|\ol{v}|} e^{-\a_1|\ol{v}|^2}\,d\ol{v} \leq C.
\]

When $-3 < \gamma < -2$, we introduce the coordinates \eqref{decomposition_integral} to obtain
\begin{align*}
&\int_{\R^3} \frac{w_\gamma(|v -v^*|)}{\nu_0^2 + (v \cdot \xi)^2} e^{-\a_1 |v -v^*|^2} \,dv\\
=&\int_{\R^3} \frac{1}{(\nu_0^2 + (v \cdot \xi)^2)|v -v^*|^{|\gamma|}} e^{- \a_1 |v -v^*|^2}\,dv\\
\leq& \int_{\R} \frac{1}{\nu_0^2 + (v_\parallel)^2 |\xi|^2} \int_{\R^2} \frac{1}{(|v_\parallel - v^*_{\parallel}|^2 + |v_\perp - v^*_{\perp}|^2)^{\frac{|\gamma|}{2}}}\,dv_\perp dv_\parallel\\
=& \int_{\R} \frac{|v_\parallel - v^*_{\parallel}|^{\gamma + 2}}{\nu_0^2 + (v_\parallel)^2 |\xi|^2} \int_{\R^2} \frac{1}{(1 + |\ol{v}|^2)^{\frac{|\gamma|}{2}}}\,d\ol{v} dv_\parallel\\
\leq& C_\gamma \int_{\R} \frac{|v_\parallel - v^*_{\parallel}|^{\gamma + 2}}{\nu_0^2 + (v_\parallel)^2 |\xi|^2} \,dv_\parallel.
\end{align*}
For $|v_\parallel - v^*_{\parallel}| > |\xi|^{-1}$, we have
\begin{align*}
\int_{\{ |v_\parallel - v^*_{\parallel}| > |\xi|^{-1} \}} \frac{|v_\parallel - v^*_{\parallel}|^{\gamma + 2}}{\nu_0^2 + (v_\parallel)^2 |\xi|^2} \,dv_\parallel \leq |\xi|^{-(\gamma + 2)} \int_{\R} \frac{1}{\nu_0^2 + (v_\parallel)^2 |\xi|^2}\,dv_\parallel \leq C |\xi|^{-\gamma - 3}.
\end{align*}
On the other hand, for $|v_\parallel - v^*_{\parallel}| \leq |\xi|^{-1}$, we have
\begin{align*}
\int_{\{ |v_\parallel - v^*_{\parallel}| \leq |\xi|^{-1} \}} \frac{|v_\parallel - v^*_{\parallel}|^{\gamma + 2}}{\nu_0^2 + (v_\parallel)^2 |\xi|^2} \,dv_\parallel \leq \frac{2}{\nu_0^2} \int_{|\xi|^{-1}}^\infty r^{\gamma + 2}\,dr \leq C_\gamma |\xi|^{-\gamma - 3}.
\end{align*}
Therefore, we get the estimate \eqref{est:smoothing}.

When $\gamma = -2$, we follow the above computation to obtain
\begin{align*}
&\int_{\R^3} \frac{w_\gamma(|v -v^*|)}{\nu_0^2 + (v \cdot \xi)^2} e^{-\a_1 |v -v^*|^2} \,dv\\
\leq& \int_{\R} \frac{1}{\nu_0^2 + (v_\parallel)^2 |\xi|^2} e^{- \a_1 |v_\parallel -v^*_{\parallel}|^2}\\
&\times \int_{\R^2} \frac{1}{|v_\parallel - v^*_{\parallel}|^2 + |v_\perp - v^*_{\perp}|^2} e^{- \a_1 |v_\perp -v^*_{\perp}|^2} \,dv_\perp dv_\parallel\\
=& \int_{\R} \frac{1}{\nu_0^2 + (v_\parallel)^2 |\xi|^2} e^{- \a_1 |v_\parallel -v^*_{\parallel}|^2} \int_{\R^2} \frac{1}{1 + |\ol{v}|^2} e^{- \a_1 |v_\parallel -v^*_{\parallel}|^2 |\ol{v}|^2} \,d\ol{v} dv_\parallel.
\end{align*}
We first discuss the inner integral with respect to $\ol{v}$. Introducing the polar coordinates, we have
\[
\int_{\R^2} \frac{1}{1 + |\ol{v}|^2} e^{- \a_1|v_\parallel -v^*_{\parallel}|^2 |\ol{v}|^2} \,d\ol{v} = 2\pi \int_0^\infty \frac{r}{1 + r^2} e^{- \a_1|v_\parallel -v^*_{\parallel}|^2 r^2} \,dr.
\]
For $r > |v_\parallel - v^*_{\parallel}|^{-1}$, we have
\begin{align*}
\int_{|v_\parallel - v^*_{\parallel}|^{-1}}^\infty \frac{r}{1 + r^2} e^{- \a_1|v_\parallel -v^*_{\parallel}|^2 r^2} \,d \leq& \int_{|v_\parallel - v^*_{\parallel}|^{-1}}^\infty r^{-1} e^{- \a_1|v_\parallel -v^*_{\parallel}|^2 r^2} \,dr\\
\leq& |v_\parallel - v^*_{\parallel}| \int_{|v_\parallel - v^*_{\parallel}|^{-1}}^\infty e^{- \a_1|v_\parallel -v^*_{\parallel}|^2 r^2} \,dr\\
=& \int_1^\infty e^{- \a_1 s^2} \,ds\\
\leq& C.
\end{align*}
On the other hand, for $r \leq |v_\parallel - v^*_{\parallel}|^{-1}$, we have
\begin{align*}
\int_0^{|v_\parallel - v^*_{\parallel}|^{-1}} \frac{r}{1 + r^2} e^{- \a_1|v_\parallel -v^*_{\parallel}|^2 r^2} \,dr \leq& C \int_0^{|v_\parallel - v^*_{\parallel}|^{-1}} \frac{d}{dr} \log(1 + r^2)\,dr\\
\leq& C \log \left( 1 + |v_\parallel - v^*_{\parallel}|^{-2} \right).
\end{align*}
Thus, we have
\begin{align*}
&\int_{\R^3} \frac{1}{(\nu_0^2 + (v \cdot \xi)^2)|v -v^*|^2} e^{- \a_1|v -v^*|^2}\,dv\\
\leq& C \int_{\R} \frac{1}{\nu_0^2 + (v_\parallel)^2 |\xi|^2} e^{- \a_1|v_\parallel -v^*_{\parallel}|^2} \left\{ 1 + \log \left( 1 + |v_\parallel - v^*_{\parallel}|^{-2} \right) \right\}\,dv_\parallel.
\end{align*}
When $|v_\parallel -v^*_{\parallel}| > |\xi|^{-1}$, we have
\begin{align*}
&\int_{\{|v_\parallel -v^*_{\parallel}| > |\xi|^{-1}\}} \frac{1}{\nu_0^2 + (v_\parallel)^2 |\xi|^2} e^{- \a_1|v_\parallel -v^*_{\parallel}|^2} \left\{ 1 + \log \left( 1 + |v_\parallel - v^*_{\parallel}|^{-2} \right) \right\}\,dv_\parallel\\
\leq& \left\{ 1 + \log \left( 1 + |\xi|^2 \right) \right\} \int_{\R} \frac{1}{\nu_0^2 + (v_\parallel)^2 |\xi|^2} \,dv_\parallel\\
\leq& \frac{C}{|\xi|} \left( \log(1 + |\xi|) + 1 \right).
\end{align*}
On the other hand, when $|v_\parallel -v^*_{\parallel}| \leq |\xi|^{-1}$, we have
\begin{align*}
&\int_{\{|v_\parallel -v^*_{\parallel}| \leq |\xi|^{-1}\}} \frac{1}{\nu_0^2 + (v_\parallel)^2 |\xi|^2} e^{- \a_1|v_\parallel -v^*_{\parallel}|^2} \left\{ 1 + \log \left( 1 + |v_\parallel - v^*_{\parallel}|^{-2} \right) \right\}\,dv_\parallel\\
\leq& C \int_0^{|\xi|^{-1}} e^{- \a_1 s^2} \left\{ 1 + \log \left( 1 + s^{-2} \right) \right\}\,ds.
\end{align*}
Here, for $|\xi| > 1$, we have
\begin{align*}
\int_0^{|\xi|^{-1}} e^{- \a_1 s^2} \left\{ 1 + \log \left( 1 + s^{-2} \right) \right\}\,ds \leq& \int_0^{|\xi|^{-1}} e^{- \a_1 s^2} \left\{ 1 + \log \left( 1 + s^2 \right) - \log s^2 \right\}\,ds\\
\leq& \frac{C}{|\xi|} (\log|\xi| + 1).
\end{align*}
Therefore, we have
\[
\int_0^{|\xi|^{-1}} e^{- \a_1 s^2} \left\{ 1 + \log \left( 1 + s^{-2} \right) \right\}\,ds \leq \frac{C}{|\xi|} (\log(|\xi| + 1) + 1)
\]
and
\[
\int_{\R^3} \frac{1}{\nu(|v|)^2 + (v \cdot \xi)^2} |k_\a(v, v^*)| (1 + |v|)^{2\gamma - 1}\,dv \leq \frac{C (\log(|\xi| + 1) + 1)}{1 + |\xi|} (1 + |v^*|)^\gamma.
\]
Since $\log (1 + |\xi|) + 1 \leq C_\epsilon (1 + |\xi|^2)^{\epsilon}$ for any $\epsilon > 0$, the above estimate implies that
\begin{align*}
&\int_{\R^3} \int_{\R^3} (1 + |\xi|^2)^{s + \frac{1}{2} - \epsilon} |\cF S K_\a f(\xi, v)|^2 (1 + |v|)^\gamma \,dv d\xi\\
\leq& C_\epsilon \int_{\R^3} \int_{\R^3} (1 + |\xi|^2)^s | (\cF f)(\xi, v)|^2 (1 + |v|)^\gamma \,dv d\xi
\end{align*}
for any $0 < \epsilon < 1/2$. Thus, we obtain the estimate \eqref{est:smoothing} when $\gamma = -2$. This completes the proof.
\end{proof}

\begin{corollary} \label{cor:smoothing_x}
Let $\O$ be a bounded convex domain with $C^1$ boundary, $0 \leq \a < 1/2$ and $-3 < \gamma \leq 1$. Then, the operator $S_\O K_\a : L^2_{0, \gamma/2}(\O \times \R^3) = L^2_{0 ,\gamma/2}(\R^3; L^2(\O)) \to L^2_{0 ,\gamma/2}(\R^3; H^{s_{2, \gamma}}(\O))$ is bounded, where $s_{2, \gamma}$ is the constant in \eqref{def:s2g}.
\end{corollary}

\begin{proof}
For a function $f \in L^2_{0, \gamma/2}(\O \times \R^3)$, let $\tilde{f} \in L^2_{0, \gamma/2}(\R^3 \times \R^3)$ be its zero extension, namely,
\[
\tilde{f}(x, v) :=
\begin{cases}
f(x, v), &x \in \O,\\
0, &x \in \R^3 \setminus \O.
\end{cases}
\]

By Lemma \ref{lem:smoothing_x}, we have $SK_\a \tilde{f} \in L^2_{0 ,\gamma/2}(\R^3; H^{s_{2, \gamma}}(\R^3))$ for $f \in L^2_{0 ,\gamma/2}(\R^3; L^2(\O))$. Noticing the identity
\[
\left. S K_\a \tilde{f} \right|_{\O} = S_\O K_\a f,
\]
we reach at the conclusion.
\end{proof}

\section{$L^2$ well-posedness} \label{sec:L2}

In this section, we prove the following lemma, which implies Lemma \ref{lem:existence_lin}.

\begin{lemma} \label{lem:inverse}
Let $\O$ be a bounded convex domain with $C^1$ boundary. For $0 \leq \a < 1/2$ and $-3 < \gamma \leq 1$, the operator $I - S_\O K: L^2_{\a, \gamma/2}(\O \times \R^3) \to L^2_{\a, \gamma/2}(\O \times \R^3)$ has a bounded inverse. 
\end{lemma}

We remark that Maslova \cite{M} described the same idea for the case $0 \leq \gamma \leq 1$ without an explicit proof. 

We prove Lemma \ref{lem:inverse} by the Fredholm alternative theorem. It suffices to show the compactness of the operator $S_\O K$ and the injectivity of the operator $I - S_\O K$ on $L^2_{\a, \gamma/2}(\O \times \R^3)$. However, it is hard to show the compactness directly. Thus, we transform the operator as follows.

For $f \in L^2_{\a, \gamma/2}(\O \times \R^3)$, let 
\begin{equation} \label{eq:transform_a_to_0}
f_\a := f e^{\a |v|^2}. 
\end{equation}
Then, we have $f_\a \in L^2_{0, \gamma/2}(\O \times \R^3)$. Moreover, we have
\[
(S_\O K f)_\a = S_\O K_\a f_\a \in L^2_{0, \gamma/2}(\O \times \R^3).
\]
We notice that the transformation \eqref{eq:transform_a_to_0} is isometry in the sense that $\| f \|_{L^2_{\a, \b}(\O \times \R^3)} = \| f_\a \|_{L^2_{0, \b}(\O \times \R^3)}$ for all $\b \in \R$. Thus, the compactness of $S_\O K$ on $L^2_{\a, \gamma/2}(\O \times \R^3)$ is equivalent to that of $S_\O K_\a$ on $L^2_{0, \gamma/2}(\O \times \R^3)$. In what follows, we discuss the latter compactness.

\begin{lemma} \label{lem:compactL2}
Let $\O$ be a bounded convex domain with $C^1$ boundary, $0 \leq \a < 1/2$ and $-3 < \gamma \leq 1$. Suppose that the operator $V_\gamma K^*_\a S_\O K_\a$ is compact on $L^2_{0, \gamma/2}(\O \times \R^3)$. Then, the operator $S_\O K_\a$ is also compact on $L^2_{0, \gamma/2}(\O \times \R^3)$. 
\end{lemma}

\begin{proof}
Let $\{ g_n \}$ be a sequence in $L^2_{0, \gamma/2}(\O \times \R^3)$ which converges to $0$ weakly in $L^2_{0, \gamma/2}(\O \times \R^3)$ and let $f_n := S_\O K_\a g_n$. It suffices to show that the sequence $\{ f_n \}$ converges to $0$ strongly in $L^2_{0, \gamma/2}(\O \times \R^3)$ under the assumption that $V_\gamma K^*_\a S_\O K_\a$ is compact on $L^2_{0, \gamma/2}(\O \times \R^3)$. 

Recall that the function $f_n$ solves the following boundary value problem:
\[
\begin{cases}
v \cdot \nabla_x f_n + \nu f_n = K_\a g_n &\mbox{ in } \O \times \R^3,\\
f_n = 0 &\mbox{ on } \Gamma^-.
\end{cases}
\]
Multiplying the equation by $f_n$ and integrating it on $\O \times \R^3$, by Green's identity and the Cauchy-Schwarz inequality, we have
\[
\begin{split}
&\frac{1}{2} \int_{\Gamma^+} |f_n(z, v)|^2 n(z) \cdot v \,d\sigma_z dv + \int_{\O \times \R^3} \nu(|v|) |f_n(x, v)|^2\,dxdv\\ 
=& \int_{\O \times \R^3} (K_\a g_n(x, v)) f_n(x, v)\,dxdv\\
=& \int_{\O \times \R^3} g_n(x, v) (K^*_\a S_\O K_\a g_n(x, v))\,dxdv\\
\leq& \| g_n \|_{L^2_{0, \gamma/2}(\O \times \R^3)} \| V_\gamma K^*_\a S_\O K_\a g_n \|_{L^2_{0, \gamma/2}(\O \times \R^3)}.
\end{split}
\]
By Proposition \ref{prop:est_nu}, we see that
\[
\| f_n \|_{L^2_{0, \gamma/2}(\O \times \R^3)}^2 \leq C \| g_n \|_{L^2_{0, \gamma/2}(\O \times \R^3)} \| V_\gamma K^*_\a S_\O K_\a g_n \|_{L^2_{0, \gamma/2}(\O \times \R^3)}.
\]
Since $\{ g_n \}$ is a weakly convergent sequence, it is uniformly bounded in $L^2_{0, \gamma/2}(\O \times \R^3)$. Also, by the compactness of the operator $V_\gamma K^*_\a S_\O K_\a$ on $L^2_{0, \gamma/2}(\O \times \R^3)$, the sequence $\{ V_\gamma K^*_\a S_\O K_\a g_n \}$ converges to $0$ strongly in $L^2_{0, \gamma/2}(\O \times \R^3)$ as $n \to \infty$. Therefore, the sequence $\{ f_n \}$ converges to $0$ strongly in $L^2_{0, \gamma/2}(\O \times \R^3)$. This completes the proof.
\end{proof}

Motivated by Lemma \ref{lem:compactL2}, we prove the compactness of the operator $V_\gamma K_\a^* S_\O K_\a$.

\begin{lemma} \label{lem:compact_VKSK}
Let $\O$ be a bounded convex domain with $C^1$ boundary. For $0 \leq \a < 1/2$ and $-3 < \gamma \leq 1$, the operator $V_\gamma K^*_\a S_\O K_\a: L^2_{0, \gamma/2}(\O \times \R^3) \to L^2_{0, \gamma/2}(\O \times \R^3)$ is compact.
\end{lemma}

To prove Lemma \ref{lem:compact_VKSK}, we truncate the operator. Let $R > 0$ and let $B_R$ be the open ball in $\R^3$ with the center the origin and with radius $R$. For a measurable set $A$ in $\R^3$, let $\chi_A$ denote its characteristic function and $T_R$ denote the multiplication operator by $\chi_{B_R}$. 

We introduce the following proposition.

\begin{proposition} \label{prop:preRellich}
Let $\O$ be a bounded convex domain with $C^1$ boundary, $0 < s < 1$ and $R > 0$. We have
\[
L^2(B_R; H^s(\O)) \cap L^2(\O; H^s(B_R)) \subset H^s(\O \times B_R).
\]
\end{proposition}

\begin{proof}
We recall that $|(x, v) - (y, u)| = (|x - y|^2 + |v - u|^2)^{1/2}$ for $(x, v), (y, u) \in \O \times B_R \subset \R^6$. By the definition of the Slobodeckij seminorm for fractional Sobolev spaces in $\R^6$, we see that
\begin{align*}
| f |_{H^s(\O \times B_R)}^2 =& \int_\O \int_{B_R} \int_\O \int_{B_R} \frac{|f(x, v) - f(y, u)|^2}{(|x - y|^2 + |v - u|^2)^{\frac{1}{2}(6 + 2s)}}\,dv dx du dy\\
\leq& 2 \int_\O \int_{B_R} \int_\O \int_{B_R} \frac{|f(x, v) - f(y, v)|^2}{(|x - y|^2 + |v - u|^2)^{3+s}}\,dv dx du dy\\
&+ 2 \int_\O \int_{B_R} \int_\O \int_{B_R} \frac{|f(y, v) - f(y, u)|^2}{(|x - y|^2 + |v - u|^2)^{3+s}}\,dv dx du dy.
\end{align*}

We give an estimate for the first term on the right hand side. By changing the order of integration, we see that
\begin{align*}
&\int_\O \int_{B_R} \int_\O \int_{B_R} \frac{|f(x, v) - f(y, v)|^2}{(|x - y|^2 + |v - u|^2)^{3+s}}\,dv dx du dy\\
=& \int_{B_R} \int_\O \int_\O |f(x, v) - f(y, v)|^2 \left( \int_{B_R} \frac{1}{(|x - y|^2 + |v - u|^2)^{3+s}}\,du \right)\,dx dy dv.
\end{align*}
Introducing the spherical coordinates $u = v + |x - y| t \omega$ with $t > 0$ and $\omega \in S^2$, we have
\begin{align*}
&\int_{B_R} \frac{1}{(|x - y|^2 + |v - u|^2)^{3+s}}\,du\\
=& \frac{1}{|x - y|^{6+2s}}  \int_{S^2} \int_0^{R|x - y|^{-1}} \frac{1}{(1 + t^2)^{3+s}} |x - y|^3 t^2\,dt d\omega\\
\leq& \frac{C}{|x - y|^{3+2s}}.
\end{align*}
Thus, we have
\begin{align*}
&\int_\O \int_{B_R} \int_\O \int_{B_R} \frac{|f(x, v) - f(y, v)|^2}{(|x - y|^2 + |v - u|^2)^{3+s}}\,dv dx du dy\\
\leq& C \int_{B_R} \int_\O \int_\O \frac{|f(x, v) - f(y, v)|^2}{|x - y|^{3+2s}}\,dx dy dv\\
\leq& C \| f \|_{L^2(B_R; H^s(\O))}^2.
\end{align*}

In the same way, we obtain
\[
\int_\O \int_{B_R} \int_\O \int_{B_R} \frac{|f(y, v) - f(y, u)|^2}{(|x - y|^2 + |v - u|^2)^{3+s}}\,dv dx du dy \leq C \| f \|_{L^2(\O; H^s(B_R))}^2.
\]
Thus, we have
\[
| f |_{H^s(\O \times B_R)^2} \leq C \left( \| f \|_{L^2(B_R; H^s(\O))}^2 + \| f \|_{L^2(\O; H^s(B_R))}^2 \right),
\]
which implies the inclusion $L^2(B_R; H^s(\O)) \cap L^2(\O; H^s(B_R)) \subset H^s(\O \times B_R)$. This completes the proof.
\end{proof}

We are ready to prove Lemma \ref{lem:compact_VKSK}.

\begin{proof}[Proof of Lemma \ref{lem:compact_VKSK}]
Our proof consists of two steps. We first show that the operator $T_R V_\gamma K_\a^* S_\O K_\a: L^2_{0, \gamma/2}(\O \times \R^3) \to H^s(\O \times B_R)$ is bounded for some $s > 0$, which implies the compactness of the operator $T_R V_\gamma K_\a^* S_\O K_\a$ on $L^2_{0, \gamma/2}(\O \times \R^3)$. The second step is to show that the operator $T_R V_\gamma K_\a^* S_\O K_\a$ converges to $V_\gamma K_\a^* S_\O K_\a$ in the sense of the operator norm as $R \to \infty$.

For the first step, by Proposition \ref{prop:bounded_VK*} and Corollary \ref{cor:smoothing_x}, we see that $V_\gamma K_\a^* S_\O K_\a f \in L^2(B_R; H^{s_{2, \gamma}}(\O))$ for $f \in L^2_{0, \gamma/2}(\O \times \R^3)$, where $s_{2, \gamma}$ is the constant in \eqref{def:s2g}. Here, we used the equivalence between $L^2_{0, \gamma/2}(B_R)$ and $L^2(B_R)$ since the domain $B_R$ is bounded. Also, by Lemma \ref{lem:H1_VK*}, Lemma \ref{lem:Hs_VK*} and Proposition \ref{prop:bounded_SKa}, we see that $V_\gamma K_\a^* S_\O K_\a f \in L^2(\Omega; H^{s_{3, \gamma}}(B_R))$ for $f \in L^2_{0, \gamma/2}(\O \times \R^3)$, where
\begin{equation} \label{def:s3g}
s_{3, \gamma} :=
\begin{cases}
1, &-2 < \gamma \leq 1,\\
(3 + \gamma)/2 - \epsilon, &-3 < \gamma \leq -2
\end{cases}
\end{equation}
for any small $\epsilon > 0$. Letting $s := \min\{ s_{2, \gamma}, s_{3, \gamma} \}$, we have $V_\gamma K_\a^* S_\O K_\a f \in L^2(B_R; H^s(\O)) \cap L^2(\Omega; H^s(B_R))$, which with Proposition \ref{prop:preRellich} means that $V_\gamma K_\a^* S_\O K_\a f \in H^s(\O \times B_R)$ for $f \in L^2_{0, \gamma/2}(\O \times \R^3)$. Thus, the first step is completed.

For the second step, it suffices to check the convergence of $T_R V_\gamma K_\a^*$ to $V_\gamma K_\a^*$ with respect to the operator norm on $L^2_{0, \gamma/2}(\O \times \R^3)$. By the Cauchy-Schwarz inequality and Proposition \ref{prop:est_vka*v*}, we see that
\begin{align*}
&\| (I - T_R) V_\gamma K_\a^* f \|_{L^2_{0, \gamma/2}(\O \times \R^3)}^2\\ 
=& \int_{\{ |v| \geq R \}} \int_\O \left( \int_{\R^3} k_{\a, \gamma}^*(v, v^*) f(x, v^*)\,dv^* \right)^2 (1+|v|)^\gamma\,dxdv\\
\leq& \int_{\{ |v| \geq R \}} \int_\O \left( \int_{\R^3} |k_{\a, \gamma}^*(v, v^*)| \,dv^* \right)\\
&\times \left( \int_{\R^3} |k_{\a, \gamma}^*(v, v^*)| |f(x, v^*)|^2\,dv^* \right) (1+|v|)^\gamma\,dxdv\\
\leq& \int_{\{ |v| \geq R \}} \int_\O C(1 + |v|)^{\gamma-1} \left( \int_{\R^3} |k_{\a, \gamma}^*(v, v^*)| |f(x, v^*)|^2\,dv^* \right)\,dxdv\\
\leq& \frac{C}{1 + R} \int_\O \int_{\R^3} \left( \int_{\R^3} |k_{\a, \gamma}^*(v, v^*)| (1+|v|)^\gamma\,dv \right) |f(x, v^*)|^2\,dv^*\,dx\\
\leq& \frac{C}{1 + R} \| f \|_{L^2_{0, \gamma/2}(\O \times \R^3)}^2,
\end{align*}
which implies that
\[
\| (I - T_R) V_\gamma K_\a^* \| \leq \frac{C}{(1 + R)^{\frac{1}{2}}}.
\]
Therefore, the operator $T_R V_\gamma K_\a^*$ converges to $V_\gamma K_\a^*$ in the sense of the operator norm as $R \to \infty$. This completes the proof.
\end{proof}

We have shown that the operator $S_\O K$ is compact on $L^2_{\a, \gamma/2}(\O \times \R^3)$. As a consequence, we can say that the operator $I - S_\O K$ is a Fredholm operator on $L^2_{\a, \gamma/2}(\O \times \R^3)$. To prove Lemma \ref{lem:inverse}, it suffices to show its injectivity.

\begin{lemma} \label{lem:injective_L2}
For $-3 < \gamma \leq 1$, the operator $I - S_\O K: L^2_{0, \gamma/2}(\O \times \R^3) \to L^2_{0, \gamma/2}(\O \times \R^3)$ is injective. 
\end{lemma} 

\begin{proof}
Let $f$ be a function in $L^2_{0, \gamma/2}(\O \times \R^3)$ such that $(I - S_\O K) f = 0$. Then, it solves the homogeneous boundary value problem:
\[
\begin{cases}
v \cdot \nabla_x f = Lf &\mbox{ in } \O \times \R^3, \\
f = 0&\mbox{ on } \Gamma^-.
\end{cases}
\]
In the same way as in the proof of Lemma \ref{lem:compactL2}, Green's identity yields
\[
\frac{1}{2} \int_{\Gamma^+} |f(z, v)|^2 n(z) \cdot v\,d\sigma_z dv - \int_{\O \times \R^3} (Lf)f\,dxdv = 0. 
\]

Let $P$ be the projection operator on $L^2_{0 ,\gamma/2}(\O \times \R^3)$ to $\mathcal{N}(L)$, where $\mathcal{N}(L)$ is the null space of the operator $L$. It is known that 
\[
- \int_{\O \times \R^3} (Lf)f\,dxdv \geq c_0 \| (I - P) f \|_{L^2_{\nu}(\O \times \R^3)}^2
\]
with some positive constant $c_0$ \cite{2003Guo, Guo}, where
\[
\| f \|_{L^2_\nu(\O \times \R^3)}^2 := \int_\O \int_{\R^3} |f(x, v)|^2 \nu(v)\,dvdx.
\]
We mention that the norm $\| \cdot \|_{L^2_\nu(\O \times \R^3)}$ is equivalent to $\| \cdot \|_{L^2_{0, \gamma/2}(\O \times \R^3)}$ by Proposition \ref{prop:est_nu}. Thus, we have $(I - P) f = 0$, which means that $f$ satisfies $L f = 0$. In terms of the original boundary value problem, this implies that it solves the homogeneous boundary value problem of the free transport equation:
\[
\begin{cases}
v \cdot \nabla_x f = 0 &\mbox{ in } \O \times \R^3, \\
f = 0&\mbox{ on } \Gamma^-.
\end{cases}
\]
The method of characteristic lines shows that $f = 0$. This completes the proof.
\end{proof}

Since $L^2_{\a, \gamma/2}(\O \times \R^3) \subset L^2_{0, \gamma/2}(\O \times \R^3)$ for $0 \leq \a < 1/2$, Lemma \ref{lem:injective_L2} implies that the operator $I - S_\O K: L^2_{\a, \gamma/2}(\O \times \R^3) \to L^2_{\a, \gamma/2}(\O \times \R^3)$ is also injective. Lemma \ref{lem:inverse} follows from the Fredholm alternative theorem. This completes the proof of Lemma \ref{lem:existence_lin}.

\section{$L^2$-$L^\infty$ estimate} \label{sec:L2_Linfty}

In this section, we give a proof of Lemma \ref{lem:L2-Linfty} following an argument in \cite{RegularChen}. To this end, we introduce the following function space: For $\a \geq 0$, $\b \in \R$ and $2 \leq p, q < \infty$, we say $f \in L^p(\O; L^q_{\a, \b}(\R^3))$ if
\[
\| f \|_{L^p(\O; L^q_{\a, \b}(\R^3))} := \left( \int_\O \left( \int_{\R^3} |f(x, v)|^q e^{q \a |v|^2} (1 + |v|)^{q\b}\,dv \right)^{\frac{p}{q}}\,dx \right)^{\frac{1}{p}} < \infty.
\]
We remark that $L^p_{\a, \b}(\O \times \R^3) = L^p(\O; L^p_{\a, \b}(\R^3))$. In the same way, we say $f \in L^p(\O; L^\infty_{\a, \b}(\R^3))$ if
\[
\| f \|_{L^p(\O; L^\infty_{\a, \b}(\R^3))} := \left( \int_\O \left( \esssup_{v \in \R^3} |f(x, v)| e^{\a |v|^2} (1 + |v|)^\b \right)^p\,dx \right)^{\frac{1}{p}} < \infty.
\]
We remark that, for a bounded domain $\O$, the inclusion $L^\infty_{\a, \b}(\O \times \R^3) \subset L^p(\O; L^q_{\a, \b_q}(\R^3))$ holds for all $2 \leq p, q < \infty$ if $\b_q < \b - 3 q^{-1}$. Indeed, we have
\[
\int_{\R^3} |h(x, v)|^q e^{q \a|v|^2} (1 + |v|)^{q \b_q}\,dv \leq \| h \|_{L^\infty_{\a, \b}(\O \times \R^3)}^q \int_{\R^3} (1 + |v|)^{q (\b_q - \b)}\,dv
\]
for all $h \in L^\infty_{\a, \b}(\O \times \R^3)$ and the integral on the right hand side converges if $\b_q < \b - 3 q^{-1}$. Thus, the inclusion holds.

In what follows, for $1 \leq p \leq \infty$, let $p'$ denote the H\"older conjugate exponent of $p$.

We first introduce a decay property of $S_\O K$ with respect to the exponent $\b$.

\begin{lemma} \label{lem:polynomial_decay}
Let $\O$ be a bounded convex domain with $C^1$ boundary, $0 \leq \a < 1/2$, $\b_1 \in \R$, $-3 < \gamma \leq 1$ and $2 \leq p \leq \infty$. Then, the operator $S_\O K: L^p_{\a, \b_1}(\O \times \R^3) \to L^p_{\a, \b_1 + 3 - \gamma}(\O \times \R^3)$ is bounded.
\end{lemma}

\begin{proof}
For $2 \leq p < \infty$, by the H\"older inequality, Proposition \ref{prop:S_decay}, the identity \eqref{change of variables} and Proposition \ref{prop:est_ka}, we have
\begin{align*}
&\| S_\O K h \|_{L^p_{\a, \b_1 + 3 - \gamma}(\O \times \R^3)}^p\\
=& \int_\O \int_{\R^3} \left| \int_0^{\tau_-(x, v)} e^{-\nu(|v|) t} Kh (x - tv, v)\,dt \right|^p e^{p \a |v|^2} (1 + |v|)^{p(\b_1 + 3 - \gamma)}\,dvdx\\
\leq& \int_\O \int_{\R^3} \left( \int_0^{\tau_-(x, v)} e^{-\nu(|v|) t}\,dt \right)^{\frac{p}{p'}} \left(\int_0^{\tau_-(x, v)} e^{-\nu(|v|) t} |Kh (x - tv, v)|^p \,dt \right)\\
&\times e^{p \a |v|^2} (1 + |v|)^{p(\b_1 + 3 - \gamma)} \,dvdx\\
\leq& C \int_\O \int_{\R^3} \left(\int_0^{\tau_-(y, -u)} e^{-\nu(|u|) t}\,dt \right) |Kh (y, u)|^p e^{p \a |u|^2} (1 + |u|)^{p(\b_1 + 3 - \gamma) -\frac{p}{p'}}\,dudy\\
\leq& C \int_\O \int_{\R^3} \left( \int_{\R^3} |k_\a(u, u^*)|\,du^* \right)^{\frac{p}{p'}} \left( \int_{\R^3} |k_\a(u, u^*)| |h(y, u^*)|^p e^{p \a |u^*|^2}\,du^* \right)\\ 
&\times (1 + |u|)^{p(\b_1 + 2 - \gamma)}\,dudy\\
\leq& C \int_\O \int_{\R^3} \left( \int_{\R^3} |k_\a(u, u^*)| (1 + |u|)^{p(\b_1 + 2 - \gamma) +\frac{p}{p'}(\gamma - 2)}\,du \right) |h(y, u^*)|^p e^{p \a |u^*|^2}\,du^* dy\\
\leq& C \| h \|_{L^p_{\a, \b_1}(\O \times \R^3)}^p
\end{align*}
for all $h \in L^p_{\a, \b_1}(\O \times \R^3)$.

For $p = \infty$, by Proposition \ref{prop:est_ka}, we have
\begin{align*}
|Kh(x, v)| \leq& \int_{\R^3} |k_\a(v, v^*)| |h(x, v^*)| e^{\a|v^*|^2}\,dv^* e^{-\a|v|^2}\\
\leq& \| h(x, \cdot) \|_{L^\infty_{\a, \b_1}(\R^3)} \int_{\R^3} |k_\a(v, v^*)| (1 + |v^*|)^{-\b_1}\,dv^* e^{-\a|v|^2}\\
\leq& C \| h \|_{L^\infty_{\a, \b_1}(\O \times \R^3)} (1 + |v|)^{-\b_1 + \gamma - 2} e^{-\a|v|^2}
\end{align*}
for a.e.~$(x, v) \in \O \times \R^3$ and for all $h \in L^\infty_{\a, \b_1}(\O \times \R^3)$. The desired boundedness follows from the above estimate and Proposition \ref{prop:S_decay}. This completes the proof.
\end{proof}

We next show two regularization effects of $K$ in terms of the $L^p_v$ regularity.

\begin{lemma} \label{lem:Lp_Linfty_v}
Let $\O$ be a bounded convex domain with $C^1$ boundary, $0 \leq \a < 1/2$, $\b_2 \in \R$, $-3 < \gamma \leq 1$ and $2 \leq p_1, p_2 \leq \infty$. Suppose that $p_1' < \min \{3, 3/|\gamma| \}$. Then, the operator $K: L^{p_2}(\O; L^{p_1}_{\a, \b_2}(\R^3)) \to L^{p_2}(\O; L^\infty_{\a, \b_3}(\R^3))$ is bounded for $\b_3 = \b_2 + 2 - \gamma - p_1^{-1}$.
\end{lemma}

\begin{proof}
The case $p_1 = \infty$ can be handled in the same way as in the proof of Lemma \ref{lem:polynomial_decay}. In what follows, we discuss the case $2 \leq p_1 < \infty$. Let $h \in L^{p_2}(\O; L^{p_1}_{\a, \b_2}(\R^3))$. By the H\"older inequality, we have
\begin{align*}
|Kh(x, v)| \leq& \int_{\R^3} |k_\a(v, v^*)| |h(x, v^*)| e^{\a |v^*|^2}\,dv^* e^{-\a |v|^2}\\
\leq& \left( \int_{\R^3} |k_\a(v, v^*)|^{p_1'} (1 + |v^*|)^{- p_1' \b_2}\,dv^* \right)^{\frac{1}{p_1'}} \| h(x, \cdot) \|_{L^{p_1}_{\a, \b_2}(\R^3)} e^{-\a |v|^2}.
\end{align*}
Since $1 < p_1' < \min \{ 3, 3/|\gamma| \}$ by the assumption, we may apply Proposition \ref{prop:est_ka} to obtain
\[
\int_{\R^3} |k_\a(v, v^*)|^{p_1'} (1 + |v^*|)^{- p_1' \b_2}\,dv^* \leq C (1 + |v|)^{- p_1' \b_2 - 1 + p_1'(\gamma - 1)}.
\]
Thus, we have
\begin{align*}
|Kh(x, v)| \leq& C (1 + |v|)^{- \b_2 + \gamma - 1 - \frac{1}{p_1'}} \| h(x, \cdot) \|_{L^{p_1}_{\a, \b_2}(\R^3)} e^{-\a |v|^2}\\ 
=& C \| h(x, \cdot) \|_{L^{p_1}_{\a, \b_2}(\R^3)} e^{-\a |v|^2} (1 + |v|)^{-\b_3}
\end{align*}
for a.e.~$v \in \R^3$, or 
\[
\| Kh(x, \cdot) \|_{L^\infty_{\a, \b_3}(\R^3)} \leq C \| h(x, \cdot) \|_{L^{p_1}_{\a, \b_2}(\R^3)}.
\]
Taking $L^{p_2}_x$ integration of the above inequality concludes the lemma.
\end{proof}

\begin{lemma} \label{lem:Lp1_Lp2}
Let $\O$ be a bounded convex domain with $C^1$ boundary, $0 \leq \a < 1/2$, $\b_2 \in \R$, $-3 < \gamma \leq 1$ and $2 \leq p_1 < \infty$. Then, the following statements hold:
\begin{enumerate}
\item Suppose that $p_1' < \min\{ 3, 3/|\gamma| \}$. Then, for any $p_1 \leq p_2 < \infty$, the operator $K: L^{p_2}(\O; L^{p_1}_{\a, \b_2}(\R^3)) \to L^{p_2}_{\a, \b_4}(\O \times \R^3)$ is bounded for $\b_4 < \b_2 + 2 - \gamma - {p_1}^{-1} - 3 {p_2}^{-1}$.

\item Suppose that $p_1' \geq \min\{ 3, 3/|\gamma| \}$. Let $1 \leq q < \min\{3, 3/|\gamma|\}$ and define $p_2$ by $1 + p_2^{-1} = q^{-1} + p_1^{-1}$. Then, the operator $K: L^{p_2}(\O; L^{p_1}_{\a, \b_2}(\R^3)) \to L^{p_2}_{\a, \b_4}(\O \times \R^3)$ is bounded for $\b_4 = \b_2 + 2 - \gamma - p_1^{-1} + p_2^{-1}$.
\end{enumerate}
\end{lemma}

\begin{proof}
We first prove the statement (1). Let $h \in L^{p_2}(\O; L^{p_1}_{\a, \b_2}(\R^3))$. By Lemma \ref{lem:Lp_Linfty_v}, we see that $Kh \in L^{p_2}(\O; L^\infty_{\a, \b_3}(\R^3))$ with $\b_3 = \b_2 + 2 - \gamma - p_1^{-1}$. Thus, for $\b_4 < \b_3 - 3 p_2^{-1} = \b_2 + 2 - \gamma - p_1^{-1} - 3 p_2^{-1}$, we have
\begin{align*}
&\int_{\R^3} |Kh(x, v)|^{p_2} e^{p_2 \a |v|^2} (1 + |v|)^{p_2 \b_4}\,dv\\
\leq& \| Kh(x, \cdot) \|_{L^\infty_{\a, \b_3}(\R^3)}^{p_2} \int_{\R^3} (1 + |v|)^{p_2(\b_4 - \b_2 - 2 + \gamma) + 1}\,dv\\
\leq& C \| h(x, \cdot) \|_{L^{p_1}_{\a, \b_2}(\R^3)}^{p_2},
\end{align*}
or
\[
\| Kh(x, \cdot) \|_{L^{p_2}_{v, \a, \b_4}(\R^3)} \leq C \| h(x, \cdot) \|_{L^{p_1}_{\a, \b_2}(\R^3)}.
\]
This proves the statement (1).

For the statement (2), let $r_1^{-1} := p_1^{-1} - p_2^{-1}$ and $r_2^{-1} := q^{-1} - p_2^{-1}$. We notice that $r_1, r_2 \geq 1$ and
\[
\frac{1}{r_1} + \frac{1}{r_2} + \frac{1}{p_2} = \frac{1}{p_1} + \frac{1}{q} - \frac{1}{p_2} = 1.
\]
We decompose
\begin{align*}
|k(v, v^*) h(x, v^*)| =& \left( |h(x, v^*)|^{p_1} e^{p_1 \a |v^*|^2} (1 + |v^*|)^{p_1 \b_2} \right)^{\frac{1}{r_1}} \left( |k_\a(v, v^*)|^q \right)^{\frac{1}{r_2}}\\
&\times \left( |k_\a(v, v^*)|^q |h(x, v^*)|^{p_1} e^{p_1 \a |v^*|^2} (1 + |v^*|)^{-\frac{p_1 p_2}{r_1} \b_2} \right)^{\frac{1}{p_2}} e^{-\a |v|^2}.  
\end{align*}
The generalized H\"older inequality and Proposition \ref{prop:est_ka} yield
\begin{align*}
|Kh(x, v)| e^{\a |v|^2} \leq& \| h(x, \cdot) \|_{L^{p_1}_{\a, \b_2}(\R^3)}^{\frac{p_1}{r_1}} \left( \int_{\R^3} |k_\a(v, v^*)|^q\,dv^* \right)^{\frac{1}{r_2}}\\
&\times \left( \int_{\R^3} |k_\a(v, v^*)|^q |h(x, v^*)|^{p_1} e^{p_1 \a |v^*|^2} (1 + |v^*|)^{-\frac{p_1 p_2}{r_1} \b_2}\,dv^* \right)^{\frac{1}{p_2}}\\
\leq& C \| h(x, \cdot) \|_{L^{p_1}_{\a, \b_2}(\R^3)}^{\frac{p_1}{r_1}} (1 + |v|)^{\frac{q(\gamma - 1) - 1}{r_2}}\\
&\times \left( \int_{\R^3} |k_\a(v, v^*)|^q |h(x, v^*)|^{p_1} e^{p_1 \a |v^*|^2} (1 + |v^*|)^{-\frac{p_1 p_2}{r_1} \b_2}\,dv^* \right)^{\frac{1}{p_2}}.
\end{align*}
Thus, we have
\begin{align*}
&\int_{\R^3} |Kh(x, v)|^{p_2} e^{p_2 \a |v|^2} (1 + |v|)^{p_2 \b_4}\,dv\\
\leq& C \| h(x, \cdot) \|_{L^{p_1}_{\a, \b_2}(\R^3)}^{\frac{p_1 p_2}{r_1}} \int_{\R^3} \left(\int_{\R^3} |k_\a(v, v^*)|^q (1 + |v|)^{\frac{p_2 \{q(\gamma - 1) - 1\}}{r_2} + p_2 \b_4}\,dv \right)\\
&\times |h(x, v^*)|^{p_1} e^{p_1 \a |v^*|^2} (1 + |v^*|)^{-\frac{p_1 p_2}{r_1} \b_2}\,dv^*\\
\leq& C \| h(x, \cdot) \|_{L^{p_1}_{\a, \b_2}(\R^3)}^{\frac{p_1 p_2}{r_1}} \int_{\R^3} |h(x, v^*)|^{p_1} e^{p_1 \a |v^*|^2}\\
&\times (1 + |v^*|)^{\frac{p_2 \{q(\gamma - 1) - 1\}}{r_2} + q(\gamma - 1) - 1 + p_2 \b_4 - \frac{p_1 p_2}{r_1} \b_2}\,dv^*.
\end{align*}
Since
\[
\frac{p_2}{r_2} = \frac{p_2}{q} - 1, \quad \frac{p_2}{r_1} = \frac{p_2}{p_1} - 1,
\]
we see that
\[
\frac{p_2 \{q(\gamma - 1) - 1\}}{r_2} + q(\gamma - 1) - 1 + p_2 \b_4 - \frac{p_1 p_2}{r_1} \b_2 = p_1 \b_2.
\]
We obtain
\[
\| Kh(x, \cdot) \|_{L^{p_2}_{v, \a, \b_4}(\R^3)} \leq C \| h(x, \cdot) \|_{L^{p_1}_{\a, \b_2}(\R^3)}^{\frac{p_1}{r_1} + \frac{p_1}{p_2}} = C \| h(x, \cdot) \|_{L^{p_1}_{\a, \b_2}(\R^3)},
\]
which implies the statement (2) of Lemma \ref{lem:Lp1_Lp2}. This completes the proof.
\end{proof}

We also introduce two regularization effects of $S_\O K$ in terms of the $L^p_x$ regularity. 

\begin{lemma} \label{lem:Lpx}
Let $\O$ be a bounded convex domain with $C^1$ boundary, $0 \leq \a < 1/2$, $\b_2 \in \R$, $-3 < \gamma \leq 1$ and $2 \leq p_1 < \infty$. Then, the following statements hold:
\begin{enumerate}
\item Suppose that $p_1' < \min \{3, 3/|\gamma| \}$. Let $1 \leq q < 3/2$ and define $p_2 := q p_1$. Then, the operator $S_\O K: L^{p_1}_{\a, \b_2}(\O \times \R^3) \to L^{p_2}(\O; L^{p_1}_{\a, \b_4}(\R^3))$ is bounded for $\b_4 < \b_2 + 3 - \gamma - 4p_1^{-1}$.

\item Let $-3 < \gamma < -1$ and suppose that $p_1' \geq 3/|\gamma|$. Let 
\[
1 \leq q < \min \left\{ \frac{3}{2}, \frac{4}{1 + |\gamma|} \right\} 
\]
and define $p_2 = q p_1$. Then, the operator $S_\O K: L^{p_1}_{\a, \b_2}(\O \times \R^3) \to L^{p_2}(\O; L^{p_1}_{\a, \b_4}(\R^3))$ is bounded for $\b_4 < \b_2 + 3 - \gamma - 4 p_1^{-1} + 4 p_2^{-1}$.
\end{enumerate}
\end{lemma}

\begin{remark}
In the statement (2) of Lemma \ref{lem:Lpx}, since
\[
\min \left\{ \frac{3}{2}, \frac{4}{1 + |\gamma|} \right\} > 1
\]
for $-3 < \gamma < -1$, we may take $q$ strictly greater than $1$. 
\end{remark}

\begin{proof}
We first prove the statement (1). Let $h \in L^{p_1}_{\a, \b_2}(\O \times \R^3)$. By the H\"older inequality, Proposition \ref{prop:S_decay} and Lemma \ref{lem:Lp_Linfty_v}, we have
\begin{align*}
&\| S_\O K h (x, \cdot) \|_{L^{p_1}_{\a, \b_4}(\R^3)}^{p_1}\\
\leq& \int_{\R^3} \left( \int_0^{\tau_-(x, v)} e^{-{p_1}' \nu(|v|) t}\,dt \right)^{\frac{p_1}{p_1'}} \left( \int_0^{\tau_-(x, v)} |Kh(x - tv, v)|^{p_1}\,dt \right)\\
&\times e^{p_1\a|v|^2} (1 + |v|)^{p_1 \b_4}\,dv\\
\leq& C \int_{\R^3} \left( \int_0^{\tau_-(x, v)} \|Kh(x - tv, \cdot) \|_{L^\infty_{\a, \b_3}}^{p_1}\,dt \right) (1 + |v|)^{p_1 \b_4 - \frac{p_1}{p_1'} - p_1 \left( \b_2 + 2 - \gamma - \frac{1}{p_1} \right)}\,dv\\
\leq& C \int_{\R^3}  \left( \int_0^{\tau_-(x, v)} \|h(x - tv, \cdot) \|_{L^{p_1}_{\a, \b_2}}^{p_1}\,dt \right) (1 + |v|)^{p_1 \b_4 - p_1 \b_2 - p_1 (3 - \gamma) + 2}\,dv,
\end{align*}
where $\b_3 = \b_2 + 2 - \gamma - p_1^{-1}$. We change the variable of integration $r = |v| t$ and introduce the spherical coordinates $v = \rho \omega$ to obtain
\begin{align*}
&\int_{\R^3}  \left( \int_0^{\tau_-(x, v)} \|h(x - tv, \cdot) \|_{L^{p_1}_{\a, \b_2}}^{p_1}\,dt \right) (1 + |v|)^{p_1 \b_4 - p_1 \b_2 - p_1 (3 - \gamma) + 2}\,dv\\
=& \int_0^\infty \int_{S^2} \int_0^{|x - q(x, \omega)|} \| h(x - r\omega, \cdot) \|_{L^{p_1}_{\a, \b_2}(\R^3)}^{p_1} \rho (1 + \rho)^{p_1 \b_4 - p_1 \b_2 - p_1 (3 - \gamma) + 2}\,dr\,d\sigma_\omega d\rho\\
=& \left( \int_0^\infty \rho (1 + \rho)^{p_1 \b_4 - p_1 \b_2 - p_1 (3 - \gamma) + 2}\,d\rho \right)\\ 
&\times \left( \int_{S^2} \int_0^{|x - q(x, \omega)|} r^{-2} \| h(x - r\omega, \cdot) \|_{L^{p_1}_{\a, \b_2}(\R^3)}^{p_1} r^2\,dr\,d\sigma_\omega \right).
\end{align*}
The first integral with respect to $\rho$ is finite if 
\[
p_1 \b_4 - p_1 \b_2 - p_1 (3 - \gamma) + 2 + 1 < -1,
\]
which is equivalent to
\[
\b_4 < \b_2 + 3 - \gamma - \frac{4}{p_1}.
\]
Thus, we have
\[
\int_0^\infty \rho (1 + \rho)^{p_1 \b_4 - p_1 \b_2 - p_1 (3 - \gamma) + 2}\,d\rho \leq C.
\]
For the second integral, we change the variable of integration $y = x - r\omega$ to obtain
\begin{align*}
&\int_{S^2} \int_0^{|x - q(x, \omega)|} r^{-2} \| h(x - r\omega, \cdot) \|_{L^{p_1}_{\a, \b_2}(\R^3)}^{p_1} r^2\,dr\,d\sigma_\omega\\ 
=& \int_\O \frac{1}{|x - y|^2} \| h(y, \cdot) \|_{L^{p_1}_{\a, \b_2}(\R^3)}^{p_1}\,dy.
\end{align*}
Hence, we obtain
\begin{equation} \label{est:Young}
\| S_\O K h (x, \cdot) \|_{L^{p_1}_{\a, \b_4}(\R^3)} \leq C \left( \int_\O \frac{1}{|x - y|^2} \| h(y, \cdot) \|_{L^{p_1}_{\a, \b_2}(\R^3)}^{p_1}\,dy \right)^{\frac{1}{p_1}}.
\end{equation}

Let $1 \leq q < 3/2$. The H\"older inequality implies that
\begin{align*}
&\int_\O \frac{1}{|x - y|^2} \| h(y, \cdot) \|_{L^{p_1}_{\a, \b_2}(\R^3)}^{p_1}\,dy\\ 
\leq& \left( \int_\O \frac{1}{|x - y|^{2q}} \| h(y, \cdot) \|_{L^{p_1}_{\a, \b_2}(\R^3)}^{p_1}\,dy \right)^{\frac{1}{q}} \| h \|_{L^{p_1}_{\a, \b_2}(\O \times \R^3)}^{\frac{p_1}{q'}}.
\end{align*}
Thus, for $p_2 = q p_1$, we have
\begin{align*}
&\int_\O \| S_\O K h (x, \cdot) \|_{L^{p_1}_{\a, \b_4}(\R^3)}^{p_2}\,dx\\
\leq& C \int_\O \left( \int_\O \frac{1}{|x - y|^2} \| h(y, \cdot) \|_{L^{p_1}_{\a, \b_2}(\R^3)}^{p_1}\,dy \right)^q\,dx\\
\leq& C \| h \|_{L^{p_1}_{\a, \b_2}(\O \times \R^3)}^{(q - 1)p_1} \int_\O \left( \int_\O \frac{1}{|x - y|^{2q}}\,dx \right) \| h(y, \cdot) \|_{L^{p_1}_{\a, \b_2}(\R^3)}^{p_1}\,dy\\
\leq& C \| h \|_{L^{p_1}_{\a, \b_2}(\O \times \R^3)}^{p_2}.
\end{align*}
The statement (1) is proved.

Next, we proceed the proof of the statement (2). Let $h \in L^{p_1}_{\a, \b_2}(\O \times \R^3)$. By the same argument as the proof of the statement (1), we have
\begin{align*}
&\| S_\O K h (x, \cdot) \|_{L^{p_1}_{\a, \b_4}(\R^3)}^{p_1}\\
\leq& \int_{\R^3} \left( \int_0^{\tau_-(x, v)} e^{- \nu(|v|) t}\,dt \right)^{\frac{p_1}{p_1'}} \left( \int_0^{\tau_-(x, v)} e^{- \nu(|v|) t} |Kh(x - tv, v)|^{p_1}\,dt \right)\\
&\times e^{p_1\a|v|^2} (1 + |v|)^{p_1 \b_4}\,dv\\
\leq& C \int_{\R^3} \left( \int_0^{\tau_-(x, v)} e^{- \nu(|v|) t} |K h(x - tv, v) |^{p_1}\,dt \right) e^{p_1\a|v|^2} (1 + |v|)^{p_1 \b_4 - \frac{p_1}{p_1'}}\,dv\\
=& C \int_0^\infty \int_{S^2} \int_0^{|x - q(x, \omega)|} e^{-\frac{\nu(\rho)}{\rho}r} | K h(x - r\omega, \rho \omega) |^{p_1} e^{p_1 \a \rho^2}\\
&\times \rho (1 + \rho)^{p_1 \b_4 - (p_1 - 1)} \,dr\,d\sigma_\omega d\rho\\
=& C \int_0^\infty \int_\O e^{-\frac{\nu(\rho)}{\rho} |x - y|} \left| K h \left(y, \rho \frac{x - y}{|x - y|} \right) e^{\a \rho^2} \right|^{p_1}\,\frac{dy}{|x - y|^2} \rho (1 + \rho)^{p_1 \b_4 - (p_1 - 1)}\,d\rho.
\end{align*}
We focus on an estimate of $Kh$. We decompose the integrand as follows:
\begin{align*}
&\left|k \left( \rho \frac{x - y}{|x - y|}, v^* \right) h(y, v^*) e^{\a \rho^2} \right|\\
=& \left|k_\a \left( \rho \frac{x - y}{|x - y|}, v^* \right) \right|^{\frac{1}{p_1'}} (1 + |v^*|)^{-\b_2}\\
&\times \left|k_\a \left( \rho \frac{x - y}{|x - y|}, v^* \right) \right|^{1 - \frac{1}{p_1'}} |h(y, v^*)| e^{\a |v^*|^2} (1 + |v^*|)^{\b_2}
\end{align*}
By the H\"older inequality and Proposition \ref{prop:est_ka}, we have
\begin{align*}
&\left| K h \left(y, \rho \frac{x - y}{|x - y|} \right) e^{\a \rho^2} \right|^{p_1}\\
\leq& \left( \int_{\R^3} \left|k_\a \left( \rho \frac{x - y}{|x - y|}, v^* \right) \right| (1 + |v^*|)^{-p_1' \b_2}\,dv^*\right)^{\frac{p_1}{p_1'}}\\
&\times \left( \int_{\R^3} \left|k_\a \left( \rho \frac{x - y}{|x - y|}, v^* \right) \right| |h(y, v^*)|^{p_1} e^{p_1 \a |v^*|^2} (1 + |v^*|)^{p_1 \b_2}\,dv^* \right)\\
\leq& C (1 + \rho)^{-p_1 \b_2 + (p_1 - 1) (\gamma - 2)}\\
&\times \int_{\R^3} \left|k_\a \left( \rho \frac{x - y}{|x - y|}, v^* \right) \right| |h(y, v^*)|^{p_1} e^{p_1 \a |v^*|^2} (1 + |v^*|)^{p_1 \b_2}\,dv^*.
\end{align*}
Thus, we have
\begin{align*}
&\| S_\O K h (x, \cdot) \|_{L^{p_1}_{\a, \b_4}(\R^3)}^{p_1}\\
\leq& C \int_0^\infty \int_\O \int_{\R^3} \left|k_\a \left( \rho \frac{x - y}{|x - y|}, v^* \right) \right| |h(y, v^*)|^{p_1} e^{p_1 \a |v^*|^2} (1 + |v^*|)^{p_1 \b_2}\,dv^*\\
&\times e^{-\frac{\nu(\rho)}{\rho} |x - y|} \,\frac{dy}{|x - y|^2} \rho (1 + \rho)^{p_1 \b_4 - p_1 \b_2 + (p_1 - 1)(\gamma - 2)}\,d\rho\\
\leq& C \left( \int_0^\infty \int_\O \int_{\R^3} e^{- q\frac{\nu(\rho)}{\rho} |x - y|} \left|k_\a \left( \rho \frac{x - y}{|x - y|}, v^* \right) \right|^q \frac{1}{|x - y|^{2q}} \right.\\
&\times \left. |h(y, v^*)|^{p_1} e^{p_1 \a |v^*|^2} (1 + |v^*|)^{p_1 \b_2} \rho^q (1 + \rho)^{\b_\e} \,dv^*\,dy\,d\rho \right)^{\frac{1}{q}}\\
&\times \left( \int_0^\infty \int_\O \int_{\R^3} |h(y, v^*)|^{p_1} e^{p_1 \a |v^*|^2} (1 + |v^*|)^{p_1 \b_2}\,dv^* \,dy (1 + \rho)^{-(1 + \e)} \,d\rho\right)^{\frac{1}{q'}}\\
\leq& C_\e \| h \|_{L^{p_1}_{\a, \b_2}(\O \times \R^3)}^{\frac{p_1}{q'}} \left( \int_0^\infty \int_\O \int_{\R^3} e^{- q\frac{\nu(\rho)}{\rho} |x - y|} \left|k_\a \left( \rho \frac{x - y}{|x - y|}, v^* \right) \right|^q \right.\\
&\times \left. \frac{1}{|x - y|^{2q}} |h(y, v^*)|^{p_1} e^{p_1 \a |v^*|^2} (1 + |v^*|)^{p_1 \b_2} \rho^q (1 + \rho)^{\b_\e} \,dv^*\,dy\,d\rho \right)^{\frac{1}{q}},
\end{align*}
where $q \geq 1$ is the constant in the statement, $\e$ is a small positive number and
\[
\b_\e := p_1 q (\b_4 - \b_2 + \gamma - 3) + q(4 - \gamma + \e) - (1 + \e).
\]
For $p_2 = q p_1$, we have
\begin{align*}
&\int_\O \| S_\O K h (x, \cdot) \|_{L^{p_1}_{\a, \b_4}(\R^3)}^{p_2}\,dx\\
\leq& C_\e \| h \|_{L^{p_1}_{\a, \b_2}(\O \times \R^3)}^{p_1 (q - 1)} \int_\O \int_{\R^3} \left( \int_0^\infty \int_\O e^{- q\frac{\nu(\rho)}{\rho} |x - y|} \left|k_\a \left( \rho \frac{x - y}{|x - y|}, v^* \right) \right|^q \right.\\
&\times \left. \frac{1}{|x - y|^{2q}} \rho^q (1 + \rho)^{\b_\e}\,dx d\rho \right) |h(y, v^*)|^{p_1} e^{p_1 \a |v^*|^2} (1 + |v^*|)^{p_1 \b_2} \,dv^*\,dy.
\end{align*}

We handle the inner integral with respect to $x$ and $\rho$. By change of variables $x = y + r\omega$, we have
\begin{align*}
&\int_0^\infty \int_\O e^{- q\frac{\nu(\rho)}{\rho} |x - y|} \left|k_\a \left( \rho \frac{x - y}{|x - y|}, v^* \right) \right|^q \frac{1}{|x - y|^{2q}} \rho^q (1 + \rho)^{\b_\e}\,dx d\rho\\
=& \int_0^\infty \int_{S^2} \left( \int_0^{|y - q(y, \omega)|} e^{- q\frac{\nu(\rho)}{\rho} r} r^{2(1 - q)}\,dr \right) |k_\a( \rho \omega, v^*)|^q\,d\sigma_\omega \rho^q (1 + \rho)^{\b_\e}\,d\rho.
\end{align*}
We claim that
\[
\int_0^{|y - q(y, \omega)|} e^{- q\frac{\nu(\rho)}{\rho} r} r^{2(1 - q)}\,dr \leq C \min \left\{ 1, \frac{\rho^{3 - 2q}}{(1 + \rho)^{\gamma(3 - 2q)}} \right\}
\]
for $1 \leq q < 3/2$. On the one hand, we have
\[
\int_0^{|y - q(y, \omega)|} e^{- q\frac{\nu(\rho)}{\rho} r} r^{2(1 - q)}\,dr \leq \int_0^{|y - q(y, \omega)|} r^{2(1 - q)}\,dr \leq C.
\]
On the other hand, by changing the variable $r = (\rho/\nu(\rho))t$, we have
\[
\int_0^{|y - q(y, \omega)|} e^{- q\frac{\nu(\rho)}{\rho} r} r^{2(1 - q)}\,dr \leq \left( \frac{\rho}{\nu(\rho)} \right)^{3 - 2q} \int_0^\infty e^{-qt} t^{2(1 - q)}\,dt \leq C \frac{\rho^{3 - 2q}}{(1 + \rho)^{\gamma(3 - 2q)}}.
\]
Thus, the claim is proved. Using the claimed estimate and the change of variables $v = \rho \omega$, we have
\begin{align*}
&\int_0^\infty \int_{S^2} \left( \int_0^{|y - q(y, \omega)|} e^{- q\frac{\nu(\rho)}{\rho} r} r^{2(1 - q)}\,dr \right) |k_\a( \rho \omega, v^*)|^q\,d\sigma_\omega \rho^q (1 + \rho)^{\b_\e}\,d\rho\\
\leq& C \int_{\{ |v| > 1 \}} |k_\a(v, v^*)|^q |v|^{q - 2} (1 + |v|)^{\b_\e}\,dv\\
+& C \int_{\{ |v| \leq 1 \}} |k_\a(v, v^*)|^q |v|^{1 - q} (1 + |v|)^{\b_\e - \gamma (3 - 2q)}\,dv.
\end{align*}
For the former part, we get
\begin{align*}
\int_{\{ |v| > 1 \}} |k_\a(v, v^*)|^q |v|^{q - 2} (1 + |v|)^{\b_\e}\,dv \leq& C \int_{\R^3} |k_\a(v, v^*)|^q (1 + |v|)^{\b_\e + q - 2}\,dv\\
\leq& C (1 + |v^*|)^{\b_\e + q - 2 + q(\gamma - 1) - 1}\\
=& C (1 + |v^*|)^{p_1 q (\b_4 - \b_2 + \gamma - 3) + (q - 1)(4 + \e)}.
\end{align*}
For the latter part, since $1 \leq 1 + |v| \leq 2$ for $|v| \leq 1$, we obtain
\[
\int_{\{ |v| \leq 1 \}} |k_\a(v, v^*)|^q |v|^{1 - q} (1 + |v|)^{\b_\e - \gamma (3 - 2q)}\,dv \leq C \int_{\R^3} |v|^{1 - q} |k_\a(v, v^*)|^q\,dv.
\]
For $-3 < \gamma < -1$, the right hand side is finite if
\[
1 - q + \gamma q > -3,
\]
which is equivalent to 
\[
q < \frac{4}{1 + |\gamma|}.
\]
Thus, we have
\begin{align*}
&\int_0^\infty \int_{S^2} \left( \int_0^{|y - q(y, \omega)|} e^{- q\frac{\nu(\rho)}{\rho} r} r^{2(1 - q)}\,dr \right) |k_\a( \rho \omega, v^*)|^q\,d\sigma_\omega \rho^q (1 + \rho)^{\b_\e}\,d\rho\\
\leq& C (1 + |v^*|)^{p_1 q (\b_4 - \b_2 + \gamma - 3) + (q - 1)(4 + \e)}.
\end{align*}
The above estimate is uniformly bounded if
\[
p_1 q (\b_4 - \b_2 + \gamma - 3) + (q - 1)(4 + \e) \leq 0.
\]
Since $\e > 0$ is arbitrarily small, it is sufficient that
\[
\b_4 < \b_2 + 3 - \gamma + \frac{4}{p_2} - \frac{4}{p_1}.
\]

Therefore, under the assumptions of the statement (2), we have
\begin{align*}
&\int_\O \| S_\O K h (x, \cdot) \|_{L^{p_1}_{\a, \b_4}(\R^3)}^{p_2}\,dx\\
\leq& C \| h \|_{L^{p_1}_{\a, \b_2}(\O \times \R^3)}^{p_1 (q - 1)} \int_\O \int_{\R^3} |h(y, v^*)|^{p_1} e^{p_1 \a |v^*|^2} (1 + |v^*|)^{p_1 \b_2} \,dv^*\,dy\\
=& C \| h \|_{L^{p_1}_{\a, \b_2}(\O \times \R^3)}^{p_2}.
\end{align*}
The statement (2) is proved.
\end{proof}

\begin{lemma} \label{lem:Lp_Linfty_x}
Let $\O$ be a bounded convex domain with $C^1$ boundary, $0 \leq \a < 1/2$, $\b_2 \in \R$, $-3 < \gamma \leq 1$ and $2 \leq p < \infty$. Suppose that $p' < \min \{ 3/2, 3/|\gamma| \}$. Take $p_1 \geq p$ and define $p_2$ by
\[
1 + \frac{p}{p_2} = \frac{1}{p'} + \frac{p}{p_1}.
\]
Then, the operator $S_\O K: L^{p_1}(\O; L^p_{\a, \b_2}(\R^3)) \to L^{p_2}(\O; L^p_{\a, \b_4}(\R^3))$ is bounded for $\b_4 < \b_2 + 3 - \gamma - 4 p^{-1}$. 
\end{lemma}

\begin{proof}
We first consider the case where $p_2 < \infty$. In the same way as in the proof of Lemma \ref{lem:Lpx}, we obtain the estimate \eqref{est:Young} for $h \in L^{p_1}_x(\O; L^p_{\a, \b_2}(\R^3))$ with $\b_4 < \b_2 + 3 - \gamma - 4 p^{-1}$. We let 
\[
\frac{1}{r_1} := \frac{p}{p_1} - \frac{p}{p_2}, \quad \frac{1}{r_2} := \frac{1}{p'} - \frac{p}{p_2},
\]
and decompose the integrand as
\begin{align*}
\frac{1}{|x - y|^2} \| h(y, \cdot) \|_{L^p_{\a, \b_2}(\R^3)}^p =& \| h(y, \cdot) \|_{L^p_{\a, \b_2}(\R^3)}^{\frac{p_1}{r_1}} \left( \frac{1}{|x - y|^{2p'}} \right)^{\frac{1}{r_2}}\\
&\times \left( \frac{1}{|x - y|^{2p'}} \| h(y, \cdot) \|_{L^p_{\a, \b_2}(\R^3)}^{p_1} \right)^{\frac{p}{p_2}}. 
\end{align*}
The generalized H\"older inequality yields
\begin{align*}
&\int_\O \frac{1}{|x - y|^2} \| h(y, \cdot) \|_{L^p_{\a, \b_2}}^p\,dy\\ 
\leq& \left( \int_\O \| h(y, \cdot) \|_{L^p_{\a, \b_2}(\R^3)}^{p_1}\,dy \right)^{\frac{1}{r_1}} \left( \int_\O \frac{1}{|x - y|^{2p'}}\,dy \right)^{\frac{1}{r_2}}\\
&\times \left( \int_\O \frac{1}{|x - y|^{2p'}} \| h(y, \cdot) \|_{L^p_{\a, \b_2}(\R^3)}^{p_1}\,dy \right)^{\frac{p}{p_2}}\\
\leq& C \| h \|_{L^{p_1}_x(\O; L^p_{\a, \b_2}(\R^3))}^{\frac{p_1}{r_1}} \left( \int_\O \frac{1}{|x - y|^{2p'}} \| h(y, \cdot) \|_{L^p_{\a, \b_2}(\R^3)}^{p_1}\,dy \right)^{\frac{p}{p_2}}.
\end{align*}
Thus, we obtain
\begin{align*}
&\int_\O \| S_\O K h(x, \cdot) \|_{L^p_{\a, \b_4}(\R^3)}^{p_2}\,dx\\
\leq& C \| h \|_{L^{p_1}_x(\O; L^p_{\a, \b_2}(\R^3))}^{\frac{p_1}{r_1} \frac{p_2}{p}} \int_\O \left( \int_\O \frac{1}{|x - y|^{2p'}}\,dx \right) \| h(y, \cdot) \|_{L^p_{\a, \b_2}(\R^3)}^{p_1}\,dy\\
\leq& C \| h \|_{L^{p_1}_x(\O; L^p_{\a, \b_2}(\R^3))}^{p_2}.
\end{align*}

For the case $p_2 = \infty$, we just use the H\"older inequality to obtain
\begin{align*}
\int_\O \frac{1}{|x - y|^2} \| h(y, \cdot) \|_{L^p_{\a, \b_2}}^p\,dy \leq& \left( \int_\O \frac{1}{|x - y|^{2p'}}\,dy \right)^{\frac{1}{p'}} \left( \int_\O \| h(y, \cdot) \|_{L^p_{\a, \b_2}}^{p \frac{p_1}{p}}\,dy \right)^{\frac{p}{p_1}}\\
\leq& C \| h \|_{L^{p_1}_x(\O; L^p_{\a, \b_2}(\R^3))}^p
\end{align*}
for all $x \in \O$, and hence we have
\[
\| S_\O K h \|_{L^\infty_x(\O; L^p_{\a, \b_4}(\R^3))} \leq C \| h \|_{L^{p_1}_x(\O; L^p_{\a, \b_2}(\R^3))}.
\]
This completes the proof.
\end{proof}

We are ready to prove Lemma \ref{lem:L2-Linfty}. Let $f$ be the solution to the integral equation \eqref{IE} in $L^2_{\a, \gamma/2}(\O \times \R^3)$, whose existence was shown by Lemma \ref{lem:existence_lin}. 

We first consider the case where $-3/2 < \gamma \leq 1$. In this case, we have $2 < \min \{ 3, 3/|\gamma| \}$. Hence, we apply the statement (1) of Lemma \ref{lem:Lpx} to obtain $S_\O K f \in L^{2q}(\O; L^2_{\a, \b_4}(\R^3))$ with $1 \leq q < 3/2$ and $\b_4 < \gamma/2 + 3 - \gamma - 2 = 1 - \gamma/2$. For the sake of simplicity, let $q = 5/4$. Since $J f_0 + S_\O \phi \in L^\infty_{\a, \b}(\O \times \R^3) \subset L^{5/2}(\O; L^2_{\a, \b_2}(\R^3))$ with $\b_2 < \b - 3/2$, the equation \eqref{IE} implies that $f \in  L^{5/2}(\O; L^2_{\a, \b_2}(\R^3))$ with $\b_2 < \min \{1 - \gamma/2, \b - 3/2 \}$. Applying the statement (1) of Lemma \ref{lem:Lp1_Lp2} and Lemma \ref{lem:polynomial_decay} if necessary, we obtain $f \in L^{5/2}_{\a, \b_1}(\O \times \R^3)$ with $\b_1 < \b - 6/5$. The same procedure yields $f \in L^{25/8}_{\a, \b_1}(\O \times \R^3)$ with $\b_1 < \b - 24/25$. We notice that $(25/8)' = 25/17 < \min \{ 3/2, 3/|\gamma| \}$. Thus, we make use of Lemma \ref{lem:Lp_Linfty_x} to obtain $f \in L^\infty(\O; L^{25/8}_{\a, \b_2}(\R^3))$ for some $\b_2 \in \R$. Lemma \ref{lem:Lp_Linfty_v} and Lemma \ref{lem:polynomial_decay} give us that $f \in L^\infty_{\a, \b}(\O \times \R^3)$.

Since $L^\infty_{\a, \b}(\O \times \R^3)$ is contained in all the function spaces appearing in the above argument, we can see that
\[
\| f \|_{L^\infty_{\a, \b}(\O \times \R^3)} \leq C \left( \| f \|_{L^2_{\a, \gamma/2}(\O \times \R^3)} + \| Jf_0 + S_\O \phi \|_{L^\infty_{\a, \b}(\O \times \R^3)} \right).
\]
By Lemma \ref{lem:inverse}, we have
\begin{align*}
\| f \|_{L^2_{\a, \gamma/2}(\O \times \R^3)} \leq& C \| Jf_0 + S_\O \phi \|_{L^2_{\a, \gamma/2}(\O \times \R^3)}\\
\leq& C \| Jf_0 + S_\O \phi \|_{L^\infty_{\a, \b}(\O \times \R^3)}.
\end{align*}
Thus, by Lemma \ref{lem:est_J_infty} and Proposition \ref{prop:S_bound}, we obtain
\begin{align*}
\| f \|_{L^\infty_{\a, \b}(\O \times \R^3)} \leq& C \| Jf_0 + S_\O \phi \|_{L^\infty_{\a, \b}(\O \times \R^3)}\\
\leq& C \left( \| Jf_0 \|_{L^\infty_{\a, \b}(\O \times \R^3)} + \| S_\O \phi \|_{L^\infty_{\a, \b}(\O \times \R^3)} \right)\\
\leq& C \left( \| f_0 \|_{L^\infty_{\a, \b}(\Gamma^-)} + \| \phi \|_{L^\infty_{\a, \b - \gamma}(\O \times \R^3)} \right).
\end{align*}

We next consider the case where $-3 < \gamma \leq -3/2$. In this case, we apply the statement (2) of Lemma \ref{lem:Lpx}, the statement of Lemma \ref{lem:Lp1_Lp2} and Lemma \ref{lem:polynomial_decay} to obtain $f \in L^{p_1}_{\a, \b_1}(\O \times \R^3)$ for $p_1 > 2$ and $\b_1 < \b - 3/p_1$. Through the same procedure inductively, we obtain $f \in L^{p_2}_{\a, \b_2}(\O \times \R^3)$ with $p_2' < \min \{3/2, 3/|\gamma| \}$ and $\b_2 < \b - 3/p_2$. Therefore, we apply Lemma \ref{lem:Lp_Linfty_x}, Lemma \ref{lem:Lp_Linfty_v} and Lemma \ref{lem:polynomial_decay} to reach at $f \in L^\infty_{\a, \b}(\O \times \R^3)$. The desired $L^\infty$ estimate is also obtained in the same way. This completes the proof of Lemma \ref{lem:L2-Linfty}.

\section{$H^s_x$ regularity} \label{sec:Hsx}

In this subsection, we discuss $H^s_x$ regularity of the solution to the integral equation \eqref{IE}. We first discuss it for source terms.

\begin{lemma} \label{lem:Hs_Jf0}
Let $\O$ be a bounded convex domain with $C^1$ boundary, $0 \leq \a < 1/2$, $\b > (3 + \gamma)/2$, $-3 < \gamma \leq 1$ and $0 < s_1 \leq 1$. Then, for any $0 < s < s_1$, the operator $J: \cB^{s_1}_{\a, \b}(\Gamma^-) \to X^s_{\a, \b, \g}$ is bounded.
\end{lemma}

To prove Lemma \ref{lem:Hs_Jf0}, we introduce some geometric estimates. For $x, y \in \O$ and $v \in \R^3 \setminus \{ 0 \}$, let 
\begin{align*}
N(x, v) :=& \frac{|n(q(x, v)) \cdot v|}{|v|},\\
N(x, y, v) :=& \min \{ N(x, v), N(y, v) \}.
\end{align*}
We have the following H\"older estimates.

\begin{proposition}[{\cite[Proposition 6.2]{CHK}}] \label{prop:EPT}
Let $\Omega$ be a bounded convex domain in $\R^3$ with $C^1$ boundary. For $x, y \in \Omega$ and $v \in \R^3 \setminus \{ 0 \}$, we have
\begin{align*}
&|q(x, v) - q(y, v)| \leq \frac{|x - y|}{N(x, y, v)},\\
&|\tau_-(x, v) - \tau_-(y, v)| \leq \frac{2 |x - y| }{N(x, y, v) |v| }.
\end{align*}
\end{proposition}

\begin{proof}[Proof of Lemma \ref{lem:Hs_Jf0}]
Lemma \ref{lem:est_J_infty} implies that $Jf_0 \in L^\infty_{\a, \b}(\O \times \R^3)$. Thus, it suffices to discuss the $H^s$ seminorm of $Jf_0$.

Due to the estimate,
\begin{align*}
|Jf_0(x, v) - Jf_0(y, v)| \leq& |e^{-\nu(|v|) \tau_-(x, v)} - e^{-\nu(|v|)\tau_-(y, v)}| |f_0(q(x, v), v)|\\
&+ e^{-\nu(|v|) \tau_-(y, v)} |f_0(q(x, v), v) - f_0(q(y, v), v)|,
\end{align*}
we have
\begin{align*}
&|J f_0|_{L^2_{\a, \gamma/2}(\R^3; H^s(\O))}^2\\ 
=& \int_{\R^3} \left( \int_\O \int_\O \frac{|Jf_0(x, v) - Jf_0(y, v)|^2}{|x - y|^{3 + 2s}}\,dxdy \right) e^{2\a|v|^2} (1 + |v|)^\gamma \,dv\\
\leq& C \int_{\R^3} \left( \int_\O \int_\O \frac{|e^{-\nu(|v|) \tau_-(x, v)} - e^{-\nu(|v|)\tau_-(y, v)}|^2 |f_0(q(x, v), v)|^2}{|x - y|^{3 + 2s}}\,dxdy \right)\\
&\times e^{2\a|v|^2} (1 + |v|)^\gamma \,dv\\
&+ C \int_{\R^3} \left( \int_\O \int_\O \frac{e^{-2\nu(|v|) \tau_-(y, v)} |f_0(q(x, v), v) - f_0(q(y, v), v)|^2}{|x - y|^{3 + 2s}}\,dxdy \right)\\
&\quad \times e^{2\a|v|^2} (1 + |v|)^\gamma \,dv.
\end{align*}

For the first term on the right hand side, we apply the fundamental theorem of calculus and Proposition \ref{prop:EPT} to obtain
\begin{align*}
&|e^{-\nu(|v|) \tau_-(x, v)} - e^{-\nu(|v|)\tau_-(y, v)}| |f_0(q(x, v), v)|\\
\leq& 2^{1 - s_1} \left| \nu(|v|) \int_{\tau_-(y, v)}^{\tau_-(x, v)} e^{-\nu(|v|)t}\,dt \right|^{s_1} \| f_0 \|_{L^\infty_{\a, \b}(\Gamma^-)} e^{- \a |v|^2} (1 + |v|)^{-\b}\\
\leq& C |\tau_-(x, v) - \tau_-(y, v)|^{s_1} \| f_0 \|_{L^\infty_{\a, \b}(\Gamma^-)} e^{- \a |v|^2} (1 + |v|)^{s_1 \gamma - \b}\\
\leq& C \frac{|x - y|^{s_1}}{N(x, y, z)^{s_1} |v|^{s_1}} \| f_0 \|_{L^\infty_{\a, \b}(\Gamma^-)} e^{- \a |v|^2} (1 + |v|)^{s_1 \gamma - \b}.
\end{align*}
Thus, we have
\begin{align*}
&\int_{\R^3} \left( \int_\O \int_\O \frac{|e^{-\nu(|v|) \tau_-(x, v)} - e^{-\nu(|v|)\tau_-(y, v)}|^2 |f_0(q(x, v), v)|^2}{|x - y|^{3 + 2s}}\,dxdy \right)\\
& \times e^{2\a|v|^2} (1 + |v|)^\gamma \,dv\\
\leq& C \| f_0 \|_{L^\infty_{\a, \b}(\Gamma^-)}^2 \int_{\R^3} \left( \int_\O \int_\O \frac{1}{|x - y|^{3 - 2(s_1 - s)}N(x, y, v)^{2s_1}}\,dxdy \right)\\
&\times \frac{(1 + |v|)^{(2s_1 + 1)\gamma  - 2\b}}{|v|^{2s_1}}\,dv.
\end{align*}
For the inner integral, we recall the definition of $N(x, y, v)$ to obtain
\begin{align*}
&\int_\O \int_\O \frac{1}{|x - y|^{3 - 2(s_1 - s)}N(x, y, v)^{2s_1}}\,dxdy\\
\leq& \int_\O \frac{1}{N(x, v)^{2s_1}} \left( \int_\O \frac{1}{|x - y|^{3 - 2(s_1 - s)}}\,dy \right)\,dx\\
&+ \int_\O \frac{1}{N(y, v)^{2s_1}} \left( \int_\O \frac{1}{|x - y|^{3 - 2(s_1 - s)}}\,dx \right)\,dy\\
\leq& C \int_\O \frac{1}{N(x, v)^{2s_1}}\,dx.
\end{align*}
We employ Lemma \ref{change_integration} to get
\begin{align*}
&\int_{\R^3} \left( \int_\O \int_\O \frac{1}{|x - y|^{3 - 2(s_1 - s)}N(x, y, v)^{2s_1}}\,dxdy \right) \frac{(1 + |v|)^{(2s_1 + 1)\gamma  - 2\b}}{|v|^{2s_1}}\,dv\\
\leq& C \int_{\R^3} \int_\O \frac{1}{N(x, v)^{2s_1}} \frac{(1 + |v|)^{(2s_1 + 1)\gamma  - 2\b}}{|v|^{2s_1}}\,dxdv\\
=& C \int_{\Gamma^-} \int_0^{\tau_+(z, v)} \frac{1}{N(z + tv, v)^{2s_1}} \frac{(1 + |v|)^{(2s_1 + 1)\gamma  - 2\b}}{|v|^{2s_1}}\,N(z, v)|v|\,dtd\sigma_zdv\\
\leq& C \int_{\p \O} \left( \int_{\Gamma^-_z} \frac{1}{N(z, v)^{2s_1-1}} \frac{(1 + |v|)^{(2s_1 + 1)\gamma  - 2\b}}{|v|^{2s_1}}\,dv \right)\,d\sigma_z,
\end{align*}
where
\[
\Gamma^-_z := \{ v \in \R^3 \mid n(z) \cdot v < 0 \}.
\]
Here, we have used the fact that $N(z + tv, v) = N(z, v)$ for all $0 < t < \tau_+(z, v)$. We introduce the spherical coordinates for the $v$ integral so that $\theta = 0$ corresponds to the $n(z)$ direction to obtain
\begin{align*}
&\int_{\Gamma^-_z} \frac{1}{N(z, v)^{2s_1-1}} \frac{(1 + |v|)^{(2s_1 + 1)\gamma  - 2\b}}{|v|^{2s_1}}\,dv\\
=& 2 \pi \int_0^{\frac{\pi}{2}} \int_0^{\infty} \frac{1}{\cos^{2s_1 - 1} \theta} \rho^{2 - 2s_1} (1 + \rho)^{(2s_1 + 1)\gamma - 2\b} e\,\sin \theta\,d\rho d\theta\\
=& 2\pi \left( \int_0^1 \frac{1}{t^{2s_1 - 1}}\,dt \right) \left( \int_0^{\infty} \rho^{2 - 2s_1} (1 + \rho)^{(2s_1 + 1)\gamma - 2\b} \,d\rho \right).
\end{align*}
We notice that $2 - 2s_1 \geq 0$ and
\[
2 - 2s_1 + (2s_1 + 1)\gamma - 2\b = 2 + \gamma - 2\b + 2(\gamma - 1) s_1 < -1
\]
for $\b > (3 + \gamma)/2$, $-3 < \gamma \leq 1$ and $0 < s_1 \leq 1$. Thus the above integral with respect to $\rho$ is convergent. We have
\[
\int_{\Gamma^-_z} \frac{1}{N(z, v)^{2s_1-1}} \frac{(1 + |v|)^{(2s_1 + 1)\gamma  - 2\b}}{|v|^{2s_1}}\,dv \leq C.
\]
for all $z \in \p \O$, and hence we have
\begin{align*}
&\int_{\R^3} \left( \int_\O \int_\O \frac{|e^{-\nu(|v|) \tau_-(x, v)} - e^{-\nu(|v|)\tau_-(y, v)}|^2 |f_0(q(x, v), v)|^2}{|x - y|^{3 + 2s}}\,dxdy \right)\\
& \times e^{2\a|v|^2} (1 + |v|)^\gamma\,dv \leq C \| f_0 \|_{L^\infty_{\a, \b}(\Gamma^-)}^2.
\end{align*}

For the second term on the right hand side, we use the assumption $f_0 \in \cB^{s_1}_{\a, \b}$ and Proposition \ref{prop:EPT} to get
\begin{align*}
&e^{-\nu(|v|) \tau_-(y, v)} |f_0(q(x, v), v) - f_0(q(y, v), v)|\\ 
\leq& \| f_0 \|_{\cB^{s_1}_{\a, \b}(\Gamma^-)} |q(x, v) - q(y, v)|^{s_1} e^{-\a |v|^2} (1 + |v|)^{-\b}\\
\leq& C \| f_0 \|_{\cB^{s_1}_{\a, \b}(\Gamma^-)} \frac{|x - y|^{s_1}}{N(x, y, v)^{s_1}} e^{-\a |v|^2} (1 + |v|)^{-\b}.
\end{align*}
Thus, in the same fashion, we have
\begin{align*}
&\int_{\R^3} \left( \int_\O \int_\O \frac{e^{-2\nu(|v|) \tau_-(y, v)} |f_0(q(x, v), v) - f_0(q(y, v), v)|^2}{|x - y|^{3 + 2s}}\,dxdy \right)\\
&\quad \times e^{2\a|v|^2} (1 + |v|)^\gamma\,dv\\
\leq& C \| f_0 \|_{\cB^{s_1}_{\a, \b}(\Gamma^-)}^2 \int_{\R^3} \left( \int_\O \int_\O \frac{1}{N(x, y, v)^{2s_1}|x - y|^{3 - 2(s_1 - s)}}\,dxdy \right) (1 + |v|)^{\gamma - 2\b} \,dv\\
\leq& C \| f_0 \|_{\cB^{s_1}_{\a, \b}(\Gamma^-)}^2 \int_{\R^3} \int_\O \frac{1}{N(x, v)^{2s_1}} (1 + |v|)^{\gamma - 2\b} \,dxdv\\
=& C \| f_0 \|_{\cB^{s_1}_{\a, \b}(\Gamma^-)}^2 \left( \int_0^1 \frac{1}{t^{2s_1 - 1}}\,dt \right) \left( \int_0^\infty \rho^2 (1 + \rho)^{\gamma - 2\b}\,d\rho \right)\\
\leq& C \| f_0 \|_{\cB^{s_1}_{\a, \b}(\Gamma^-)}^2.
\end{align*}

Therefore, we conclude that $|J f_0|_{L^2_{\a, \gamma/2}(\R^3; H^s(\O))} \leq C \| f_0 \|_{\cB^{s_1}_{\a, \b}(\Gamma^-)}^2$ for $0 < s < s_1$. This completes the proof.
\end{proof}

\begin{lemma} \label{lem:Hs_S}
Let $\O$ be a bounded convex domain with $C^1$ boundary, $0 \leq \a < 1/2$, $\b > (3 + \gamma)/2$, $-3 < \gamma \leq 1$ and $0 < s < 1$. Then, the operator $S_\O: Y^s_{\a, \b, \g} \to X^s_{\a, \b, \g}$ is bounded.
\end{lemma}

\begin{proof}
Let $\phi \in Y^s_{\a, \b, \g}$. By Corollary \ref{cor:S_bound2}, we have $\| S_\O \phi \|_{L^\infty_{\a, \b}(\O \times \R^3)} \leq C \| \phi \|_{L^\infty_{\a, \b - \gamma}(\O \times \R^3)}$. Also, by Corollary \ref{cor:S_bound}, we see that $\| S_\O \phi \|_{L^2_{\a, \gamma/2}(\O \times \R^3)} \leq C \| \phi \|_{L^2_{\a, -\gamma/2}(\O \times \R^3)}$. Thus, it suffices to discuss the $H^s$ seminorm of $S_\O \phi$. 

For a fixed $v \in \R^3 \setminus \{ 0 \}$, let
\begin{align*}
T_1(v) :=& \{ (x, y) \in \O \times \O \mid \tau_-(x, v) \leq \tau_-(y, v) \},\\
T_2(v) :=& \{ (x, y) \in \O \times \O \mid \tau_-(y, v) \leq \tau_-(x, v) \}.
\end{align*}
Then, we have
\begin{align*}
&|S_\O \phi(\cdot, v)|_{H^s(\O)}^2\\
=& \int_\O \int_\O \frac{|S_\O \phi(x, v) - S_\O \phi(y, v)|^2}{|x - y|^{3 + 2s}}\,dxdy\\
\leq& C \int_{T_1(v)} \frac{1}{|x - y|^{3 + 2s}} \left( \int_0^{\tau_-(x, v)} e^{-\nu(|v|)t} (\phi(x - tv, v) - \phi(y - tv, v))\,dt \right)^2\,dxdy\\
&+C \int_{T_1(v)} \frac{1}{|x - y|^{3 + 2s}} \left( \int_{\tau_-(x, v)}^{\tau_-(y, v)} e^{-\nu(|v|)t} \phi(y - tv, v)\,dt \right)^2\,dxdy\\
&+ C \int_{T_2(v)} \frac{1}{|x - y|^{3 + 2s}} \left( \int_0^{\tau_-(y, v)} e^{-\nu(|v|)t} (\phi(y - tv, v) - \phi(x - tv, v))\,dt \right)^2\,dxdy\\
&+C \int_{T_2(v)} \frac{1}{|x - y|^{3 + 2s}} \left( \int_{\tau_-(y, v)}^{\tau_-(x, v)} e^{-\nu(|v|)t} \phi(x - tv, v)\,dt \right)^2\,dxdy.
\end{align*}
By switching dummy variables $x$ and $y$ in $T_2(v)$, we notice that
\begin{align*}
&|S_\O \phi(\cdot, v)|_{H^s(\O)}^2\\
\leq& C \int_{T_1(v)} \frac{1}{|x - y|^{3 + 2s}} \left( \int_0^{\tau_-(x, v)} e^{-\nu(|v|)t} (\phi(x - tv, v) - \phi(y - tv, v))\,dt \right)^2\,dxdy\\
&+C \int_{T_1(v)} \frac{1}{|x - y|^{3 + 2s}} \left( \int_{\tau_-(x, v)}^{\tau_-(y, v)} e^{-\nu(|v|)t} \phi(y - tv, v)\,dt \right)^2\,dxdy.
\end{align*}
Thus, we obtain
\begin{align*}
&\| |S_\O \phi |_{H^s(\O)} \|_{L^2_{\a, \gamma/2}(\R^3)}^2\\
\leq& C \int_{\R^3} \int_{T_1(v)} \frac{1}{|x - y|^{3 + 2s}}\\
&\quad \times \left( \int_0^{\tau_-(x, v)} e^{-\nu(|v|)t} (\phi(x - tv, v) - \phi(y - tv, v))\,dt \right)^2\,dxdy\,e^{2\a|v|^2} (1 + |v|)^\gamma\,dv\\
&+C \int_{\R^3} \int_{T_1(v)} \frac{1}{|x - y|^{3 + 2s}} \left( \int_{\tau_-(x, v)}^{\tau_-(y, v)} e^{-\nu(|v|)t} \phi(y - tv, v)\,dt \right)^2\,dxdy\\
&\quad \times e^{2\a|v|^2} (1 + |v|)^\gamma\,dv.
\end{align*}

For the first term on the right hand side, since the boundary $\partial \O$ is $C^1$, there exists an $H^s$ extension $\overline{\phi}$ of $\phi$ such that $\| \overline{\phi}(\cdot, v) \|_{H^s(\R^3)} \leq C \| \phi (\cdot, v) \|_{H^s(\O)}$ for some positive constant $C$ \cite{DPV}. We remark that $C$ is independent of $\phi$ and $v$. Thus, by the Cauchy-Schwarz inequality and Proposition \ref{prop:S_decay}, we have
\begin{align*}
&\int_{T_1(v)} \frac{1}{|x - y|^{3 + 2s}} \left( \int_0^{\tau_-(x, v)} e^{-\nu(|v|)t} (\phi(x - tv, v) - \phi(y - tv, v))\,dt \right)^2\,dxdy\\
\leq& C (1 + |v|)^{-1} \int_{T_1(v)} \frac{1}{|x - y|^{3 + 2s}}\\
&\times \int_0^{\tau_-(x, v)} e^{-\nu(|v|)t} |\phi(x - tv, v) - \phi(y - tv, v))|^2\,dt\,dxdy\\
\leq& C (1 + |v|)^{-1} \int_{\R^3} \int_{\R^3} \frac{1}{|x - y|^{3 + 2s}}\\
&\times \int_0^\infty e^{-\nu(|v|)t} |\overline{\phi}(x - tv, v) - \overline{\phi}(y - tv, v))|^2\,dt\,dxdy.
\end{align*}
By changing variables $x' = x - tv$ and $y' = y - tv$, and by the Fubini theorem, we get
\begin{align*}
&\int_{\R^3} \int_{\R^3} \frac{1}{|x - y|^{3 + 2s}} \int_0^\infty e^{-\nu(|v|)t} |\overline{\phi}(x - tv, v) - \overline{\phi}(y - tv, v))|^2\,dt\,dxdy\\
=& \int_0^\infty e^{-\nu(|v|)t}\,dt \int_{\R^3} \int_{\R^3} \frac{|\overline{\phi}(x', v) - \overline{\phi}(y', v))|^2}{|x' - y'|^{3 + 2s}}\,dx'dy'\\
\leq& C (1 + |v|)^{-\gamma} \| \ol{\phi}(\cdot, v) \|_{H^s(\R^3)}^2\\
\leq& C (1 + |v|)^{-\gamma} \| \phi(\cdot, v) \|_{H^s(\O)}^2.
\end{align*}
Hence, we have
\begin{align*}
&\int_{\R^3} \int_{T_1(v)} \frac{1}{|x - y|^{3 + 2s}}\\
&\quad \times \left( \int_0^{\tau_-(x, v)} e^{-\nu(|v|)t} (\phi(x - tv, v) - \phi(y - tv, v))\,dt \right)^2\,dxdy\,e^{2\a|v|^2} (1 + |v|)^\gamma\,dv\\
\leq& C \int_{\R^3} \| \phi(\cdot, v) \|_{H^s(\O)}^2 e^{2\a|v|^2} (1 + |v|)^{-1}\,dv\\
\leq& C \| \phi \|_{L^2_{\a, -\gamma/2}(\R^3; H^s(\O))}^2.
\end{align*}

For the second term on the right hand side, we have
\begin{align*}
&\left( \int_{\tau_-(x, v)}^{\tau_-(y, v)} e^{-\nu(|v|)t} \phi(y - tv, v)\,dt \right)^2\\
\leq& \| \phi \|_{L^\infty_{\a, \b - \gamma}(\O \times \R^3)}^2 e^{-2\a|v|^2} (1 + |v|)^{2(\gamma - \b)} \left( \int_{\tau_-(x, v)}^{\tau_-(y, v)} e^{-\nu(|v|)t} \,dt \right)^2.
\end{align*}
Also, by Proposition \ref{prop:S_decay} and Proposition \ref{prop:EPT}, we get
\begin{align*}
\left( \int_{\tau_-(x, v)}^{\tau_-(y, v)} e^{-\nu(|v|)t} \,dt \right)^2 \leq& C (1 + |v|)^{-(1 - s)} |\tau_-(x, v) - \tau_-(x, v)|^{1 + s}\\
\leq& C (1 + |v|)^{-(1 - s)} \frac{|x - y|^{1 + s}}{N(x, y, v)^{1 + s} |v|^{1 + s}}.
\end{align*}
Thus, we obtain
\begin{align*}
&\int_{T_1(v)} \frac{1}{|x - y|^{3 + 2s}} \left( \int_{\tau_-(x, v)}^{\tau_-(y, v)} e^{-\nu(|v|)t} \phi(y - tv, v)\,dt \right)^2\,dxdy\\
\leq& C\| \phi \|_{L^\infty_{\a, \b - \gamma}(\O \times \R^3)}^2 e^{-2\a|v|^2} \frac{ (1 + |v|)^{2(\gamma - \b) - 1 + s}}{|v|^{1 + s}} \int_{T_1(v)} \frac{1}{|x - y|^{2 + s} N(x, y, v)^{1 + s}}\,dxdy.
\end{align*}
Furthermore, we observe that
\begin{align*}
&\int_{T_1(v)} \frac{1}{|x - y|^{2 + s} N(x, y, v)^{1 + s}}\,dxdy\\
\leq& \int_\O \frac{1}{N(x, v)^{1 + s}} \left( \int_\O \frac{1}{|x - y|^{2 + s}}\,dy \right)\,dx + \int_\O \frac{1}{N(y, v)^{1 + s}} \left( \int_\O \frac{1}{|x - y|^{2 + s}}\,dx \right)\,dy\\
\leq& C \int_\O \frac{1}{N(x, v)^{1 + s}}\,dx,
\end{align*}
and hence, we have
\begin{align*}
&\int_{\R^3} \int_{T_1(v)} \frac{1}{|x - y|^{3 + 2s}} \left( \int_{\tau_-(x, v)}^{\tau_-(y, v)} e^{-\nu(|v|)t} \phi(y - tv, v)\,dt \right)^2\,dxdy\\
&\quad \times e^{2\a|v|^2} (1 + |v|)^\gamma\,dv\\
\leq& C\| \phi \|_{L^\infty_{\a, \b - \gamma}(\O \times \R^3)}^2 \int_{\R^3} \frac{ (1 + |v|)^{2(\gamma - \b) - 1 + s + \gamma}}{|v|^{1  + s}} \int_\O \frac{1}{N(x, v)^{1 + s}}\,dx\,dv.
\end{align*}
A similar calculation to that in the proof of Lemma \ref{lem:Hs_Jf0} yields
\begin{align*}
&\int_{\R^3} \frac{ (1 + |v|)^{2(\gamma - \b) - 1 + s + \gamma}}{|v|^{1 + s}} \int_\O \frac{1}{N(x, v)^{1 + s}}\,dx\,dv\\
\leq& C \int_{\Gamma^-} \frac{ (1 + |v|)^{2(\gamma - \b) - 1 + s + \gamma}}{|v|^{1 + s}} \frac{1}{N(z, v)^s}\,dv d\sigma_z\\
=& C \left( \int_0^1 \frac{1}{t^s}\,dt \right) \left( \int_0^\infty (1 + \rho)^{3\gamma - 2\b} \left( \frac{\rho}{1 + \rho} \right)^{1 - s}\,d\rho \right).
\end{align*}
The first integral with respect to $t$ converges since $s < 1$. Also, since $3\gamma - 2\b < 2 \gamma -3 \leq -1$ for $\b > (3 + \gamma)/2$ and $\rho/(1 + \rho) \leq 1$, the second integral with respect to $\rho$ converges. Therefore, we have
\begin{align*}
&\int_{\R^3} \int_{T_1(v)} \frac{1}{|x - y|^{3 + 2s}} \left( \int_{\tau_-(x, v)}^{\tau_-(y, v)} e^{-\nu(|v|)t} \phi(y - tv, v)\,dt \right)^2\,dxdy\\
\leq& C\| \phi \|_{L^\infty_{\a, \b - \gamma}(\O \times \R^3)}^2.
\end{align*}

Combining the above two estimates, we conclude that
\begin{align*}
\| |S_\O \phi |_{H^s(\O)} \|_{L^2_{\a, \gamma/2}(\R^3)}^2 \leq& C \left( \| \phi \|_{L^2_{\a, -\gamma/2}(\R^3; H^s(\O))}^2 + \| \phi \|_{L^\infty_{\a, \b - \gamma}(\O \times \R^3)}^2 \right)\\
\leq& C \| \phi \|_{Y^s_{\a, \b, \g}}^2.
\end{align*}
This completes the proof.
\end{proof}

Let $f_0 \in \cB^{s_1}_{\a, \b}(\Gamma^-)$ and $\phi \in Y^{s_2}_{\a, \b, \g}$. So far, by Lemma \ref{lem:smoothing_x}, Lemma \ref{lem:Hs_Jf0} and Lemma \ref{lem:Hs_S}, we obtain $f \in X^s_{\a, \b, \g}$ with $0 < s < \min \{ s_1, s_2, s_{2, \gamma} \}$, where $s_{2, \g}$ is the constant in \eqref{def:s2g}. To improve the regularity, we introduce the following lemma.

\begin{lemma} \label{lem:smoothing_x2}
Let $\O$ be a bounded convex domain with $C^1$ boundary, $0 \leq \a < 1/2$, $\b > (3 + \gamma)/2$, $-3 < \gamma \leq 1$ and $0 \leq s < 1/2$. Then, we have
\[
\| S_\O K f \|_{L^2_{\a ,\gamma/2}(\R^3; H^{s + s_{2, \gamma}}(\O))} \leq C \| f \|_{X^s_{\a, \b, \gamma}}
\] 
for all $f \in X^s_{\a, \b, \gamma}$, where $s_{2, \gamma}$ is the constant in \eqref{def:s2g}.
\end{lemma}

\begin{proof}
Via the transformation \eqref{eq:transform_a_to_0}, it suffices to prove
\[
\| S_\O K_\a f \|_{L^2_{0 ,\gamma/2}(\R^3; H^{s + s_{2, \gamma}}(\O))} \leq C \| f \|_{X^s_{0, \b, \g}}
\] 
for all $f \in X^s_{0, \b, \g}$.

As in the proof of Corollary \ref{cor:smoothing_x}, for a function $f \in X^s_{0, \b, \g}$, let $\tilde{f}$ be its zero extension. We claim that it belongs to $L^2_{0 ,\gamma/2}(\R^3; H^s(\R^3)) \cap L^\infty_{0, \b}(\R^3 \times \R^3)$. Since $\| \tilde{f} \|_{L^\infty_{0, \b}(\R^3 \times \R^3)} = \| f \|_{L^\infty_{0, \b}(\O \times \R^3)}$, we only need to check that $\tilde{f} \in L^2_{0 ,\gamma/2}(\R^3; H^s(\R^3))$.

We have $\| \tilde{f} \|_{L^2_{0 ,\gamma/2}(\R^3; L^2(\R^3))} = \| f \|_{L^2_{0 ,\gamma/2}(\R^3; L^2(\O))}$ and
\begin{align*}
&\int_{\R^3} \left( \int_{\R^3} \int_{\R^3} \frac{|\tilde{f}(x, v) - \tilde{f}(y, v)|^2}{|x - y|^{3 + 2s}}\,dxdy \right)(1 + |v|)^\gamma\,dv\\
=& \int_{\R^3} \left( \int_{\O} \int_{\O} \frac{|f(x, v) - f(y, v)|^2}{|x - y|^{3 + 2s}}\,dxdy \right)(1 + |v|)^\gamma\,dv\\
&+ \int_{\R^3} \left( \int_{\O} \left(\int_{\R^3 \setminus \O} \frac{1}{|x - y|^{3 + 2s}}\,dy \right) |f(x, v)|^2\,dx \right)(1 + |v|)^\gamma\,dv\\ 
&+ \int_{\R^3} \left( \int_{\O} \left( \int_{\R^3 \setminus \O} \frac{1}{|x - y|^{3 + 2s}}\,dx \right) |f(y, v)|^2\,dy \right)(1 + |v|)^\gamma\,dv.
\end{align*}
Since 
\[
\int_{\R^3 \setminus \O} \frac{1}{|x - y|^{3 + 2s}}\,dy \leq \int_{\{|x - y| \geq d_x \}} \frac{1}{|x - y|^{3 + 2s}}\,dy \leq C d_x^{-2s},
\]
where $d_x = \dist(x, \p \O)$, we have
\begin{align*}
&\int_{\R^3} \left( \int_{\R^3} \int_{\R^3} \frac{|\tilde{f}(x, v) - \tilde{f}(y, v)|^2}{|x - y|^{3 + 2s}}\,dxdy \right)(1 + |v|)^\gamma\,dv\\
\leq& C \| f \|_{L^2_{0, \gamma/2}(\R^3; H^s(\O))}^2 + C \int_{\R^3} \left( \int_{\O} d_{x}^{-2s} |f(x, v)|^2\,dx \right)(1 + |v|)^\gamma\,dv\\
\leq& C \| f \|_{L^2_{0, \gamma/2}(\R^3; H^s(\O))}^2 + C \| f \|_{L^\infty_{0, \b}(\O \times \R^3)}^2 D_s
\end{align*}
with
\[
D_s := \int_{\O} d_{x}^{-2s} \,dx.
\]

In what follows, we prove that the integral $D_s$ converges if $0 \leq s < 1/2$. Let $\{ U_n \}_{n = 1}^N$ be an open covering of $\p \O$, and by taking them sufficiently small, we may assume the existence of the $C^1$ diffeomorphism $\psi_n: U_n \to B_1$ for each $n$ such that $\psi_n(U_n \cap \p \O) = B_1 \cap \R^3_0$ and $\psi_n(U_n \cap \O) = B_1 \cap \R^3_+$, where
\[
\R^3_0 := \{ (x_1, x_2, x_3) \in \R^3 \mid x_3 = 0 \}, \quad \R^3_+ := \{ (x_1, x_2, x_3) \in \R^3 \mid x_3 > 0 \},
\]
and $B_1$ denotes the unit ball in $\R^3$. Also, let $U_0 := \O \setminus \cup_{n = 1}^N U_n$. Since $\ol{U_0} \subset \O$, we have
\[
\int_{U_0} d_x^{-2s}\,dx \leq C.
\]
On each $U_n$, we have
\[
\int_{U_n \cap \O} d_x^{-2s}\,dx = \int_{B_1 \cap \R^3_+} d_{\psi_n^{-1}(y)}^{-2s} |\det \nabla_y \psi^{-1}_n(y)|\,dy,
\]
whose convergence is equivalent to that of  
\[
\int_0^1 {y_3}^{-2s}\,dy_3.
\]
Since the above integral converges for $0 \leq s < 1/2$, we obtain
\begin{align*}
\int_{\R^3} \left( \int_{\R^3} \int_{\R^3} \frac{|\tilde{f}(x, v) - \tilde{f}(y, v)|^2}{|x - y|^{3 + 2s}}\,dxdy \right)(1 + |v|)^\gamma\,dv \leq C \| f \|_{X^s_{0, \b, \gamma}}^2.
\end{align*}
Thus, the claim is proved.

The conclusion follows from Lemma \ref{lem:smoothing_x}. This completes the proof.
\end{proof}

We apply Lemma \ref{lem:smoothing_x2} to the integral equation \eqref{IE}. The bootstrap argument yields $f \in X^s_{\a, \b, \g}$ with $0 < s < \min \{s_1, s_2, 1/2 + s_{2, \gamma} \}$. Lemma \ref{lem:regularity} is proved since $s_\gamma = 1/2 + s_{2, \gamma}$.

\section{Bilinear estimate} \label{sec:non}

In this section, we give a proof of Lemma \ref{lem:bilinear}. Recall that
\[
\Gamma(h_1,h_2) = \pi^{-\frac{3}{4}} \left( \Gamma_{\gain}(h_1, h_2) - \Gamma_{\loss}(h_1, h_2) \right),
\]
where
\begin{align*}
\Gamma_{\gain}(h_1, h_2) :=& \int_{\mathbb{R}^{3}}\int_{0}^{2\pi}\int_{0}^{\frac{\pi}{2}}e^{-\frac{|v_*|^2}{2}} h_1(v') h_2(v_*') B(|v-v_*|, \theta)\,d\theta \, d\phi\,dv_*,\\
\Gamma_{\loss}(h_1, h_2) :=& \int_{\mathbb{R}^{3}}\int_{0}^{2\pi}\int_{0}^{\frac{\pi}{2}}e^{-\frac{|v_*|^2}{2}} h_1(v) h_2(v_*) B(|v-v_*|, \theta)\,d\theta \, d\phi\,dv_*.
\end{align*}
The variables $v'$ and $v_*'$ are defined by \eqref{collision_coordinate}. Here, we used the identity: 
\begin{equation} \label{energy_conservation}
|v|^2 + |v_*|^2 = |v'|^2 + |v_*'|^2.
\end{equation}
We call $\Gamma_\gain$ and $\Gamma_\loss$ the gain term and the loss term respectively.

\begin{lemma} \label{lem:Linfty_non}
Let $\O$ be a bounded convex domain with $C^1$ boundary, $0 \leq \a < 1/2$, $\b \geq 0$ and $-3 < \gamma \leq 1$. For $h_1, h_2 \in L^\infty_{\a, \b}(\O \times \R^3)$, we have $\Gamma(h_1, h_2) \in L^\infty_{\a, \b - \gamma}(\O \times \R^3)$. Moreover, we have
\[
\| \Gamma(h_1, h_2) \|_{L^\infty_{\a, \b - \gamma}(\O \times \R^3)} \leq C \| h_1 \|_{L^\infty_{\a, \b}(\O \times \R^3)} \| h_2 \|_{L^\infty_{\a, \b}(\O \times \R^3)}
\]
for all $h_1, h_2 \in L^\infty_{\a, \b}(\O \times \R^3)$.
\end{lemma}

\begin{proof}
For the gain term, by \eqref{assumption_B1}, we have
\begin{align*}
|\Gamma_{\gain}(h_1, h_2)| \leq& C \| h_1 \|_{L^\infty_{\a, \b}(\O \times \R^3)} \| h_2 \|_{L^\infty_{\a, \b}(\O \times \R^3)} \int_{\mathbb{R}^{3}}\int_{0}^{2\pi}\int_{0}^{\frac{\pi}{2}} e^{-\frac{|v_*|^2}{2} -\a(|v'|^2 + |v_*'|^2)}\\ 
&\times (1 + |v'|)^{-\b} (1 + |v_*'|)^{-\b} |v - v_*|^\gamma \sin \theta \cos \theta \,d\theta \, d\phi\,dv_*.
\end{align*}
We notice that the estimate $(1 + |v'|^2)^{-\b/2} \leq C (1 + |v'|)^{-\b} \leq C (1 + |v'|^2)^{-\b/2}$ holds for all $v' \in \R^3$. We employ the identity \eqref{energy_conservation} to obtain
\begin{align*}
&e^{-\frac{|v_*|^2}{2} -\a(|v'|^2 + |v_*'|^2)} (1 + |v'|)^{-\b} (1 + |v_*'|)^{-\b}\\
\leq& C e^{-\frac{|v_*|^2}{2} -\a(|v'|^2 + |v_*'|^2)} (1 + |v'|^2 + |v_*'|^2)^{-\frac{\b}{2}}\\
=& C e^{-\frac{|v_*|^2}{2} -\a(|v|^2 + |v_*|^2)} (1 + |v|^2 + |v_*|^2)^{-\frac{\b}{2}}\\
\leq& C e^{-\a |v|^2} e^{-\left( \a + \frac{1}{2} \right)|v_*|^2} (1 + |v|)^{-\b}.
\end{align*}
Thus, we have
\begin{align*}
|\Gamma_{\gain}(h_1, h_2)| \leq& C \| h_1 \|_{L^\infty_{\a, \b}(\O \times \R^3)} \| h_2 \|_{L^\infty_{\a, \b}(\O \times \R^3)} e^{-\a |v|^2} (1 + |v|)^{-\b}\\
&\times \int_{\mathbb{R}^{3}}\int_{0}^{2\pi}\int_{0}^{\frac{\pi}{2}} e^{-\left( \a + \frac{1}{2} \right)|v_*|^2} |v - v_*|^\gamma \sin \theta \cos \theta \,d\theta \, d\phi\,dv_*\\
\leq& C \| h_1 \|_{L^\infty_{\a, \b}(\O \times \R^3)} \| h_2 \|_{L^\infty_{\a, \b}(\O \times \R^3)} e^{-\a |v|^2} (1 + |v|)^{-\b}\\
&\times \int_{\mathbb{R}^{3}} e^{-\left( \a + \frac{1}{2} \right)|v_*|^2} |v - v_*|^\gamma\,dv_*.
\end{align*}

For $0 \leq \gamma \leq 1$, we have $|v - v_*|^\gamma \leq C_\gamma(|v|^\gamma + |v_*|^\gamma)$ and hence we have
\begin{align*}
\int_{\mathbb{R}^{3}} e^{-\left( \a + \frac{1}{2} \right)|v_*|^2} |v - v_*|^\gamma\,dv_* \leq& C |v|^\gamma \int_{\mathbb{R}^{3}} e^{-\left( \a + \frac{1}{2} \right)|v_*|^2} \,dv_*\\
& + C \int_{\mathbb{R}^{3}} e^{-\left( \a + \frac{1}{2} \right)|v_*|^2} |v_*|^\gamma \,dv_*\\
\leq& C (1 + |v|)^\gamma.
\end{align*}

For $-3 < \gamma < 0$, we decompose the domain of integration as follows:
\begin{align*}
\int_{\mathbb{R}^{3}} e^{-\left( \a + \frac{1}{2} \right)|v_*|^2} |v - v_*|^\gamma\,dv_* =& \int_{\{ |v_*| \leq |v - v_*| \}} e^{-\left( \a + \frac{1}{2} \right)|v_*|^2} |v - v_*|^\gamma\,dv_*\\
&+ \int_{\{ |v_*| > |v - v_*| \}} e^{-\left( \a + \frac{1}{2} \right)|v_*|^2} |v - v_*|^\gamma\,dv_*.
\end{align*}
In the first region, we have
\begin{align*}
\int_{\{ |v_*| \leq |v - v_*| \}} e^{-\left( \a + \frac{1}{2} \right)|v_*|^2} |v - v_*|^\gamma\,dv_* \leq& \int_{\{ |v_*| \leq |v - v_*| \}} e^{-\left( \a + \frac{1}{2} \right)|v_*|^2} |v_*|^\gamma\,dv_*\\
\leq& \int_{\R^3} e^{-\left( \a + \frac{1}{2} \right)|v_*|^2} |v_*|^\gamma\,dv_* \leq C
\end{align*}
and, since $|v - v_*| \geq |v|/2$ when $|v_*| \leq |v - v_*|$, we have
\[
\int_{\{ |v_*| \leq |v - v_*| \}} e^{-\left( \a + \frac{1}{2} \right)|v_*|^2} |v - v_*|^\gamma\,dv_* \leq C |v|^\gamma \int_{\R^3} e^{-\left( \a + \frac{1}{2} \right)|v_*|^2} \,dv_* \leq C |v|^\gamma.
\]
Thus, we get
\[
\int_{\{ |v_*| \leq |v - v_*| \}} e^{-\left( \a + \frac{1}{2} \right)|v_*|^2} |v - v_*|^\gamma\,dv_* \leq C (1 + |v|)^\gamma.
\]
In the second region, since $|v_*| > |v|/2$ when $|v_*| > |v - v_*|$, we have
\begin{align*}
&\int_{\{ |v_*| > |v - v_*| \}} e^{-\left( \a + \frac{1}{2} \right)|v_*|^2} |v - v_*|^\gamma\,dv_*\\ 
\leq& e^{-\frac{1}{8} \left( \a + \frac{1}{2} \right) |v|^2} \int_{\{ |v_*| > |v - v_*| \}} e^{- \frac{1}{2} \left( \a + \frac{1}{2} \right)|v - v_*|^2} |v - v_*|^\gamma\,dv_*\\
\leq& C(1 + |v|)^\gamma.
\end{align*}
Thus, we obtain
\begin{equation} \label{est:convolution}
\int_{\mathbb{R}^{3}} e^{-\left( \a + \frac{1}{2} \right)|v_*|^2} |v - v_*|^\gamma\,dv_* \leq C (1 + |v|)^\gamma
\end{equation}
and
\[
|\Gamma_{\gain}(h_1, h_2)| \leq C \| h_1 \|_{L^\infty_{\a, \b}(\O \times \R^3)} \| h_2 \|_{L^\infty_{\a, \b}(\O \times \R^3)} e^{-\a |v|^2} (1 + |v|)^{\gamma - \b}.
\]

We can give the same estimate for the loss term: 
\begin{align*}
|\Gamma_{\loss}(h_1, h_2)| \leq& C \| h_1 \|_{L^\infty_{\a, \b}(\O \times \R^3)} \| h_2 \|_{L^\infty_{\a, \b}(\O \times \R^3)} e^{-\a |v|^2} (1 + |v|)^{-\b}\\
&\times \int_{\mathbb{R}^{3}}\int_{0}^{2\pi}\int_{0}^{\frac{\pi}{2}} e^{-\left( \a + \frac{1}{2} \right)|v_*|^2} (1 + |v_*|)^{-\b} |v - v_*|^\gamma \sin \theta \cos \theta \,d\theta \, d\phi\,dv_*\\
\leq& C \| h_1 \|_{L^\infty_{\a, \b}(\O \times \R^3)} \| h_2 \|_{L^\infty_{\a, \b}(\O \times \R^3)} e^{-\a |v|^2} (1 + |v|)^{-\b}\\
&\times \int_{\mathbb{R}^{3}} e^{-\left( \a + \frac{1}{2} \right)|v_*|^2} |v - v_*|^\gamma\,dv_*\\
\leq& C \| h_1 \|_{L^\infty_{\a, \b}(\O \times \R^3)} \| h_2 \|_{L^\infty_{\a, \b}(\O \times \R^3)} e^{-\a |v|^2} (1 + |v|)^{\gamma-\b}.
\end{align*}
This completes the proof.
\end{proof}

\begin{lemma} \label{lem:Hs_non}
Let $\O$ be a bounded convex domain with $C^1$ boundary, $0 \leq \a < 1/2$, $\b > (\gamma + 3)/2$ and $-3 < \gamma \leq 1$. For $h_1, h_2 \in X^s_{\a, \b, \g}$, we have $\Gamma(h_1, h_2) \in L^2_{\a, -\gamma/2}(\R^3; H^s(\O))$. Moreover, we have
\begin{align*}
&\| \Gamma(h_1, h_2) \|_{L^2_{\a, -\gamma/2}(\R^3; H^s(\O))}\\ 
\leq& C \left( \| h_1 \|_{L^2_{\a, \gamma/2}(\R^3; H^s(\O))} \| h_2 \|_{L^\infty_{\a, \b}(\O \times \R^3)} + \| h_1 \|_{L^\infty_{\a, \b}(\O \times \R^3)} \| h_2 \|_{L^2_{\a, \gamma/2}(\R^3; H^s(\O))} \right)
\end{align*}
for all $h_1, h_2 \in X^s_{\a, \b, \g}$.
\end{lemma}

\begin{proof}
We first treat the loss term. Let $x, y \in \O$ and $v \in \R^3$. We decompose the loss term into two parts:
\[
\Gamma_\loss(h_1, h_2)(x, v) - \Gamma_\loss(h_1, h_2)(y, v) = L_1(x, y, v) + L_2(x, y, v),
\]
where
\begin{align*}
L_1(x, y, v) :=& (h_1(x, v) - h_1(y, v)) \int_{\R^3} \int_0^{2\pi} \int_0^{\frac{\pi}{2}} e^{-\frac{|v_*|^2}{2}} h_2(x, v_*) B(|v - v_*|, \theta)\,d\theta d\phi dv_*,\\
L_2(x, y, v) :=& h_1(y, v) \int_{\R^3} \int_0^{2\pi} \int_0^{\frac{\pi}{2}} e^{-\frac{|v_*|^2}{2}} (h_2(x, v_*) - h_2(y, v_*)) B(|v - v_*|, \theta)\,d\theta d\phi dv_*.
\end{align*}

For the $L_1$ term, since $\b > 0$, the estimate \eqref{est:convolution} implies that
\begin{align*}
&\left| \int_{\R^3} \int_0^{2\pi} \int_0^{\frac{\pi}{2}} e^{-\frac{|v_*|^2}{2}} h_2(x, v_*) B(|v - v_*|, \theta)\,d\theta d\phi dv_* \right|\\
\leq& C \| h_2 \|_{L^\infty_{\a, \b}(\O \times \R^3)} \int_{\R^3} e^{- \left(\frac{1}{2} + \a \right) |v_*|^2} |v - v_*|^\gamma (1 + |v_*|)^{-\b} dv_*\\
\leq& C \| h_2 \|_{L^\infty_{\a, \b}(\O \times \R^3)} (1 + |v|)^\gamma.
\end{align*}
Hence, we have
\begin{align*}
&\int_{\R^3} \left( \int_{\O} \int_{\O} \frac{|L_1(x, y, v)|^2}{|x - y|^{3 + 2s}}\,dxdy \right) e^{2\a|v|^2} (1 + |v|)^{-\gamma}\,dv\\
\leq& C \| h_2 \|_{L^\infty_{\a, \b}(\O \times \R^3)}^2 \int_{\R^3} \left( \int_{\O} \int_{\O} \frac{|h_1(x, v) - h_1(y, v)|^2}{|x - y|^{3 + 2s}}\,dxdy \right) e^{2\a|v|^2} (1 + |v|)^{\gamma}\,dv\\
=& C \| h_1 \|_{L^2_{\a, \gamma/2}(\R^3; H^s(\O))}^2 \| h_2 \|_{L^\infty_{\a, \b}(\O \times \R^3)}^2.
\end{align*}

We discuss an estimate on the $L_2$ term in two cases. First, for $0 \leq \gamma \leq 1$, we apply the Cauchy-Schwarz inequality and the estimate \eqref{est:convolution} to obtain 
\begin{align*}
&\left| \int_{\R^3} \int_0^{2\pi} \int_0^{\frac{\pi}{2}} e^{-\frac{|v_*|^2}{2}} (h_2(x, v_*) - h_2(y, v_*)) B(|v - v_*|, \theta)\,d\theta d\phi dv_* \right|^2\\
\leq& C \left( \int_{\R^3} e^{- (2\a + 1) |v_*|^2} |v - v_*|^{\gamma} \,dv_* \right) \left( \int_{\R^3} |h_2(x, v_*) - h_2(y, v_*)|^2 e^{2\a|v_*|^2} |v - v|^\gamma\,dv_* \right)\\
\leq& C(1 + |v|)^\gamma \left( \int_{\R^3} |h_2(x, v_*) - h_2(y, v_*)|^2 e^{2\a|v_*|^2} |v - v|^\gamma \,dv_* \right).
\end{align*}
Hence we have
\begin{align*}
&\int_{\R^3} \left( \int_{\O} \int_{\O} \frac{|L_2(x, y, v)|^2}{|x - y|^{3 + 2s}}\,dxdy \right) e^{2\a|v|^2} (1 + |v|)^{-\gamma}\,dv\\
\leq& C \| h_1 \|_{L^\infty_{\a, \b}(\O \times \R^3)}^2 \int_{\R^3} \left(\int_{\R^3} |v - v_*|^\gamma (1 + |v|)^{-2\b}\,dv \right) \\
&\times \left(\int_{\O} \int_{\O} \frac{|h_2(x, v_*) - h_2(y, v_*)|^2}{|x - y|^{3 + 2s}}\,dxdy \right) e^{2\a|v_*|^2}\,dv_*.
\end{align*}
Since $-2\b \leq \gamma - 2\b < -3$, we get
\begin{align*}
\int_{\R^3} |v - v_*|^\gamma (1 + |v|)^{-2\b}\,dv \leq& C \int_{\R^3} (1 + |v|)^{\gamma - 2\b}\,dv + C |v_*|^\gamma \int_{\R^3} (1 + |v|)^{-2\b}\,dv\\
\leq& C (1 + |v_*|)^\gamma.
\end{align*}
Thus, we obtain
\begin{equation} \label{est:L2_Hs}
\begin{split}
&\int_{\R^3} \left( \int_{\O} \int_{\O} \frac{|L_2(x, y, v)|^2}{|x - y|^{3 + 2s}}\,dxdy \right) e^{2\a|v|^2} (1 + |v|)^{-\gamma}\,dv\\ 
\leq& C \| h_1 \|_{L^\infty_{\a, \b}(\O \times \R^3)}^2 \| h_2 \|_{L^2_{\a, \gamma/2}(\R^3; H^s(\O))}^2.
\end{split}
\end{equation}

Next, for $-3 < \gamma < 0$, we apply the Cauchy-Schwarz inequality to the $L_2$ term again to obtain
\begin{align*}
&\left| \int_{\R^3} \int_0^{2\pi} \int_0^{\frac{\pi}{2}} e^{-\frac{|v_*|^2}{2}} (h_2(x, v_*) - h_2(y, v_*)) B(|v - v_*|, \theta)\,d\theta d\phi dv_* \right|^2\\
\leq& C \left( \int_{\R^3} e^{- (2\a + 1) |v_*|^2} |v - v_*|^{\gamma} (1 + |v_*|)^{-\gamma} \,dv_* \right)\\
&\times \left( \int_{\R^3} |h_2(x, v_*) - h_2(y, v_*)|^2 e^{2\a|v_*|^2} |v - v|^\gamma (1 + |v_*|)^\gamma\,dv_* \right).
\end{align*}
Since
\[
(1 + |v|)^{|\gamma|} \leq C_\gamma \left\{ (1 + |v_*|)^{|\gamma|} + |v - v_*|^{|\gamma|} \right\},
\]
we have
\begin{align*}
&(1 + |v|)^{|\gamma|} \int_{\R^3} e^{- (2\a + 1) |v_*|^2} |v - v_*|^{\gamma} (1 + |v_*|)^{-\gamma} \,dv_*\\
\leq& C \int_{\R^3} e^{- (2\a + 1) |v_*|^2} (1 + |v|)^{-\gamma} \,dv_* + C \int_{\R^3} e^{- (2\a + 1) |v_*|^2} |v - v_*|^{\gamma} (1 + |v_*|)^{-2\gamma}\,dv_*\\
\leq& C \int_{\R^3} e^{- 2\a |v_*|^2} \,dv_* + C \int_{\R^3} e^{- 2\a |v_*|^2} |v - v_*|^{\gamma} \,dv_*\\
\leq& C,
\end{align*}
or
\[
\int_{\R^3} e^{- (2\a + 1) |v_*|^2} |v - v_*|^{\gamma} (1 + |v_*|)^{-\gamma} \,dv_* \leq C (1 + |v|)^{\gamma}. 
\]
Hence we have
\begin{align*}
&\int_{\R^3} \left( \int_{\O} \int_{\O} \frac{|L_2(x, y, v)|^2}{|x - y|^{3 + 2s}}\,dxdy \right) e^{2\a|v|^2} (1 + |v|)^{-\gamma}\,dv\\
\leq& C \| h_1 \|_{L^\infty_{\a, \b}(\O \times \R^3)}^2 \int_{\R^3} \left(\int_{\R^3} |v - v_*|^\gamma (1 + |v|)^{- 2\b}\,dv \right) \\
&\times \left(\int_{\O} \int_{\O} \frac{|h_2(x, v_*) - h_2(y, v_*)|^2}{|x - y|^{3 + 2s}}\,dxdy \right) e^{2\a|v_*|^2} (1 + |v_*|)^\gamma \,dv_*.
\end{align*}
In the same way as in the proof of Lemma \ref{lem:Linfty_non}, since $-3 < \gamma < 0$ and $\gamma - 2\b < -3$, we have
\begin{align*}
\int_{\{ |v| \leq |v - v_*| \}} |v - v_*|^\gamma (1 + |v|)^{-2\b}\,dv \leq& \int_{\R^3} |v|^\gamma (1 + |v|)^{-2\b}\,dv\\
\leq& C \int_{\{|v| \leq 1\}} |v|^\gamma \,dv + C \int_{\{ |v| > 1 \}} (1 + |v|)^{\gamma - 2\b}\,dv\\
\leq& C
\end{align*}
and
\begin{align*}
\int_{\{ |v| > |v - v_*| \}} |v - v_*|^\gamma (1 + |v|)^{-2\b}\,dv \leq& \int_{\R^3} |v - v_*|^\gamma (1 + |v - v_*|)^{-2\b}\,dv\\
=& \int_{\R^3} |v|^\gamma (1 + |v|)^{-2\b}\,dv\\
\leq& C.
\end{align*}
Therefore, the estimate \eqref{est:L2_Hs} also holds for $-3 < \gamma < 0$.

We proceed to an estimate for the gain term. We decompose the gain term into two parts:
\[
\Gamma_\gain(h_1, h_2)(x, v) - \Gamma_\gain(h_1, h_2)(y, v) = G_1(x, y, v) + G_2(x, y, v),
\]
where
\begin{align*}
G_1(x, y, v) :=& \int_{\R^3} \int_0^{2\pi} \int_0^{\frac{\pi}{2}} e^{-\frac{|v_*|^2}{2}} (h_1(x, v') - h_1(y, v')) h_2(x, v_*') B(|v - v_*|, \theta)\,d\theta d\phi dv_*,\\
G_2(x, y, v) :=& \int_{\R^3} \int_0^{2\pi} \int_0^{\frac{\pi}{2}} e^{-\frac{|v_*|^2}{2}} h_1(y, v') (h_2(x, v_*') - h_2(y, v_*')) B(|v - v_*|, \theta)\,d\theta d\phi dv_*.
\end{align*}
We focus on the case where $0 \leq \gamma \leq 1$. The case where $-3 < \gamma < 0$ can be treated with a slight modification on the polynomial weight as we did for the $L_2$ term. 

For the $G_1$ term, the Cauchy-Schwarz inequality and the estimate \eqref{est:convolution} yield
\begin{align*}
&|G_1(x, y, v)|^2\\ 
\leq& \left( \int_{\R^3} \int_0^{2\pi} \int_0^{\frac{\pi}{2}} |h_1(x, v') - h_1(y, v')|^2 |h_2(x, v_*')|^2 e^{2\a (|v'|^2 + |v_*'|^2)} \right.\\
&\times \left. B(|v - v_*|, \theta) \,d\theta d\phi dv_* \right)\\ 
&\times \left(\int_{\R^3} \int_0^{2\pi} \int_0^{\frac{\pi}{2}} e^{-(2\a + 1)|v_*|^2} B(|v - v_*|, \theta) \,d\theta d\phi dv_* \right) e^{-2\a|v|^2}\\
\leq& C (1 + |v|)^\gamma e^{-2\a|v|^2} \| h_2 \|_{L^\infty_{\a, \b}(\O \times \R^3)}^2\\
&\times \int_{\R^3} \int_0^{2\pi} \int_0^{\frac{\pi}{2}} |h_1(x, v') - h_1(y, v')|^2 e^{2\a |v'|^2} B(|v - v_*|, \theta) (1 + |v_*'|)^{-2\b} \,d\theta d\phi dv_* 
\end{align*}
and hence
\begin{align*}
&\int_{\R^3} \left( \int_\O \int_\O \frac{|G_1(x, y, v)|^2}{|x - y|^{3+2s}}\,dxdy \right) e^{2\a|v|^2} (1 + |v|)^{-\gamma}\,dv\\
\leq& C \| h_2 \|_{L^\infty_{\a, \b}(\O \times \R^3)}^2\\
& \times \int_{\R^3} \int_{\R^3} \int_0^{2\pi} \int_0^{\frac{\pi}{2}} \left( \int_\O \int_\O \frac{|h_1(x, v') - h_1(y, v')|^2}{|x - y|^{3+2s}}\,dxdy \right) e^{2\a |v'|^2}\\ 
&\times B(|v - v_*|, \theta) (1 + |v_*'|)^{-2\b} \,d\theta d\phi dv_*\,dv.
\end{align*}
We recall the definition \eqref{collision_coordinate} of $v'$ and $v_*'$. We notice that $|v' - v_*'| = |v - v_*|$ and
\[
v = v' + ((v_*' - v') \cdot \omega) \omega, \quad v_* = v_*' - ((v_*' - v') \cdot \omega) \omega.
\]
Also, the Jacobian determinant of the above change is unity, namely, $dv' dv_*' = dv dv_*$ \cite{CIP, Glassey}. Thus, the change gives
\begin{align*}
&\int_{\R^3} \left( \int_\O \int_\O \frac{|G_1(x, y, v)|^2}{|x - y|^{3+2s}}\,dxdy \right) e^{2\a|v|^2} (1 + |v|)^{-\gamma}\,dv\\
\leq& C \| h_2 \|_{L^\infty_{\a, \b}(\O \times \R^3)}^2 \int_{\R^3} \left( \int_\O \int_\O \frac{|h_1(x, v') - h_1(y, v')|^2}{|x - y|^{3+2s}}\,dxdy \right) e^{2\a |v'|^2}\\ 
&\times \left( \int_{\R^3} \int_0^{2\pi} \int_0^{\frac{\pi}{2}} |v' - v_*'|^\gamma \sin \theta \cos \theta (1 + |v_*'|)^{-2\b} \,d\theta d\phi dv_*' \right)\,dv'\\
\leq& C \| h_1 \|_{L^2_{\a, \gamma/2}(\R^3; H^s(\O))}^2 \| h_2 \|_{L^\infty_{\a, \b}(\O \times \R^3)}^2.
\end{align*}

In the same fashion, we have
\begin{align*}
&|G_2(x, y, v)|^2\\ 
\leq& C(1 + |v|)^\gamma e^{-2\a |v|^2} \| h_1 \|_{L^\infty_{\a, \b}(\O \times \R^3)}^2\\
&\times \int_{\R^3} \int_0^{2\pi} \int_0^{\frac{\pi}{2}} |h_2(x, v_*') - h_2(y, v_*')|^2 e^{2\a |v_*'|^2} B(|v - v_*|, \theta) (1 + v')^{-2\b} \,d\theta d\phi dv_*.
\end{align*}
and
\begin{align*}
&\int_{\R^3} \left( \int_\O \int_\O \frac{|G_2(x, y, v)|^2}{|x - y|^{3 + 2s}}\,dxdy \right) e^{2\a|v|^2} (1 + |v|)^{-\gamma}\,dv\\
\leq& C \| h_1 \|_{L^\infty_{\a, \b}(\O \times \R^3)}^2 \int_{\R^3} \left( \int_\O \int_\O \frac{|h_2(x, v_*') - h_2(y, v_*')|^2}{|x - y|^{3+2s}}\,dxdy \right)\\
&\times \left( \int_{\R^3} \int_0^{2\pi} \int_0^{\frac{\pi}{2}} |v' - v_*'|^\gamma \sin \theta \cos \theta (1 + |v'|)^{-2\b} \,d\theta d\phi dv' \right) e^{2\a |v_*'|^2}\, dv_*'\\
\leq& C \| h_1 \|_{L^\infty_{\a, \b}(\O \times \R^3)}^2 \| h_2 \|_{L^2_{\a, \gamma/2}(\R^3; H^s(\O))}^2.
\end{align*}

The desired estimate is obtained by summarizing the above estimates. This completes the proof.
\end{proof}

Lemma \ref{lem:bilinear} follows from Lemma \ref{lem:Linfty_non} and Lemma \ref{lem:Hs_non}.

\section*{Acknowledgement}

The author would like to thank Professor Chun-Hsiung Hsia and Professor I-Kun Chen for suggesting this project and for giving him fruitful comments when he prepared this article. He is supported in part by JSPS KAKENHI grant number JP24K00539.

\appendix

\section{Estimates for $k$}

In this appendix, we provide a proof of Proposition \ref{prop:est_k} for readers' convenience. In what follows, we use the following estimate:
\[
\frac{1}{|v - v^*|^{\b_1}} e^{-\delta |v - v^*|^2} \leq \frac{C_{\b_1, \b_2, \delta}}{|v - v^*|^{\b_2}}
\]
for any $0 \leq \b_1 < \b_2$ and $\delta > 0$.

We first give a proof of the estimate \eqref{est:k}. We notice that, by the identity \eqref{identity_V2}, we have
\[
\frac{1}{4} |v - v^*|^2 + |V_1(v, v^*)|^2 \leq \frac{1}{2} (|v|^2 + |v^*|^2).
\]
Thus, for any $0 < \delta < 1$ and $\b \geq 0$, we have
\begin{align*}
&(1 + |v| + |v^*|)^\b |k_1(v, v^*)|\\ 
=& C (1 + |v| + |v^*|)^\b e^{-\frac{\delta}{4}(|v|^2 + |v^*|^2)} |v - v^*|^\gamma e^{-\left( \frac{1}{2} - \frac{\delta}{4} \right)(|v|^2 + |v^*|^2)}\\ 
\leq& C |v - v^*|^\gamma E_{\delta/2}(v, v^*).
\end{align*}
Moreover, we have
\[
|k_1(v, v^*)| \leq \frac{C_{\b, \gamma, \delta}}{(1 + |v| + |v^*|)^\b} E_\delta(v, v^*)
\]
for $0 \leq \gamma \leq 1$, while
\[
|k_1(v, v^*)| \leq \frac{C_{\b, \delta}|v - v^*|^\gamma}{(1 + |v| + |v^*|)^\b} E_\delta(v, v^*)
\]
for $-3 < \gamma < 0$.

We give an estimate for $k_2$ following \cite{Caf 1}. We decompose the integral in the definition of $k_2$ into two parts;
\[
k_2(v, v^*) = \frac{1}{\pi^{\frac{3}{2}}} \frac{1}{|v - v^*|} E_0(v, v^*) \left( k_{2, o}(v, v^*) + k_{2, i}(v, v^*) \right),
\]
where
\begin{align*}
k_{2, o}(v, v^*) :=& \int_{\{w \in W_{v - v^*} \mid |w|^2 \geq \frac{1}{2}|V_2(v, v^*)|^2 \}} e^{- |w|^2} \left( |v - v^*|^2 + |w - V_2(v, v^*)|^2 \right)^{\frac{\gamma}{2}}\\
&\times \frac{b ( \cos \tilde{\theta}) |w - V_2(v, v^*)| + b ( \sin \tilde{\theta} ) |v - v^*|}{|w - V_2(v, v^*)| |v - v^*|}\,dw,\\
k_{2, i}(v, v^*) :=& \int_{\{w \in W_{v - v^*} \mid |w|^2 < \frac{1}{2}|V_2(v, v^*)|^2 \}} e^{- |w|^2} \left( |v - v^*|^2 + |w - V_2(v, v^*)|^2 \right)^{\frac{\gamma}{2}}\\
&\times \frac{b ( \cos \tilde{\theta} ) |w - V_2(v, v^*)| + b ( \sin \tilde{\theta} ) |v - v^*|}{|w - V_2(v, v^*)| |v - v^*|}\,dw,
\end{align*}
\[
\cos \tilde{\theta} := \frac{|v - v^*|}{(|v - v^*|^2 + |w - V_2(v, v^*)|^2)^{\frac{1}{2}}}, \quad \sin \tilde{\theta} := \frac{|w - V_2(v, v^*)|}{(|v - v^*|^2 + |w - V_2(v, v^*)|^2)^{\frac{1}{2}}}.
\]

Thanks to the assumption \eqref{assumption_B1}, we have
\[
\frac{b ( \cos \tilde{\theta} ) |w - V_2(v, v^*)| + b ( \sin \tilde{\theta} ) |v - v^*|}{|w - V_2(v, v^*)| |v - v^*|} \leq C (|v - v^*|^2 + |w - V_2(v, v^*)|^2)^{-\frac{1}{2}}. 
\]
Thus, for the $k_{2, o}$ part, we may apply the estimate $e^{-|w|^2} < e^{- \frac{1}{2} |V_2(v, v^*)|^2}$ to obtain
\begin{align*}
k_{2, o}(v, v^*) \leq& e^{- \frac{\delta}{2} |V_2(v, v^*)|^2} \int_{W_{v - v^*}} e^{- (1 - \delta) |w|^2} \left( |v - v^*|^2 + |w - V_2(v, v^*)|^2 \right)^{\frac{\gamma - 1}{2}}\,dw\\
\leq& C e^{- \frac{\delta}{2} |V_2(v, v^*)|^2} \int_{W_{v - v^*}} e^{- (1 - \delta) |w|^2} \left( |v - v^*|^2 + |w|^2 \right)^{\frac{\gamma - 1}{2}}\,dw
\end{align*}
for all $0 < \delta < 1$. Here, we have used the estimate
\begin{align*}
&e^{- (1 - \delta) |w|^2} \left( |v - v^*|^2 + |w - V_2(v, v^*)|^2 \right)^{\frac{\gamma - 1}{2}}\\ 
\leq& e^{- (1 - \delta) |w|^2} \left( |v - v^*|^2 + |w|^2 \right)^{\frac{\gamma - 1}{2}}\\
&+ e^{- (1 - \delta) |w - V_2(v, v^*)|^2} \left( |v - v^*|^2 + |w - V_2(v, v^*)|^2 \right)^{\frac{\gamma - 1}{2}}
\end{align*}
and change the variable of integration. 

In what follows, we give estimates for the above integral according to the parameter $\gamma$. We identify the hyperplane $W_{v - v^*}$ with $\R^2$. For $-1 < \gamma \leq 1$, noticing that $(\gamma -1)/2 \leq 0$, we have
\[
\int_{W_{v - v^*}} e^{- (1 - \delta) |w|^2} \left( |v - v^*|^2 + |w|^2 \right)^{\frac{\gamma - 1}{2}}\,dw \leq \int_{\R^2} e^{- (1 - \delta) |w|^2} |w|^{\gamma - 1}\,dw \leq C_{\delta, \gamma}
\]
for all $v, v^* \in \R^3$. When $\gamma = -1$, we introduce the polar coordinates $w = r \omega$ with $r > 0$ and $\omega \in S^1$ to get
\begin{align*}
&\int_{W_{v - v^*}} e^{- (1 - \delta) |w|^2} \left( |v - v^*|^2 + |w|^2 \right)^{-1}\,dw\\ 
\leq& 2\pi \int_0^1 e^{- (1 - \delta) r^2} \left( |v - v^*|^2 + r^2 \right)^{-1} r\,dr + 2\pi \int_1^\infty e^{- (1 - \delta) r^2}\,dr\\
\leq& C \left( |\log |v - v^*|| + 1 \right). 
\end{align*}
Similarly, for $-3 < \gamma < -1$, we introduce the polar coordinates $w = |v - v^*| r \omega$ with $r > 0$ and $\omega \in S^1$ to obtain
\begin{align*}
&\int_{W_{v - v^*}} e^{- (1 - \delta) |w|^2} \left( |v - v^*|^2 + |w|^2 \right)^{\frac{\gamma - 1}{2}}\,dw\\
=& 2\pi \int_0^\infty e^{- (1 - \delta) |v - v^*|^2 r^2} |v - v^*|^{\gamma + 1} \left( 1 + r^2 \right)^{\frac{\gamma - 1}{2}} r\,dr\\
\leq& C |v - v^*|^{\gamma + 1}.
\end{align*}
Therefore, we have
\begin{align*}
&\frac{1}{\pi^{\frac{3}{2}}} \frac{1}{|v - v^*|} E_0(v, v^*) k_{2, o}(v, v^*)\\
\leq& C_\delta w_\gamma(|v - v^*|) e^{-\frac{1}{4} |v - v^*|^2} e^{-|V_1(v, v^*)|^2} e^{-\frac{\delta}{2}|V_2(v, v^*)|^2}\\
=& C_\delta w_\gamma(|v - v^*|) e^{- (1 - \frac{\delta}{2}) \left(\frac{1}{4} |v - v^*|^2 + |V_1(v, v^*)|^2 \right)} e^{-\frac{\delta}{2}(|v|^2 + |v^*|^2)}\\
\leq& \frac{C_\delta w_\gamma(|v - v^*|)}{(1 + |v| + |v^*|)^\b} E_\delta(v, v^*)
\end{align*}
for all $\b \geq 0$, where $w_\gamma$ is the function defined by \eqref{def:wg}.

For the $k_{2, i}$ part, we have 
\[
\left( |v - v^*|^2 + |w - V_2(v, v^*)|^2 \right)^{\frac{\gamma - 1}{2}} \leq C \left( |v - v^*|^2 + |V_2(v, v^*)|^2 \right)^{\frac{\gamma - 1}{2}},
\]
and hence we have
\[
k_{2, i}(v, v^*) \leq C \left( |v - v^*|^2 + |V_2(v, v^*)|^2 \right)^{\frac{\gamma - 1}{2}} \int_{\{w \in \R^2 \mid |w|^2 < \frac{1}{2}|V_2(v, v^*)|^2 \}} e^{- |w|^2}\,dw.
\]
On the one hand, we have
\begin{align*}
&\left( |v - v^*|^2 + |V_2(v, v^*)|^2 \right)^{\frac{\gamma - 1}{2}} \int_{\{w \in \R^2 \mid |w|^2 < \frac{1}{2}|V_2(v, v^*)|^2 \}} e^{- |w|^2}\,dw\\
\leq& \left( |v - v^*|^2 + |V_2(v, v^*)|^2 \right)^{\frac{\gamma - 1}{2}} \int_{\R^2} e^{- |w|^2}\,dw\\
\leq& C \left( |v - v^*|^2 + |V_2(v, v^*)|^2 \right)^{\frac{\gamma - 1}{2}}.
\end{align*}
On the other hand, we have
\begin{align*}
&\left( |v - v^*|^2 + |V_2(v, v^*)|^2 \right)^{\frac{\gamma - 1}{2}} \int_{\{w \in \R^2 \mid |w|^2 < \frac{1}{2}|V_2(v, v^*)|^2 \}} e^{- |w|^2}\,dw\\
\leq& \left( |v - v^*|^2 + |V_2(v, v^*)|^2 \right)^{\frac{\gamma - 1}{2}} |V_2(v, v^*)|^2\\
\leq& C \left( |v - v^*|^2 + |V_2(v, v^*)|^2 \right)^{\frac{\gamma + 1}{2}}.
\end{align*}
Thus, we have
\begin{align*}
k_{2, i}(v, v^*) \leq C \min \left\{ \left(|v - v^*|^2 + |V_2(v, v^*)|^2 \right)^{\frac{\gamma - 1}{2}}, \left(|v - v^*|^2 + |V_2(v, v^*)|^2 \right)^{\frac{\gamma + 1}{2}} \right\}.
\end{align*}

When $-1 \leq \gamma \leq 1$, we have
\[
\left(|v - v^*|^2 + |V_2(v, v^*)|^2 \right)^{\frac{\gamma + 1}{2}} \leq C
\] 
for $|v - v^*|^2 + |V_2(v, v^*)|^2 < 1$ and
\[
\left(|v - v^*|^2 + |V_2(v, v^*)|^2 \right)^{\frac{\gamma - 1}{2}} \leq C \left(1 + |v - v^*|^2 + |V_2(v, v^*)|^2 \right)^{\frac{\gamma - 1}{2}}
\]
for $|v - v^*|^2 + |V_2(v, v^*)|^2 \geq 1$. Thus, we have
\begin{align*}
&\frac{1}{\pi^{\frac{3}{2}}} \frac{1}{|v - v^*|} E_0(v, v^*) k_{2, i}(v, v^*)\\
\leq& \frac{C_\delta}{|v - v^*|} E_0(v, v^*) \left(1 + |v - v^*|^2 + |V_2(v, v^*)|^2 \right)^{\frac{\gamma - 1}{2}}\\
\leq& \frac{C_\delta}{|v - v^*| \left(1 + |v - v^*|^2 + |V_1(v, v^*)|^2 + |V_2(v, v^*)|^2 \right)^{\frac{1 - \gamma}{2}}} E_{\frac{\delta}{2}}(v, v^*)\\
\leq& \frac{C_\delta}{|v - v^*| \left(1 + |v| + |v^*| \right)^{1 - \gamma}} E_\delta(v, v^*).
\end{align*}
Here, we have used the following estimate:
\begin{align*}
&\frac{1}{(1 + |v - v^*|^2 + |V_2(v, v^*)|^2)^{\frac{1 - \gamma}{2}}}\\
=&\frac{1}{(1 + |v - v^*|^2 + |V_1(v, v^*)|^2 + |V_2(v, v^*)|^2)^{\frac{1 - \gamma}{2}}}\\
&\times \left( 1 + \frac{|V_1(v, v^*)|^2}{1 + |v - v^*|^2 + |V_2(v, v^*)|^2} \right)^{\frac{1-\gamma}{2}}\\
\leq& \frac{1}{(1 + |v - v^*|^2 + |V_1(v, v^*)|^2 + |V_2(v, v^*)|^2)^{\frac{1 - \gamma}{2}}} ( 1 + |V_1(v, v^*)|^2 )^{\frac{1-\gamma}{2}}\\
\leq& \frac{C_{\delta, \gamma}}{(1 + |v - v^*|^2 + |V_1(v, v^*)|^2 + |V_2(v, v^*)|^2)^{\frac{1 - \gamma}{2}}} e^{\frac{\delta}{2} |V_1(v, v^*)|^2}.
\end{align*}
We remark that the same estimate was obtained in \cite{Caf 1} for $-1 < \gamma < 1$.

When $-3 < \gamma < -1$, since $(\gamma + 1)/2 < 0$, we have
\[
\left(|v - v^*|^2 + |V_2(v, v^*)|^2 \right)^{\frac{\gamma + 1}{2}} \leq |v - v^*|^{\gamma + 1} \leq \frac{C|v -v^*|^{\gamma + 1}}{(1 + |v- v^*|^2 + |V_2(v, v^*)|^2)^\b}
\]
for $|v - v^*|^2 + |V_2(v, v^*)|^2 < 1$ and for all $\b \geq 0$, and
\begin{align*}
\left(|v - v^*|^2 + |V_2(v, v^*)|^2 \right)^{\frac{\gamma - 1}{2}} \leq C \left(1 + |v - v^*|^2 + |V_2(v, v^*)|^2 \right)^{\frac{\gamma - 1}{2}}
\end{align*}
for $|v - v^*|^2 + |V_2(v, v^*)|^2 \geq 1$. Thus, we have
\begin{align*}
&\frac{1}{\pi^{\frac{3}{2}}} \frac{1}{|v - v^*|} e^{-\frac{1}{4} |v - v^*|^2} e^{-|V_1(v, v^*)|^2} k_{2, i}(v, v^*)\\
\leq& \frac{C_\delta}{|v - v^*|} E_0(v, v^*) \left(1 + |v - v^*|^2 + |V_2(v, v^*)|^2 \right)^{\frac{\gamma - 1}{2}}\\
&+ C_\delta |v -v^*|^\gamma E_0(v, v^*) \left(1 + |v - v^*|^2 + |V_2(v, v^*)|^2 \right)^{\frac{\gamma - 1}{2}}\\
\leq& \frac{C_\delta |v - v^*|^\gamma}{\left(1 + |v - v^*|^2 + |V_1(v, v^*)|^2 + |V_2(v, v^*)|^2 \right)^{\frac{1 - \gamma}{2}}} E_{\frac{\delta}{2}}(v, v^*)\\
\leq& \frac{C_\delta |v - v^*|^\gamma}{\left(1 + |v| + |v^*| \right)^{1 - \gamma}} E_\delta(v, v^*).
\end{align*}

Therefore, we have
\begin{align*}
\frac{1}{\pi^{\frac{3}{2}}} \frac{1}{|v - v^*|} E_0(v, v^*) k_{2, i}(v, v^*) \leq \frac{C_\delta w_\gamma(|v - v^*|)}{(1 + |v| + |v^*|)^{1 - \gamma}} E_\delta(v, v^*)
\end{align*}
for all $-3 < \gamma \leq 1$, where $w_\gamma$ is the function defined by \eqref{def:wg}. The estimate \eqref{est:k} is obtained by summing up the above estimates.

We next give a proof of the estimate \eqref{est:dk}. For $k_1$, we have
\begin{align*}
\nabla_v k_1 (v, v^*) =& B_0 \gamma |v - v^*|^{\gamma - 2} (v - v^*) e^{-\frac{1}{2}(|v|^2 + |v^*|^2)}\\ 
&+ B_0 |v - v^*|^\gamma (- v) e^{-\frac{1}{2}(|v|^2 + |v^*|^2)}.
\end{align*}
Thus, we have
\begin{align*}
|\nabla_v k_1 (v, v^*)| \leq& C \gamma |v - v^*|^{\gamma - 1} e^{-\frac{1}{2}(|v|^2 + |v^*|^2)} + C |v - v^*|^\gamma e^{- (\frac{1}{2} - \frac{\delta}{4})(|v|^2 + |v^*|^2)}\\
\leq& \frac{C_{\beta, \delta}|v - v^*|^{\gamma - 1}}{(1 + |v| + |v^*|)^\beta} E_\delta (v, v^*)
\end{align*}
for all $\beta \geq 0$. 

We change the formula of $k_2$. Due to the identity \eqref{identity_V2}, we have
\begin{align*}
k_2(v, v^*) =& \frac{1}{\pi^{\frac{3}{2}}} \frac{1}{|v - v^*|} e^{-\frac{1}{2} (|v|^2 + |v^*|^2)} \int_{W_{v - v^*}} e^{- |w|^2 - 2 w \cdot V_2(v, v^*)}\\
&\times \frac{\left( |v - v^*|^2 + |w|^2 \right)^{\frac{\gamma}{2}}}{|w| |v - v^*|} \left( b ( \cos \theta ) |w| + b ( \sin \theta ) |v - v^*| \right) \,dw
\end{align*}
with
\[
\cos \theta = \frac{|v - v^*|}{(|v - v^*|^2 + |w|^2)^{\frac{1}{2}}}, \quad \sin \theta = \frac{|w|}{(|v - v^*|^2 + |w|^2)^{\frac{1}{2}}}.
\]
We introduce the polar coordinates to the plane $W_{v -v^*}$ to obtain
\begin{align*}
k_2(v, v^*) =& \frac{1}{\pi^{\frac{3}{2}}} \frac{1}{|v - v^*|^2} e^{-\frac{1}{2} (|v|^2 + |v^*|^2)} \int_0^\infty \int_{-\pi}^{\pi} e^{- r^2 - 2 r |V_2(v, v^*)| \cos \varphi}\\ 
&\times \left( |v - v^*|^2 + r^2 \right)^{\frac{\gamma}{2}} \left( b ( \cos \theta ) r + b ( \sin \theta ) |v - v^*| \right) \,d\varphi dr.
\end{align*}
Thus, the formal differentiation yields
\[
\nabla_v k_2(v, v^*) = \sum_{j = 1}^6 \nabla_v k_{2, j}(v, v^*),
\]
where
\begin{align*}
\nabla_v k_{2, 1}(v, v^*) :=& - \frac{2(v - v^*)}{|v - v^*|^2} k_2(v, v^*),\\
\nabla_v k_{2, 2}(v, v^*) :=& -v k_2(v, v^*),\\
\nabla_v k_{2, 3}(v, v^*) :=& - \frac{2 \nabla_v |V_2(v, v^*)|}{\pi^{\frac{3}{2}}} \frac{1}{|v - v^*|^2} e^{-\frac{1}{2} (|v|^2 + |v^*|^2)} \int_0^\infty \int_{-\pi}^{\pi} (r \cos \varphi)\\ 
&\times e^{- r^2 - 2 r |V_2(v, v^*)| \cos \varphi} \left( |v - v^*|^2 + r^2 \right)^{\frac{\gamma}{2}}\\
&\times \left( b ( \cos \theta ) r + b ( \sin \theta ) |v - v^*| \right) \,d\varphi dr,\\
\nabla_v k_{2, 4}(v, v^*) :=& \frac{\gamma}{\pi^{\frac{3}{2}}} \frac{v -v^*}{|v - v^*|^2} e^{-\frac{1}{2} (|v|^2 + |v^*|^2)} \int_0^\infty \int_{-\pi}^{\pi} e^{- r^2 - 2 r |V_2(v, v^*)| \cos \varphi}\\
&\times \left( |v - v^*|^2 + r^2 \right)^{\frac{\gamma - 2}{2}} \left( b ( \cos \theta ) r + b ( \sin \theta ) |v - v^*| \right)\, d\varphi dr,\\
\nabla_v k_{2, 5}(v, v^*) :=& \frac{1}{\pi^{\frac{3}{2}}} \frac{1}{|v - v^*|^2} e^{-\frac{1}{2} (|v|^2 + |v^*|^2)} \int_0^\infty \int_{-\pi}^{\pi} e^{- r^2 - 2 r |V_2(v, v^*)| \cos \varphi}\\
&\times \left( |v - v^*|^2 + r^2 \right)^{\frac{\gamma}{2}} \nabla_v \left( b ( \cos \theta ) r + b ( \sin \theta ) |v - v^*| \right)\, d\varphi dr.
\end{align*}

By the estimate for $k_2$, we have
\[
|\nabla_v k_{2, 1}(v, v^*)| = \frac{2 |k_2(v, v^*)|}{|v - v^*|} \leq \frac{C_{\gamma, \delta}(1 + |v|) w_\gamma(|v - v^*|)}{|v -v^*|(1 + |v| + |v^*|)^{1 - \gamma}} E_\delta (v, v^*)
\]
and
\[
|\nabla_v k_{2, 2}(v, v^*)| = |v| |k_2(v, v^*)| \leq \frac{C_{\gamma, \delta}(1 + |v|) w_\gamma(|v - v^*|)}{|v -v^*|(1 + |v| + |v^*|)^{1 - \gamma}} E_\delta (v, v^*).
\]

For $\nabla_v k_{2, 3}$ and $\nabla_v k_{2, 4}$, we give an estimate for $\nabla_v |V_2(v, v^*)|$. Recalling the definition of $V_2$ and the relation $v \cdot (v^* \times v) = v^* \cdot (v^* \times v) = 0$, we have
\[
|V_2(v, v^*)| = \frac{|v^* \times v|}{|v - v^*|} = \frac{(|v|^2 |v^*|^2 - (v \cdot v^*)^2)^{\frac{1}{2}}}{|v -v^*|}.
\]
A direct computation shows that
\[
\nabla_v |V_2(v, v^*)| = -\frac{v - v^*}{|v -v^*|^2} |V_2(v, v^*)| + \frac{|v^*|^2 v - (v \cdot v^*) v^*}{|v - v^*| |v \times v^*|}.
\]
For the second term on the right hand side, we have
\[
\left| |v^*|^2 v - (v \cdot v^*) v^* \right|^2 = |v^*|^4 |v|^2 - (v \cdot v^*)^2 |v^*|^2 = |v^*|^2 |v \times v^*|^2.  
\]
Thus, we have
\[
\left| \frac{|v^*|^2 v - (v \cdot v^*) v^*}{|v - v^*| |v \times v^*|} \right| = \frac{|v^*|}{|v - v^*|}.
\]
Using the estimate
\[
|V_2(v, v^*)| = \frac{|v^* \times v|}{|v - v^*|} = \frac{|v^* \times (v - v^*)|}{|v - v^*|} \leq |v^*|,
\]
we obtain
\[
\left| \nabla_v |V_2(v, v^*)| \right| \leq \frac{2 |v^*|}{|v - v^*|}.
\]
Since
\begin{equation} \label{est:assumption_B}
\frac{b ( \cos \theta ) |w| + b ( \sin \theta ) |v - v^*|}{|w| |v - v^*|} \leq C (|v - v^*|^2 + |w|^2)^{-\frac{1}{2}}
\end{equation}
by \eqref{assumption_B1}, we adopt the original variables of integration and replace $w$ with $-w$ in $\nabla_v k_{2, 3}$ to obtain
\begin{align*}
|\nabla_v k_{2, 3}(v, v^*)| \leq& \frac{C}{|v -v^*|} e^{-\frac{1}{4}|v - v^*|^2 -|V_1(v, v^*)|^2} |\nabla_v |V_2(v, v^*)||\\
&\times \int_{W_{v -v^*}} |w| e^{-|w - V_2(v, v^*)|^2} (|v -v^*|^2 + |w|^2)^{\frac{\gamma - 1}{2}}\,dw\\
\leq& \frac{C|v^*|}{|v -v^*|^2} e^{-\frac{1}{4}|v - v^*|^2 -|V_1(v, v^*)|^2} \\
&\times \int_{W_{v -v^*}} e^{-|w - V_2(v, v^*)|^2} (|v -v^*|^2 + |w|^2)^{\frac{\gamma + 1}{2}}\,dw\\
\leq& \frac{C(|v - v^*| + |v|)}{|v -v^*|^2} e^{-\frac{1}{4}|v - v^*|^2 -|V_1(v, v^*)|} w_{\gamma + 1}(|v - v^*|).
\end{align*}
By the direct calculation, we see that
\[
w_{\gamma + 1}(|v - v^*|) e^{-\delta|v - v^*|^2} \leq C_\delta w_\gamma(|v -v^*|).
\]
Thus, by taking $\delta$ suitably, we have
\[
|\nabla_v k_{2, 3}(v, v^*)| \leq \frac{C(1 + |v|) w_\gamma(|v - v^*|)}{|v -v^*|(1 + |v| + |v^*|)^{1 - \gamma}} E_\delta(v, v^*).
\]

In the same way, by \eqref{est:assumption_B}, we obtain
\begin{align*}
|\nabla_v k_{2, 4}(v, v^*)| \leq& C e^{-\frac{1}{2} (|v|^2 + |v^*|^2)}\\
&\times \int_0^\infty \int_{-\pi}^{\pi} e^{- r^2 - 2 r |V_2(v, v^*)| \cos \varphi} \left( |v - v^*|^2 + r^2 \right)^{\frac{\gamma - 3}{2}}\,r d\varphi dr\\
\leq& \frac{C}{|v -v^*|} e^{-\frac{1}{4} |v -v^*|^2} e^{-|V_1(v, v^*)|^2}\\
&\times \int_{W_{v -v^*}} e^{-|w - V_2(v, v^*)|^2} (|v -v^*|^2 + |w|)^{\frac{\gamma - 1}{2}}\,dw\\
\leq& \frac{C w_\gamma(|v -v^*|)}{(1 + |v| + |v^*|)^{1 - \gamma}} E_{\delta/2}(v, v^*)\\
\leq& \frac{C (1+|v|) w_\gamma(|v -v^*|)}{|v -v^*|(1 + |v| + |v^*|)^{1 - \gamma}} E_{\delta}(v, v^*).
\end{align*}

For $\nabla_v k_{2, 5}$, we assume that the function $b$ is smooth. We notice that
\begin{align*}
\nabla_v \left( \frac{|v - v^*|}{(|v - v^*|^2 + r^2)^{\frac{1}{2}}} \right) =& - \frac{r(v - v^*)}{|v - v^*|^2} \partial_r \left( \frac{|v - v^*|}{(|v - v^*|^2 + r^2)^{\frac{1}{2}}} \right),\\
\nabla_v \left( \frac{r}{(|v - v^*|^2 + r^2)^{\frac{1}{2}}} \right) =& - \frac{r(v - v^*)}{|v - v^*|^2} \partial_r \left( \frac{r}{(|v - v^*|^2 + r^2)^{\frac{1}{2}}} \right).
\end{align*}
Thus, the integration by parts yields
\begin{align*}
\nabla_v k_{2, 5}(v, v^*) =& - \frac{1}{\pi^{\frac{3}{2}}} \frac{v - v^*}{|v - v^*|^4} e^{-\frac{1}{2} (|v|^2 + |v^*|^2)} \int_0^\infty \int_{-\pi}^{\pi} e^{- r^2 - 2 r |V_2(v, v^*)| \cos \varphi}\\
&\times \left( |v - v^*|^2 + r^2 \right)^{\frac{\gamma}{2}} \partial_r \left( b ( \cos \theta ) r + b ( \sin \theta ) |v - v^*| \right) r\, d\varphi dr\\
=& \sum_{j = 1}^3 \nabla_v k_{2, 5, j}(v, v^*),
\end{align*}
where
\begin{align*}
\nabla_v k_{2, 5, 1}(v, v^*) :=& - \frac{2}{\pi^{\frac{3}{2}}} \frac{v - v^*}{|v - v^*|^4} e^{-\frac{1}{2} (|v|^2 + |v^*|^2)}\\
&\times \int_0^\infty \int_{-\pi}^{\pi} (r + |V_2(v, v^*)| \cos \varphi) e^{- r^2 - 2 r |V_2(v, v^*)| \cos \varphi}\\
&\times \left( |v - v^*|^2 + r^2 \right)^{\frac{\gamma}{2}} \left( b ( \cos \theta ) r + b ( \sin \theta ) |v - v^*| \right) r\, d\varphi dr,\\
\nabla_v k_{2, 5, 2}(v, v^*) :=& \frac{\gamma}{\pi^{\frac{3}{2}}} \frac{v - v^*}{|v - v^*|^4} e^{-\frac{1}{2} (|v|^2 + |v^*|^2)} \int_0^\infty \int_{-\pi}^{\pi} e^{- r^2 - 2 r |V_2(v, v^*)| \cos \varphi}\\
&\times \left( |v - v^*|^2 + r^2 \right)^{\frac{\gamma - 2}{2}} \left( b ( \cos \theta ) r + b ( \sin \theta ) |v - v^*| \right) r^2\, d\varphi dr,\\
\nabla_v k_{2, 5, 3}(v, v^*) :=& \frac{1}{\pi^{\frac{3}{2}}} \frac{v - v^*}{|v - v^*|^4} e^{-\frac{1}{2} (|v|^2 + |v^*|^2)} \int_0^\infty \int_{-\pi}^{\pi} e^{- r^2 - 2 r |V_2(v, v^*)| \cos \varphi}\\
&\times \left( |v - v^*|^2 + r^2 \right)^{\frac{\gamma}{2}} \left( b ( \cos \theta ) r + b ( \sin \theta ) |v - v^*| \right)\, d\varphi dr.
\end{align*}
For the $\nabla_v k_{2, 5, 1}$ term, since
\[
|V_2(v, v^*)| = \frac{|(v^* - v) \times v|}{|v - v^*|} \leq |v|,
\]
we have
\begin{align*}
&|\nabla_v k_{2, 5, 1}(v, v^*)|\\ 
\leq& \frac{C}{|v - v^*|^3} e^{-\frac{1}{4}|v - v^*|^2 -|V_1(v, v^*)|^2} \int_{W_{v - v^*}} (|w| + |v|) e^{-|w + V_2(v, v^*)|^2}\\ 
&\times \left( |v - v^*|^2 + |w|^2 \right)^{\frac{\gamma}{2}} \left( b ( \cos \theta ) |w| + b ( \sin \theta ) |v - v^*| \right) \, dw\\
\leq& \frac{C}{|v - v^*|^2} e^{-\frac{1}{4}|v - v^*|^2 -|V_1(v, v^*)|^2} \int_{W_{v - v^*}} e^{-|w - V_2(v, v^*)|^2} \left( |v - v^*|^2 + |w|^2 \right)^{\frac{\gamma + 1}{2}}\,dw\\
&+ \frac{C |v|}{|v - v^*|^2} e^{-\frac{1}{4}|v - v^*|^2 -|V_1(v, v^*)|^2} \int_{W_{v - v^*}} e^{-|w - V_2(v, v^*)|^2} \left( |v - v^*|^2 + |w|^2 \right)^{\frac{\gamma}{2}}\,dw\\
\leq& \frac{C (1 + |v|)}{|v - v^*|} e^{-\frac{1}{4}|v - v^*|^2 -|V_1(v, v^*)|^2} w_{\gamma + 1}(|v - v^*|)\\
\leq& \frac{C (1 + |v|) w_\gamma(|v - v^*|)}{|v - v^*|(1 + |v| + |v^*|)^{1 - \gamma}} E_\delta(v, v^*).
\end{align*}
For the $\nabla_v k_{2, 5, 2}$ term, we have
\begin{align*}
&|\nabla_v k_{2, 5, 2}(v, v^*)|\\ 
\leq& \frac{C}{|v - v^*|^3} e^{-\frac{1}{4}|v - v^*|^2 -|V_1(v, v^*)|^2} \int_{W_{v - v^*}} e^{-|w + V_2(v, v^*)|^2}\\
&\times \left( |v - v^*|^2 + |w|^2 \right)^{\frac{\gamma - 2}{2}} \left( b ( \cos \theta ) |w| + b ( \sin \theta ) |v - v^*| \right) |w|\, dw\\
\leq& \frac{C}{|v - v^*|^2} e^{-\frac{1}{4}|v - v^*|^2 -|V_1(v, v^*)|^2} \int_{W_{v - v^*}} e^{-|w - V_2(v, v^*)|^2} \left( |v - v^*|^2 + |w|^2 \right)^{\frac{\gamma - 1}{2}} \, dw\\
\leq& \frac{C (1 + |v|) w_\gamma(|v - v^*|)}{|v - v^*|(1 + |v| + |v^*|)^{1 - \gamma}} E_\delta(v, v^*).
\end{align*}
For the $\nabla_v k_{2, 5, 3}$ term, by \eqref{est:assumption_B}, we have
\begin{align*}
&|\nabla_v k_{2, 5, 3}(v, v^*)|\\ 
\leq& \frac{C}{|v - v^*|^3} e^{-\frac{1}{4}|v - v^*|^2 -|V_1(v, v^*)|^2} \int_{W_{v - v^*}} e^{-|w + V_2(v, v^*)|^2}\\
&\times \frac{\left( |v - v^*|^2 + r^2 \right)^{\frac{\gamma}{2}}}{|w|} \left( b ( \cos \theta ) |w| + b ( \sin \theta ) |v - v^*| \right)\, dw\\
\leq& \frac{C}{|v - v^*|^2} e^{-\frac{1}{4}|v - v^*|^2 -|V_1(v, v^*)|^2} \int_{W_{v - v^*}} e^{-|w - V_2(v, v^*)|^2} \left( |v - v^*|^2 + |w|^2 \right)^{\frac{\gamma - 1}{2}}\,dw\\
\leq& \frac{C (1 + |v|) w_\gamma(|v - v^*|)}{|v - v^*|(1 + |v| + |v^*|)^{1 - \gamma}} E_\delta(v, v^*).
\end{align*}
Summarizing the above estimates, we conclude that
\[
|\nabla_v k_{2, 5}(v, v^*)| \leq \frac{C (1 + |v|) w_\gamma(|v - v^*|)}{|v - v^*|(1 + |v| + |v^*|)^{1 - \gamma}} E_\delta(v, v^*).
\]
Since the constant in the estimate does not depend on $b'$, we may apply the approximation argument to obtain the same estimate for a nonsmooth function $b$.

Therefore, the estimate \eqref{est:dk} is proved.



\begin{thebibliography}{99}

\bibitem{Ag84} Agoshkov, Valery. I.: Spaces of functions with differential-difference characteristics and the smoothness of solutions of the transport equation, \textit{Dokl. Akad. Nauk SSSR}, 276 (1984), no. 6, 1289--1293.

\bibitem{Caf 1} Caflisch, Russel E.: The Boltzmann equation with a soft potential. I. Linear, spatially-homogeneous. \textit{Comm. Math. Phys.} 74 (1980), no. 1, 71--95.
  
\bibitem{CIP} Cercignani, Carlo; Illner, Reinhard; Pulvirenti, Mario: \textit{The mathematical theory of dilute gases}, Applied Mathematical Sciences 106, Springer, New York (1994).

\bibitem{HChen} Chen, Hongxu: On regularity of a kinetic boundary layer. \textit{Nonlinear Anal.} 261 (2025), Paper No. 113891, 33 pp.

\bibitem{ChenKim} Chen, Hongxu; Kim, Chanwoo: Regularity of stationary Boltzmann equation in convex domains. \textit{Arch. Ration. Mech. Anal.} 244 (2022), no. 3, 1099--1222.

\bibitem{ChenKimGra}Chen, Hongxu; Kim, Chanwoo: Gradient decay in the Boltzmann theory of non-isothermal boundary. \textit{Arch. Ration. Mech. Anal. } 248 (2024), no. 2, 52pp. 

\bibitem{RegularChen} Chen, I-Kun: Regularity of stationary solutions to the linearized Boltzmann equations, \textit{SIAM J. Math. Anal.} 50 (2018), no. 1, 138--161. 

\bibitem{CCHS} Chen, I-Kun; Chuang, Ping-Han; Hsia, Chun-Hsiung; Su; Jhe-Kuan: A revisit of the velocity averaging lemma: On the regularity of stationary Boltzmann equation in a boudned convex domain. \textit{J. Stat. Phys.} 189 (2022), no. 2, Paper No. 17, 43pp.

\bibitem{CHK} Chen, I-Kun; Hsia, Chun-Hsiung; Kawagoe, Daisuke: Regularity for diffuse reflection boundary problem to the stationary linearized Boltzmann equation in a convex domain, \textit{Ann. I.~H.~Poincar\'e} 36 (2019) 745--782.

\bibitem{CHKFred} Chen, I-Kun; Hsia, Chun-Hsiung; Kawagoe, Daisuke: On the existence and regularity of weakly nonlinear stationary Boltzmann equations: Fredholm alternative approach, preprint, arXiv:2501.02419.

\bibitem{CHKS1} Chen, I-Kun; Hsia, Chun-Hsiung; Kawagoe, Daisuke; Su, Jhe-Kuan: Geometric effects on $W^{1, p}$ regularity of the stationary linearized Boltzmann equation, \textit{Indiana Univ. Math. J.} 74 (2025), no. 6, 1749–1791.

\bibitem{CHKS2} Chen, I-Kun; Hsia, Chun-Hsiung; Kawagoe, Daisuke; Su, Jhe-Kuan: On the existence and regularity for stationary Boltzmann equation in a small domain, \textit{SIAM J. Math. Anal.}, 57 (2025), no. 1, 4191--4219. 

\bibitem{CS} Choulli, Mourad; Stefanov, Plamen: An inverse boundary value problem for the stationary transport equation. \textit{Osaka J. Math.} 36 (1999), no.1, 87--104.

\bibitem{DPV} Di Nezza, Eleonora; Palatucci, Giampiero; Valdinoci, Enrico: Hitchhiker's guide to the fractional Sobolev spaces. \textit{Bull. Sci. Math.} 136 (2012), no. 5, 521--573.

\bibitem{DHWY2017} Duan, Renjun; Huang, Feimin; Wang, Yong; Yang, Tong: Global well-posedness of the Boltzmann equation with large amplitude initial data, \textit{Arch. Ration. Mech. Anal.} 225 (2017), 375--424.

\bibitem{DHWZ2019} Duan, Renjun; Huang, Feimin; Wang, Yong; Zhang, Zhu: Effects of soft interaction and non-isothermal boundary upon long-time dynamics of rarefied gas, \textit{Arch. Ration. Mech. Anal.} 234 (2019), 925--1006.

\bibitem{GuoKim}Esposito, Raffaele; Guo, Yan; Kim, Chanwoo; Marra, Rossana: Non-isothermal boundary in the Boltzmann theory and Fourier law. \textit{Comm. Math. Phys.} 323 (2013), no. 1, 177--239. 

\bibitem{Glassey} Glassey, Robert, T.: \textit{The Cauchy problem in kinetic theory}, SIAM, USA (1996).

\bibitem{GLPS} Golse, Francois; Lions, Pierre-Louis; Perthame, Benoit; Sentis, R\'emi: Regularity of the moments of the solution of a transport equation, \textit{J. Func. Anal.} 76 (1988), 110--125.

\bibitem{GPS} Golse, Francois; Perthame, Benoit; Sentis, R\'emi: Un r\'esultat de compacit\'e pour les equations de transport et application au caicui de la valeur propre principale d'un op\'erateur de transport, \textit{C. R. Acad. Sci. Paris} 301 (1985), 341--344.

\bibitem{Grad} Grad, Harold: Asymptotic Theory of the Boltzmann Equation, II \textit{1963 Rarefied Gas Dynamics }(Proc. 3rd Internat. Sympos., Palais de l'UNESCO, Paris, 1962), Vol. I pp. 26-59 Academic Press, New York.

\bibitem{2003Guo} Guo, Yan: Classical solutions to the Boltzmann equation for molecules with an angular cutoff, \textit{Arch. Ration. Mech. Anal.} 169 (2003), 305--353.

\bibitem{Guo} Guo, Yan: Decay and continuity of the Boltzmann equation in bounded domains, \textit{Arch. Ration. Mech. Anal.} 197 (2010), 713--809.
 
\bibitem{LWW} Lin, Yu-Chu; Wang, Haitao; Wu, Kung-Chien: Mixture estimate in fractional sense and its application to the well-posedness of the Boltzmann equation with very soft potential, \textit{Math. Ann.} 387 (2023), 2061--2103.

\bibitem{M} Maslova, Nina B.: \textit{Nonlinear evolution equations}, Series on advances in Mathematics for applied sciences Vol. 10, World Scientific, Singapore (1993). 

\bibitem{WW} Wu, Kung-Chien; Wang, Kuan-Hsiang: H\"older regularity of solutions of the steady Boltzmann equation with soft potentials, \textit{J. Func. Anal.}, 289 (2025), no. 11, 69pp.
\end{thebibliography}
\end{document}